\documentclass[a4paper,11pt]{article}
\usepackage[top=2.5cm, bottom=2.5cm, left=2.5cm, right=2.5cm]{geometry}

\usepackage[english]{babel}

\usepackage{amsmath}
\usepackage{amsfonts,amssymb}
\usepackage{amsthm}
\usepackage{dsfont}
\usepackage{xcolor}
\usepackage{graphicx}
\usepackage{subcaption}
\usepackage{upgreek}

\usepackage{algorithm}
\usepackage{algpseudocode}

\makeatletter
\def\theHALG@line{\thealgorithm-\arabic{ALG@line}}
\makeatother

\usepackage{mathrsfs}
\usepackage{stmaryrd}

\usepackage{array}
\usepackage{booktabs}
\usepackage{makecell}
\usepackage{multirow}
\usepackage[colorlinks=true, allcolors=blue]{hyperref}
\usepackage{cleveref}
\usepackage{authblk}
\DeclareMathOperator*{\bigtimes}{\vartimes}
\usepackage{bm}

\newcommand{\MSC}[1]{\noindent\textit{MSC classes:} #1}
\newcommand{\keywords}[1]{\noindent\textit{Keywords:} #1}

\newcommand{\cI}{\mathcal{I}}

\newcommand{\ba}{\mathbf{a}}
\newcommand{\bb}{\mathbf{b}}
\newcommand{\be}{\mathbf{e}}
\newcommand{\bg}{\mathbf{g}}

\newcommand{\bs}{\mathbf{s}}
\newcommand{\bu}{\mathbf{u}}
\newcommand{\bv}{\mathbf{v}}
\newcommand{\bw}{\mathbf{w}}
\newcommand{\bx}{\mathbf{x}}
\newcommand{\by}{\mathbf{y}}
\newcommand{\bA}{\mathbf{A}}
\newcommand{\bB}{\mathbf{B}}
\newcommand{\bC}{\mathbf{C}}
\newcommand{\bD}{\mathbf{D}}
\newcommand{\bE}{\mathbf{E}}
\newcommand{\bI}{\mathbf{I}}
\newcommand{\bG}{\mathbf{G}}
\newcommand{\bH}{\mathbf{H}}

\newcommand{\bM}{\mathbf{M}}
\newcommand{\bP}{\mathbf{P}}
\newcommand{\bQ}{\mathbf{Q}}

\newcommand{\bS}{\mathbf{S}}

\newcommand{\bU}{\mathbf{U}}
\newcommand{\bV}{\mathbf{V}}
\newcommand{\bW}{\mathbf{W}}
\newcommand{\bX}{\mathbf{X}}
\newcommand{\bY}{\mathbf{Y}}
\newcommand{\bZ}{\mathbf{Z}}
\newcommand{\bnu}{{\boldsymbol \nu}}
\newcommand{\balpha}{{\boldsymbol \alpha}}
\newcommand{\bPhi}{{\boldsymbol \Phi}}
\newcommand{\bSigma}{{\boldsymbol \Sigma}}
\newcommand{\bOmega}{{\boldsymbol \Omega}}

\newcommand{\cA}{\mathcal{A}}
\newcommand{\cB}{\mathcal{B}}
\newcommand{\cC}{\mathcal{C}}

\newcommand{\cG}{\mathcal{G}}
\newcommand{\cH}{\mathcal{H}}

\newcommand{\cN}{\mathcal{N}}
\newcommand{\cP}{\mathcal{P}}
\newcommand{\cO}{\mathcal{O}}
\newcommand{\cS}{\mathcal{S}}
\newcommand{\cU}{\mathcal{U}}
\newcommand{\cV}{\mathcal{V}}
\newcommand{\cX}{\mathcal{X}}
\newcommand{\cY}{\mathcal{Y}}

\newcommand{\bbC}{\mathbb{C}}
\newcommand{\bbR}{\mathbb{R}}
\newcommand{\bbF}{\mathbb{F}}
\newcommand{\mc}[1]{\mathcal{{#1}}}
\newcommand{\tr}[2]{\mathrm{tr}_{#1}\left ({#2}\right)}
\newcommand{\Tr}[1]{\mathrm{Tr}\left ({#1}\right)}
\newcommand{\E}[1]{\mathbb{E}\left[{#1}\right]}
\renewcommand{\P}[1]{\mathbb{P}\left\{{#1}\right\}}
\newcommand{\Var}[1]{\mathrm{Var}\left[{#1}\right]}
\newcommand{\Mom}[2]{\mathrm{Mom}\left[{#1} \right]\left ({#2}\right)}
\renewcommand{\Re}[1]{\mathrm{Re}\left[{#1}\right]}

\newcommand{\change}[1]{{\color{blue}#1}}
\newcommand{\commentOut}[1]{\iffalse {#1} \fi}

\newcommand{\name}{TTStack}
\newcommand{\nameAbbrv}{TTS}

\title{Linear-scaling Tensor Train Sketching}

\author[1]{Paul Cazeaux\thanks{Corresponding author. Email: \texttt{cazeaux@vt.edu}}}
\author[2]{Mi-Song Dupuy}
\author[1]{Rodrigo Figueroa Justiniano}
\affil[1]{Department of Mathematics, Virginia Tech, VA, USA}
\affil[2]{Laboratoire Jacques-Louis Lions, Sorbonne Universit\'e, Paris, France}

\theoremstyle{plain}
\newtheorem{theorem}{Theorem}[section]
\newtheorem{lemma}[theorem]{Lemma}
\newtheorem{corollary}[theorem]{Corollary}
\newtheorem{proposition}[theorem]{Proposition}
\newtheorem{importedlemma}[theorem]{Imported Lemma}

\newtheorem{definition}[theorem]{Definition}

\newtheorem{remark}[theorem]{Remark}

\Crefname{importedlemma}{Imported Lemma}{Imported Lemmas}
\Crefname{importedtheorem}{Imported Theorem}{Imported Theorems}
\Crefname{remark}{Remark}{Remarks}
\crefname{lemma}{lemma}{lemmas}
\Crefname{lemma}{Lemma}{Lemmas}
\Crefname{proposition}{Proposition}{Propositions}

\begin{document}
\maketitle

\begin{abstract}
We introduce the {\name} sketch, a structured random projection tailored to the tensor train (TT) format that unifies existing TT-adapted sketching operators. By varying two integer parameters $P$ and $R$, {\name} interpolates between the Khatri-Rao sketch ($R=1$) and the Gaussian TT sketch ($P=1$). We prove that {\name} satisfies an oblivious subspace embedding (OSE) property with parameters $R = \mathcal{O}(d(r+\log 1/\delta))$ and $P = \mathcal{O}(\varepsilon^{-2})$, and an oblivious subspace injection (OSI) property under the condition $R = \mathcal{O}(d)$ and $P = \mathcal{O}(\varepsilon^{-2}(r + \log r/\delta))$. Both guarantees depend only linearly on the tensor order $d$ and on the subspace dimension $r$, in contrast to prior constructions that suffer from exponential scaling in $d$. As direct consequences, we derive quasi-optimal error bounds for the QB factorization and randomized TT rounding. The theoretical results are supported by numerical experiments on synthetic tensors, Hadamard products, and a quantum chemistry application.

\medskip
\MSC{15A69 (primary); 65F55, 65F99, 65Y20, 68W20 (secondary)}

\smallskip
\end{abstract}

\footnotetext{\textbf{Acknowledgements:} P.C and M-S. D. acknowledges support from the European Research Council (ERC) under the European Union’s Horizon 2020 Research and Innovation Programme – Grant Agreement 101077204. P.C. acknowledges support from the Simons Foundation Travel Support for Mathematicians award MPS-TSM-00966604. M-S. D. acknowledges support from the Sorbonne Université Emergence 2023-2025 - BAACES program. This material is based in part upon work supported by the National Science Foundation under Grant No. DMS-1929284 while authors P.C. and R.F.J were in residence at the Institute for Computational and Experimental Research in Mathematics in Providence, RI, during the "Stochastic and Randomized Algorithms in Scientific Computing: Foundations and Applications" program.}

\section{Introduction} \label{sec:introduction}

In high-dimensional problems, tensor decompositions provide an efficient means to reduce the complexity of the objects being handled. They are essential for tasks involving large datasets, or handling many-variable functions where traditional methods struggle with scalability and computational feasibility. Tensor trains (TT)~\cite{Oseledets_Tyrtyshnikov_2009} are an example of one such decompositions, that can be viewed as a structured low-rank factorization. Introduced first as Matrix Product States (MPS) in solid-state physics~\cite{White_1992} and quantum chemistry~\cite{Szalay_Pfeffer_Murg_Barcza_Verstraete_Schneider_Legeza_2015} to compute ground-states and low-lying excited states of many-body quantum systems, it has now become a standard tool for state-of-the-art numerical simulations in many fields.

Lately, tensor trains have excelled at solving {low}-dimensional partial differential equations in problems of homogenization~\cite{Kazeev_Oseledets_Rakhuba_Schwab_2017,Kazeev_Oseledets_Rakhuba_Schwab_2022}, plasma physics~\cite{Ye_Loureiro_2024}, and fluid dynamics~\cite{Gourianov_Lubasch_Dolgov_Berg_Babaee_Givi_Kiffner_Jaksch_2022}, through the use of \emph{quantized tensor trains}~\cite{Oseledets_2010}, where space variables are transformed into labels of high-order tensors by a dyadic decomposition.

A core computational advantage of the TT format is its closure under standard algebraic operations like linear combinations, matrix-vector products, and elementwise powers. However, executing these operations naturally induces an increase of TT-ranks, which for amenable computations can be reduced in a monitored fashion using a compression scheme known as the \emph{TT rounding algorithm}~\cite{Oseledets_Tyrtyshnikov_2009}, but for high-order tensors with substantial ranks, this signifies a major computational bottleneck. 

Randomization techniques using structured sketches and suitable modifications of the deterministic algorithms
have been proposed to significantly accelerate the compression without sacrificing its accuracy~\cite{aldaas2023randomized,kressner2023streaming,rakhshan2020tensorized,Che_Wei_Yan_2026}. Nonetheless, there remains a gap between the limited theoretical understanding and the empirical efficiency of randomized TT-rounding algorithms.

In this work, we introduce a family of tensor train sketches, unifying the previously known classes into one that we call the {\name} (\nameAbbrv) sketch, and for which we provide new theoretical guarantees such as an oblivious subspace embedding with parameters that depend polynomially on the dimension or the ranks, and an oblivious subspace \emph{injection} with less stringent conditions on the parameters than its embedding counterpart. Furthermore, we apply these results to show that the randomized TT rounding algorithm and low-rank approximation output a quasi-optimal tensor train.

The remainder of this work is organized as follows. \Cref{sec:preliminaries} outlines the notation used throughout the paper and presents the probabilistic setting used in our main results. \Cref{sec:contributions} presents our core contributions, introducing the new {\name} sketching framework and its theoretical guarantees, alongside its practical implications for low-rank approximation and randomized tensor train rounding algorithms. \Cref{sec:review} provides a comparative self-contained review of existing sketching operators for the tensor train format, highlighting their respective sample complexities and inherent limitations. \Cref{sec:applications} details the performed numerical experiments where the {\name} sketch is applied to different scenarios that require randomized rounding, thus providing empirical evidence that small parameters are sufficient for practical applications. \Cref{sec:OSE} outlines the structural properties necessary to prove that {\name} acts as an oblivious subspace embedding. Subsequently, \Cref{sec:OSI} develops the novel analytical machinery required to establish that {\name} satisfies the oblivious subspace injection property, and concludes with the formal error analysis for randomized TT rounding and low-rank approximation. 

\section{Preliminaries}
\label{sec:preliminaries}

\subsection*{Notation}

Throughout this work we will denote an abstract scalar field by $\bbF$ and explicitly let $\bbF = \bbR$ or $\bbF=\bbC$ whenever necessary. The operator $\mathbb{P}$ returns the probability of an event, while $\mathbb{E}$ and $\mathrm{Var}$ compute the expectation and variance of a random variable. The symbol $\sim$ reads 'has the distribution', and as usual iid abbreviates 'independent and identically distributed'. The real normal distribution $\mathcal{N}_\bbR(0, \sigma^2)$ describes a real Gaussian random variable with expectation $0$ and variance $\sigma^2 > 0$. The complex normal distribution $\mathcal{N}_\bbC(0, \sigma^2)$ induces a complex random variable of the form $(Z_1 + \mathrm{i} Z_2)/\sqrt{2}$ where $Z_1, Z_2 \sim \mathcal{N}_\bbR(0, \sigma^2)$ iid. 

For an integer $n \geq 0$, we denote the integer range $[n] := \{ 1,\dots, n\}$. We also denote the order
of a tensor $\cX \in \bbF^{n_1 \times n_2 \times \dots \times n_d}$ by $d$, the dimension of the $k$-th mode by $n_k$, and the ambient dimension by $N=n_1n_2 \cdots n_d$.
For order-3 tensors $\mathcal{C} \in \mathbb{F}^{m \times n \times p}$, for $j \in [n]$, $\mathcal{C}[j] \in \mathbb{F}^{m \times p}$ is the matrix equal to $(\mathcal{C}_{ijk})_{i \in [m], k \in [p]}$. 

\subsection{Tensor trains}

We give here the definition of a tensor train and the notation to manipulate tensors written in their tensor train representations.
The interested reader may refer to the monograph~\cite{Hackbusch_2012} for a more thorough exposition.

\begin{definition}[Tensor Train Format]\label{def:TT}
    Let $\mc{X} \in \bbF^{n_1 \times n_2 \times \dots \times n_d}$ be a tensor.
    We say that $(\mathcal{C}_k)_{k \in [d]} \in \bigtimes_{k \in [d]} \bbF^{r_{k-1} \times n_k \times r_k}$ with $r_0=r_d=1$ is a tensor train representation of $\mathcal{X}$ if its entries can be expressed as the sequence of matrix products:
    \begin{equation}
        \cX_{i_1,\dots,i_d} = \cC_1[i_1] \cC_2[i_2] \cdots \cC_d[i_d]
        = \sum_{\alpha_1, \dots, \alpha_{d-1}} \prod_{k=1}^d \cC_k[\alpha_{k-1}, i_k, \alpha_k],
    \end{equation}
    where for all $j \in [d]$, $i_j \in [n_j]$ and the summation indices $\alpha_j \in [r_j]$. 
    The tuple $(r_0, \dots, r_d)$ constitutes the \textit{TT-ranks} of the representation $(\mathcal{C}_k)_{k \in [d]}$. Generalized tensor trains with boundary ranks $r_0 > 1$ or $r_d > 1$ are called block tensor trains~\cite{dolgov2014computation}.
\end{definition}

To facilitate the exposition of partial contractions and sketch applications, we introduce the following reshaping and product operations.

\begin{definition}[Tensor Unfolding]
    Let $\cX \in \bbF^{n_1 \times \dots \times n_d}$ and let $\ell \in [d]$. The \textit{unfolding} of $\cX$ with respect to the first $\ell$ modes is the matrix $\cX^{\leq \ell} \in \bbF^{(n_1 \cdots n_\ell) \times (n_{\ell+1} \cdots n_d)}$ with entries
    \[
        (\cX^{\leq \ell})_{i_1 \dots i_\ell, i_{\ell+1} \dots i_d} := \cX_{i_1, \dots, i_d}.
    \]
\end{definition}

\begin{definition}[Strong Kronecker Product]
    Given tensors $\cA \in \bbF^{r_0 \times n_1 \times \cdots \times n_j \times r_j}$ and $\cB \in \bbF^{r_j \times n_{j+1} \times \cdots \times n_k \times r_k}$ with $1 \leq j \leq k \leq d$, the Strong Kronecker product $\cA \bowtie \cB \in \bbF^{r_0 \times n_1 \times \cdots \times n_k \times r_k}$ is defined entry-wise by the contraction:
    \[
        (\cA \bowtie \cB)_{\alpha_0, i_1, \dots, i_k, \alpha_k} := \sum_{\alpha_j=1}^{r_j} \cA_{\alpha_0, i_1, \dots, i_j, \alpha_j} \, \cB_{\alpha_j, i_{j+1}, \dots, i_k, \alpha_k}.
    \]
    This product is associative, and furthermore the following property on unfolding matrices is easily checked:
    \begin{equation}\label{eq:kron}
        (\cA \bowtie \cB)^{\leq 1} = \cA^{\leq 1} ( \bI_{n_1 \cdots n_j} \otimes \cB^{\leq 1} ).
    \end{equation}
\end{definition}
As a consequence, any tensor $\cX$ in tensor train format can be succinctly expressed as the Strong Kronecker product of its cores: 
$ \cX = \cC_1 \bowtie \cC_2 \bowtie \cdots \bowtie \cC_d $ (with the slight abuse of notation where we identify $\bbF^{1 \times n_1 \times \cdots \times n_d \times 1}$ with $\bbF^{n_1 \times \cdots \times n_d}$). Moreover, an immediate corollary is that the following recursive expression in terms of core unfoldings holds:
\begin{equation}\label{eq:recursiveTT}
    (\cC_1 \bowtie \cdots \bowtie \cC_d)^{\leq 1} = \cC_1^{\leq 1} \left ( \bI_{n_1} \otimes \cC_2^{\leq 1} \right ) \left ( \bI_{n_1 n_2} \otimes \cC_3^{\leq 1} \right ) \cdots \left ( \bI_{n_1 \dots n_{d-1}} \otimes \cC_d^{\leq 1} \right ).
\end{equation}

Having set the necessary notation, we now formally define the fundamental geometric properties that govern sketching operators.

\subsection{Theoretical Framework: OSE and OSI}

As the benchmark for sketching quality, the Oblivious Subspace Embedding (OSE) guarantees the preservation of inner products, distances, and singular values—a geometric property essential for analyzing optimization and approximation algorithms in randomized numerical linear algebra~\cite{sarlos2006improved, woodruff2014sketching, rnla_foundations, rnla_perspective}.

\begin{definition}[Oblivious Subspace Embedding]
    A random matrix $\bOmega \in \bbF^{m \times N}$ is said to satisfy the $(\alpha, \beta, \delta, r)$-OSE property with subspace dimension $r$, embedding dimension $m \geq r$, injectivity $ 0 <\alpha < 1$ and dilation $\beta > 1$ (or more commonly seen as the $(\epsilon, \delta, r)$-OSE property with $\alpha=1-\varepsilon$, $\beta=1+\epsilon$), if for any $r$-dimensional subspace $\mathcal{U} \subseteq \bbF^N$, the following holds with probability at least $1-\delta$:
    \[
        \forall \, \bx \in \mathcal{U}, \quad \alpha\|\bx\|_2^2 \leq \|\bOmega \bx\|_2^2 \leq \beta\|\bx\|_2^2. 
    \]
\end{definition}

Although computationally expensive, dense Gaussian matrices serve as the theoretical benchmark for subspace embeddings, as their isotropic structure and well-understood concentration properties yield an optimal sample complexity.

\begin{definition}
    A matrix $\bG \in \bbF^{m \times n}$ is a Gaussian sketch if every entry $\bG_{i,j}$ is an iid Gaussian random variable with mean zero and variance $1/m$.
\end{definition}

Gaussian sketches not only satisfy the OSE property provided the embedding dimension scales as $m = \Theta\left(\varepsilon^{-2}(r+\ln(1/\delta)\right))$~\cite{woodruff2014sketching}, but this sample complexity is also optimal, matching the fundamental lower bound required for any linear map to preserve Euclidean geometry~\cite{optimal_embedding}. Furthermore, in the context of the randomized rangefinder primitive, applying a Gaussian OSE 
yields a quasi-optimal error relative to the deterministic truncated SVD~\cite{woodruff2014sketching}:
\[
    \|\bA - \bA_r\|_F^2 \leq (1 + \varepsilon) \|\bA - \bA_{\operatorname{best}}\|_F^2.
\]
Despite its powerful theoretical guarantees, satisfying the OSE property in practice often incurs a prohibitive sample complexity. To address this bottleneck, the recent work \cite{camano2025faster} introduces the Oblivious Subspace Injection (OSI), a geometric condition strictly weaker than the standard OSE. The OSI replaces the two-sided norm control of the OSE with two weaker requirements: isotropy in expectation and injectivity on any fixed $r$-dimensional subspace with high probability.

\begin{definition}[Oblivious Subspace Injection]
    A random matrix $\bOmega \in \bbF^{m \times N}$  is called an $(\alpha,\delta,r)$-OSI if both conditions are fulfilled:
    \begin{enumerate}
        \item \textbf{Isotropy:} $\E{\Vert \bOmega \bx \Vert^2_2 } = \Vert \bx \Vert^2_2$ for all $\bx \in \bbF^N$,
        \item \textbf{Injectivity:} For each fixed $r$-dimensional subspace $\mathcal{V} \subseteq \bbF^N$, with probability at least $1-\delta$,
        \[
            \alpha \Vert \bx \Vert^2_2 \leq \Vert \bOmega \bx \Vert^2_2 \qquad \text{for all } \bx \in \mathcal{V}.
        \]
    \end{enumerate}
\end{definition}

Despite being a weaker condition, the OSI property is sufficient to establish probabilistic error bounds for the randomized SVD, Nyström approximation, and sketch-and-solve algorithms, often allowing for significantly reduced sketch dimensions without substantially compromising the approximation accuracy~\cite{camano2025faster}.  In the context of the randomized low-rank approximation problem mentioned before, a sketch with an $(\alpha,\delta,r)$-OSI can achieve
\[
    \|\bA - \bA_r\|_F^2 \leq \frac{C(\delta)}{\alpha} \|\bA - \bA_{\operatorname{best}}\|_F^2,
\]
with $C(\delta) = \cO(1/\delta)$ in general by the Markov bound~\cite{camano2025faster}.

As with the OSE framework, the primary challenge in the OSI setting lies in balancing sufficient injectivity against a compact embedding dimension. This trade-off is particularly pronounced in the tensor regime, where certain classes of structured sketches suffer from an exponential dependence on the tensor order in either sample complexity or injectivity scaling~\cite{ahle2019oblivious, camano2025faster, saibaba2025}. 
This will be outlined more precisely in \Cref{sec:review}.

\section{Main Results} \label{sec:contributions}

In this section, we present the theoretical contributions of our work, establishing probabilistic guarantees for the proposed {\name} sketching framework and its use in applications. Proofs of the embedding results are postponed to \Cref{sec:OSI,sec:OSE}.

\subsection{{\name} Sketch}
We begin by formally introducing the structure of the tensorized random projection:

\begin{definition}[{\name} Sketch]\label{def:ttpr}
   Given integer parameters $P,R \in \mathbb{N}$, the {\name} sketch matrix $\bOmega_{\mathtt{\nameAbbrv}} \in \bbF^{PR \times N}$ with tensor-structured row blocks is defined as
   \begin{equation}\label{eq:ttpr}
        \bOmega_{\mathtt{\nameAbbrv}} := \frac{1}{\sqrt{ P }} \begin{bmatrix}
            \bigl(\cG^{(1,1)} \bowtie \cdots \bowtie \cG^{(1,d)}\bigr)^{\leq 1} \\ \vdots \\ \bigl(\cG^{(P,1)}  \bowtie \cdots \bowtie \cG^{(P,d)} \bigr)^{\leq 1}
        \end{bmatrix},
   \end{equation}
   where each core $\cG^{(j,k)} \in \mathbb{F}^{R \times n_k \times R}$, for $k \in [d-1]$ and $\cG^{(j,d)} \in \mathbb{F}^{R \times n_d \times 1}$  for $j \in [P]$, is a tensor with iid entries drawn from a Gaussian $\mathcal{N}_\bbF(0, 1/R)$ distribution.
\end{definition}

\begin{remark}
    We can readily observe that the {\name} sketch constitutes a parameterized family of sketching operators governed by the inputs $P$ and $R$. Notably, this formulation interpolates between other well-known tensor-adapted sketches: the Khatri-Rao sketch~\cite{saibaba2025,daas2025adaptive} when $R=1$, and the random Gaussian TT-tensor sketch~\cite{aldaas2023randomized,ma2022cost,kressner2023streaming} when $P=1$.
\end{remark}

\begin{remark}\label{rem:blocksparseterminology}
    We call this sketch "{\name}" because it may be reformulated as a contraction of a single block tensor train with a left boundary rank of dimension $PR$, by writing $\bOmega_{\mathtt{\name}} = \frac{1}{\sqrt{P}} (\widetilde{\cG}_1 \bowtie \cdots \bowtie \widetilde{\cG}_d)^{\leq 1}$, where the mode-wise slices of each core have the block-diagonal structure:
\[
    \widetilde{\cG}_j[i_j] =  \begin{bmatrix}
        \cG^{(1,j)}[i_j] &  &\\
        & \ddots & \\
        & & \cG^{(P,j)}[i_j]
    \end{bmatrix} \ \text{for } i_j \in [n_j], \ j \in [d-1], \quad \widetilde{\cG}_d[i_d] =  \begin{bmatrix}
        \cG^{(1,d)}[i_d] \\
        \vdots \\
        \cG^{(P,d)}[i_d]
    \end{bmatrix}, \ \text{for } i_d \in [n_d].
\]
\end{remark}

We note that the {\name} sketch is a structured sketch of embedding dimension $PR$ whose application cost is equivalent to contracting $P$ tensor trains of rank $R$ with the input tensor train of maximum TT-rank $\chi$. Thus, to sketch a tensor train of rank $\chi$ with {\name}, the cost scales as $\cO(dnPR \chi(\chi+R))$, which is only modestly larger than with the Khatri-Rao sketch. \Cref{sec:ttpr_application} gives further details on this topic, emphasizing further optimizations in the case of structured input cores occurring in linear combination, Hadamard products and matrix-vector products for vectors and operators in TT format.

We also propose the following variant, which increases the cost of generating the sketching matrix, but yields much improved results in numerical tests.
\begin{definition}[{Orthogonal \name} Sketch]\label{def:ottpr}
   Given integer parameters $P,R \in \mathbb{N}$, the Orthogonal {\name} sketch matrix $\bOmega_{\mathtt{O\nameAbbrv}} \in \bbF^{PR \times N}$ with tensor-structured row blocks is defined as
   \begin{equation}\label{eq:ottpr}
        \bOmega_{\mathtt{O\nameAbbrv}} := \frac{1}{\sqrt{ P }} \begin{bmatrix}
            \bigl(\sqrt{\frac{\rho_1n_1}{\rho_0}}\cU^{(1,1)} \bowtie \cdots \bowtie \sqrt{\frac{\rho_dn_d}{\rho_{d-1}}}\cU^{(1,d)}\bigr)^{\leq 1} \\ \vdots \\ \bigl(\sqrt{\frac{\rho_1 n_1}{\rho_0}}\cU^{(P,1)}  \bowtie \cdots \bowtie \sqrt{\frac{\rho_d n_d}{\rho_{d-1}}}\cU^{(P,d)} \bigr)^{\leq 1}
        \end{bmatrix},
   \end{equation}
   where each core $\cU^{(j,k)} \in \mathbb{F}^{\rho_{k-1} \times n_k \times \rho_k}$, for $k \leq d-1$, and $\cU^{(j,d)} \in \mathbb{F}^{\rho_{d-1} \times n_d \times 1}$, is a tensor with the unfoldings $(\cU^{(j,k)})^{\leq 1}$ as independent samples of the uniform distribution on the Stiefel manifold in $\bbF^{\rho_{k-1} \times n_k \rho_k}$ (i.e., matrices with orthonormal rows) with ranks $\rho_k = \min(R, n_k \cdots n_d)$ for $j \in [P]$, $k \in [d]$.
\end{definition}
\begin{remark}
    In the Khatri-Rao case ($R=1$), this orthogonalized variant amounts to choosing uniformly normalized Gaussian row vectors, i.e. the spherical Gaussian base distribution advocated for in~\cite{camano2025faster} because of its improved (but still exponential) scaling with dimension $d$.
\end{remark}
\begin{table}[t]
    \centering
    \scriptsize
    \renewcommand{\arraystretch}{1.2} 
    \setlength{\tabcolsep}{2.8pt} 
    
    \begin{tabular}{l c c c c}
        \toprule
        \multirow{2}{*}{\textbf{Method}} & \textbf{Embedding} & \multicolumn{2}{c}{\textbf{Theoretical guarantees}} & \textbf{Application} \\
        \cmidrule(lr){3-4}
        & \textbf{dimension} & \textbf{$(\varepsilon,\delta,r)$-OSE} & \textbf{$(1-\varepsilon,\delta,r)$-OSI} & \textbf{time} \\
        \midrule
        
        \makecell[l]{Khatri-Rao \\ Base $\bnu$  }  & $m=P$ & \makecell{$P = O\big(\varepsilon^{-2}r\log^d(P/\delta) \log(r/\delta)\big)$ \\ \cite{saibaba2025}} & \makecell{$P = \cO(\varepsilon^{-2} C_{\bnu}^d r) $ \\ \cite{Camano_Epperly_Tropp_2025}} & \makecell[r]{$\cO(d n P \chi^2)$} \\
        \addlinespace

        \makecell[l]{$f_{\mathtt{TT}(R)}$ \cite{rakhshan2020tensorized}} & $m=P$ & \makecell{$R = \cO(d)$ \\ $P = O\big( \varepsilon^{-2} (r + \log(1/\delta))^{2d} \big) $} & -- & \makecell[r]{$\cO(d n P R \chi (R + \chi))$} \\
        \addlinespace
        
        \makecell[l]{Gaussian \\ TT  \cite{aldaas2023randomized}} & $m=R$ & \makecell{$R = O\big(\varepsilon^{-2} d(r+\log(1/\delta))\big) $ \\ (This work)} & -- &  \makecell[r]{$\cO(d n R \chi (R + \chi))$} \\
        \addlinespace
        
        \makecell[l]{\nameAbbrv \\ (This work)} & $m=PR$ & \makecell{
        \bm{$R = O\big(d (r +\log(1/\delta))\big)$} \\ \bm{$P = \cO( \varepsilon^{-2} )$ }
        }
        & \makecell{\bm{$R = \cO(d)$} \\ \bm{$P = O\big( \varepsilon^{-2} (r +\log(r/\delta)) \big)$}} & \makecell[r]{$\cO\bigl(d n  P R \chi (R + \chi) \bigr )$} \\
        
        \bottomrule
    \end{tabular}
    \caption{Comparison of probabilistic guarantees for tensor-train-adapted sketches. $n$ denotes the maximum mode dimension, $P$ denotes the number of independent entries or blocks in the Khatri-Rao, $f_{\mathtt{TT}(R)}$ or {\name} sketches; $R$ the TT-rank of random TT cores; and $\chi$ the maximum TT-rank of tensor trains to be sketched.}
     \label{tab:sketches}
\end{table}

\subsection{Theoretical Guarantees}

\subsubsection{OSE guarantees}

We begin by establishing sufficient conditions on the parameters $P$ and $R$ for the {\name} sketch to satisfy the OSE property. We do not pursue such bounds for the Orthogonal {\name} sketch in this work, but numerical experiments indicate it outperforms the plain {\name} sketch.
\Cref{sec:review} places these bounds in the context of existing constructions.

\begin{theorem}\label{thm:OSE}
    The {\name} sketch is an $(\varepsilon,\delta,r)$-OSE provided
    \[
        \bm{R = \cO(d ( r + \log(1/\delta)))} \qquad \text{and} \qquad \bm{P = \cO(1/\varepsilon^2).}
    \]
\end{theorem}
\begin{remark}
    Another parameter regime where we can provide an $(\varepsilon,\delta,r)$-OSE guarantee is 
    \[
        R = \cO(d \log(1/\delta)) \qquad \text{and} \qquad P = \cO(r^2/\varepsilon^2).
    \]
    Although the resulting sketch dimension $PR$ is quadratic in the subspace dimension $r$, this regime may sometimes be of interest because of the trivial parallelism associated with stacking $P$ independent copies. 
\end{remark}
The proof follows the strategy in~\cite{ahle2019oblivious}, but with the precise constants tracked, where an OSE property is established for a wide class of structured sketches by bounding moments of the sketch. We postpone the proof to \Cref{sec:OSE}.

\begin{figure}[ht]
    \centering
    \includegraphics[width=\linewidth]{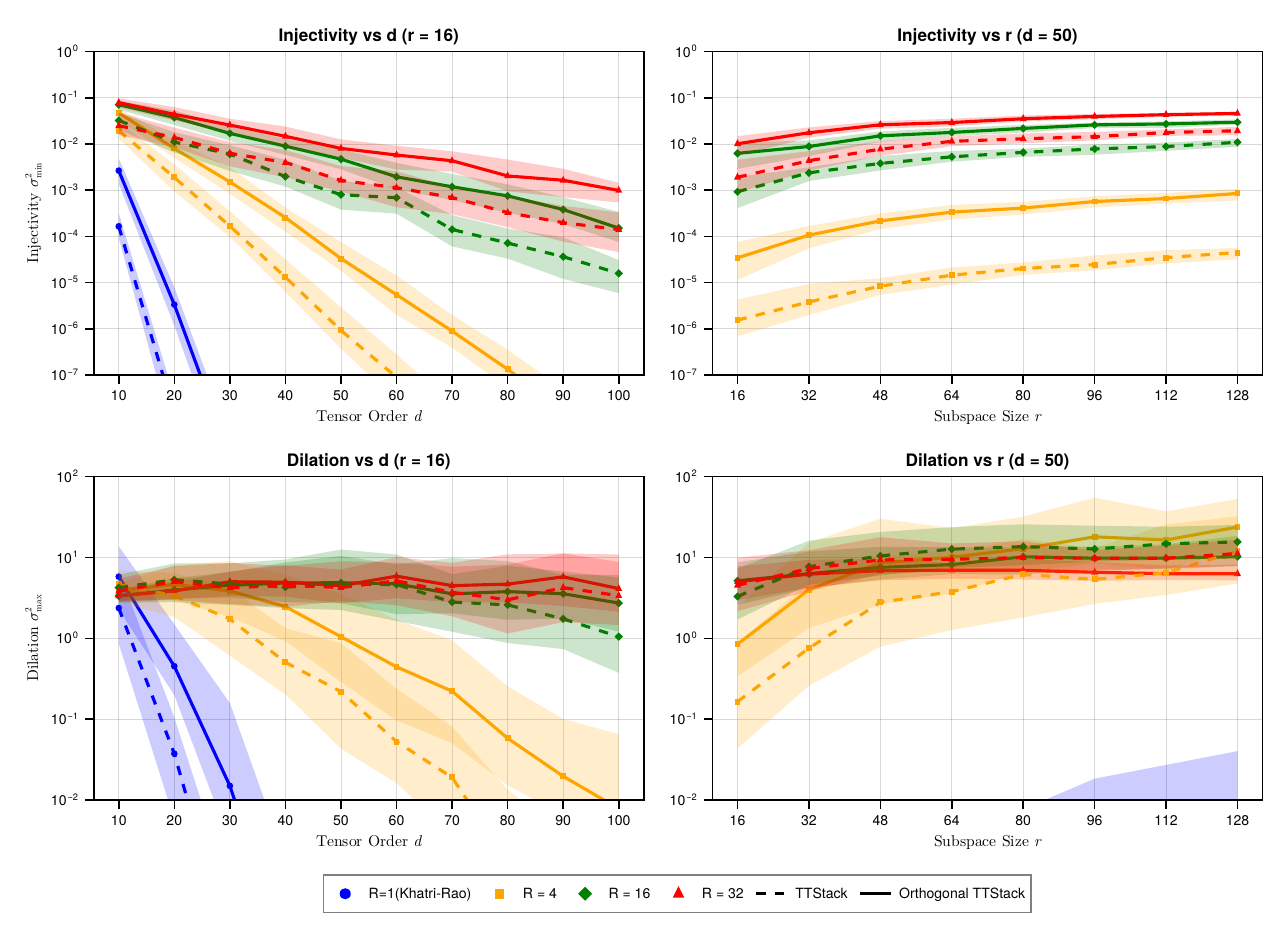}
    \caption{\textbf{{\name} Framework on Rank-1 Basis.} Empirical injectivity (left) and dilation (right) of the {\name} sketch. The operators are evaluated using embedding dimensions $PR = 2r$ and block ranks $R\in\{1,4,16,32\}$, applied to $r$-dimensional target subspaces spanned by Kronecker (rank-1) Gaussian TT vectors in $(\mathbb{R}^4)^{\otimes d}$. Markers indicate the median across 100 independent trials, with shaded regions denoting the interquartile range (25th to 75th percentiles).}

    \label{fig:Rscaling_rank1}
\end{figure}

\begin{figure}[ht]
    \centering
    \includegraphics[width=\linewidth]{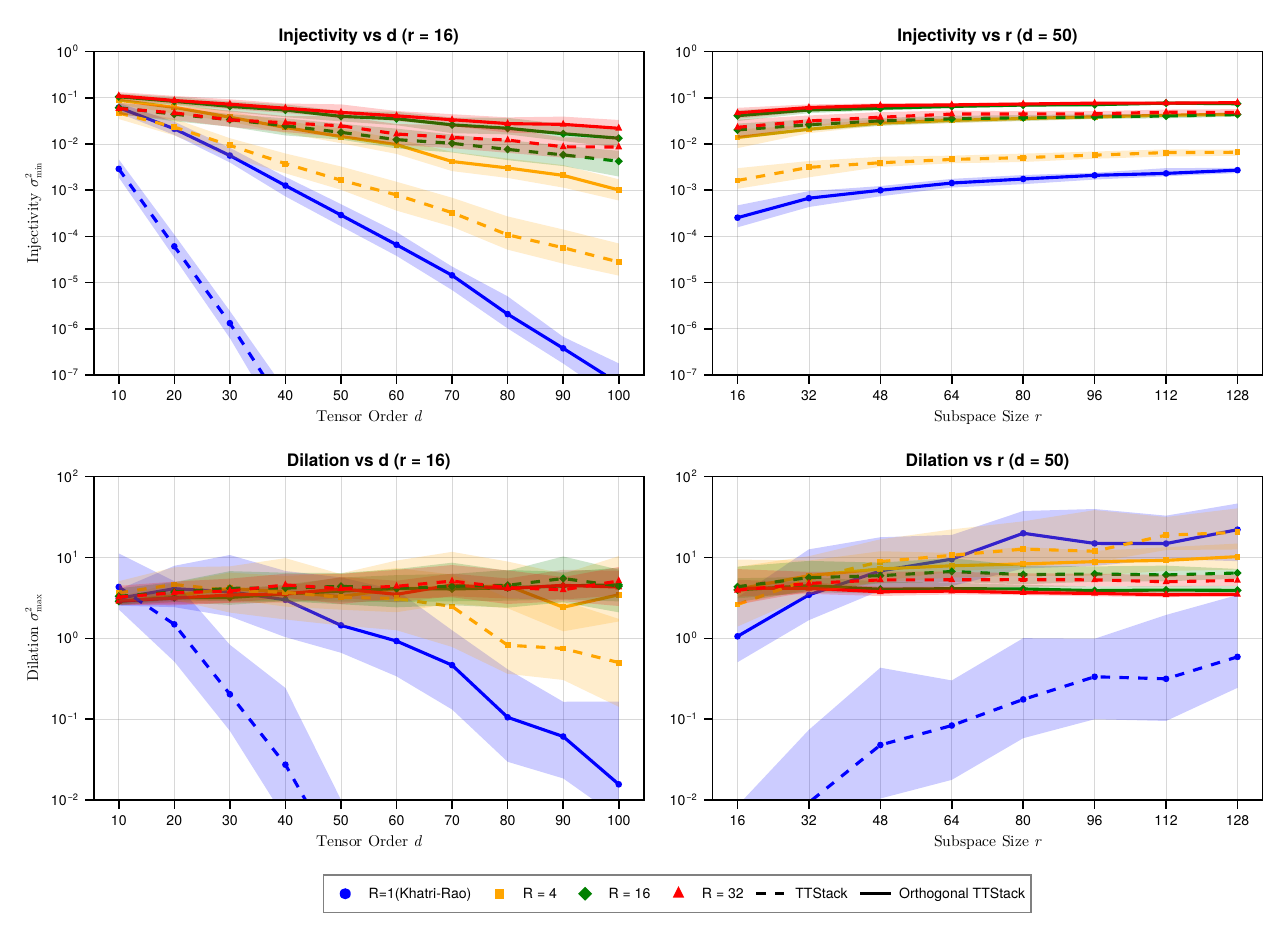}
    \caption{\textbf{{\name} Framework on Rank-4 Basis}. Empirical injectivity (left) and dilation (right) of the {\name} sketches. The operators are evaluated using embedding dimensions $PR = 2r$ and block ranks $R\in\{1,4,16,32\}$, applied to $r$-dimensional target subspaces spanned by rank-$4$ Gaussian TT vectors in $\mathbb{R}^{4^d}$. Markers indicate the median across 100 independent trials, with shaded regions denoting the interquartile range (25th to 75th percentiles).}

    \label{fig:Rscaling_rank4}
\end{figure}

Experiments reported in \Cref{fig:Rscaling_rank1,fig:Rscaling_rank4} evaluate the OSE property of the {\name} sketching framework. Specifically, we estimate the injectivity and dilation constants as functions of $d$ and $r$, for target subspaces spanned by rank-1 and rank-4 Gaussian TT vectors. While all block-structured sketches (particularly for $R \geq 16$) achieve a controlled embedding quality, this control tightens significantly as the TT-ranks of the basis increase. Notably, however, the exponential degradation of Khatri-Rao embeddings dissipates as the structure of the target subspace becomes more entangled -- a concept explored in more detail in \Cref{sec:OSI}.

Finally, we establish probabilistic guarantees for {\name} to enable the randomized singular value decomposition (or randomized rangefinder).

\begin{corollary}\label{cor:ttprsvd}
Let $\bA \in \bbF^{k \times N}$ be an input matrix, $r \leq k$ the target rank, $\bA_r$ the best rank-$r$ approximation of $\bA$. There exists a constant $c>0$  such that given a {\name} matrix  $\bOmega_{\mathtt{\nameAbbrv}} \in \bbF^{PR \times N}$ satisfying for some $(\varepsilon,\delta) \in (0,1)$
    \begin{equation} \label{eq:randsvd_bstt}
        \bm{R \geq c r d \log(2/\delta), \qquad P \geq 16 e^4 \max( 4, 1/\varepsilon ),}
    \end{equation}
    the rank-$r$ \texttt{QB} factorization of $\widehat{\bA} \in \bbF^{PR \times N}$ determined by
    \[
        \bQ := \mathrm{Orth}(\bA \bOmega_{\mathtt{\nameAbbrv}}^*) \in \bbF^{k \times PR}, \qquad \widehat{\bA} = \bQ (\bQ^* \bA),
    \]
    satisfies the following with probability at least $1-\delta$:
    \[
        \| \bA - \widehat{\bA}  \|_F \leq (1+\varepsilon) \| \bA - \bA_r \|_F.
    \]

\end{corollary}
This result is proven in \Cref{subsec:bstt}. Note that the parameter scaling is not practical, forcing large block ranks $R$ scaling as $\cO \bigl (d r \log(1/\delta) \bigr )$, while the sketch application efficiency depends in practice on keeping $R$ small or slowly increasing, with a larger number of copies $P$. In the following section, we show that $R = \cO(d)$ is enough to provide a slightly weaker guarantee.

\subsubsection{OSI guarantees}

While the OSE property provides the strongest geometric guarantee, it often imposes conservative bounds on the sketch parameters. With the analysis developed in \Cref{sec:OSI}, we show that {\name} satisfies the OSI property under milder conditions on the block count $P$ and rank $R$, and we further identify a notion of \textit{subspace entanglement} $C_\bQ(R)$ (\Cref{def:entanglement}). The constant $C_\bQ(R)$ captures the dependence of the embedding error on Kronecker-structured (or low-entanglement) vectors in the subspace spanned by $\bQ$. Such vectors are precisely those for which the Khatri-Rao sketch degrades exponentially, a phenomenon termed \textit{overwhelming orthogonality}~\cite{camano2025faster,overwhelming_ortho}.

\begin{theorem}\label{thm:OSI}
    Fix an orthonormal matrix $\bQ \in \bbF^{N \times r}$. There exists a measure of subspace entanglement $0 < C_\bQ(R) \leq \change{\sqrt{ \left ( 1+\frac{p_\bbF}{R}\right )^d - 1}}$, with $p_\bbR = 2$ or $p_\bbC = 1$ depending on the choice of scalar field $\bbF$, such that a random {\name} matrix $\bOmega_{\mathtt{\nameAbbrv}} \in \bbF^{PR \times N}$ satisfies $\sigma_{min}^2(\bOmega_{\mathtt{\nameAbbrv}} \bQ) \geq 1-\varepsilon$ with probability at least $1-\delta$ provided
    \begin{equation}\label{eq:OSI_condition}
         P \geq 4 \varepsilon^{-2} \left [ \change{4} C^2_\bQ(R) \; r +  (1+C^2_\bQ(R)) \log(2r/\delta) \right ].
    \end{equation}
    In particular, the {\name} matrix enjoys the $(1-\varepsilon,\delta,r)$-OSI property provided \bm{$R \sim p_\bbF d$} and \bm{$P = \cO(\varepsilon^{-2} (r + \log(r/\delta)))$}.
\end{theorem}

\begin{remark}
    The precise definition of $C_\bQ(R)$ can be found in \Cref{sec:OSI}. In particular, $C_\bQ(R)$ is maximized whenever there exists a vector $\bu = \bu_1 \otimes \cdots \otimes \bu_d \in \mathrm{Span}(\bQ)$ with Kronecker product structure, as seen from \Cref{rem:kronmax}. Additionally, it is easy to see that \change{the OSI condition is satisfied for $R \geq p_\bbF d$ and $P \geq 4 \varepsilon^{-2} \left [ 4 (e-1) \; r +  e \log(2r/\delta) \right ]$, therefore obtaining practical subspace-oblivious guarantees scaling linearly in the tensor order since $C^2_\bQ = \cO(p_\bbF d/R) \leq e-1$}.
\end{remark}

\Cref{tab:sketches} summarizes the comparative analysis described in \Cref{sec:review}, showing that while the Khatri-Rao sketch necessarily exhibits an exponential scaling in $d$ to achieve an OSE/OSI~\cite{ahle2019oblivious,camano2025faster}, the {\name} sketch achieves an OSE for block ranks $R = \cO(d \log(1/\delta)$, or an OSI for $R = \cO(d)$, whenever the number of blocks $P$ satisfies the conditions of \Cref{thm:OSE,thm:OSI} respectively. Numerical experiments conducted on the empirical injectivity and dilation $\sigma_{min/max}^2(\bOmega_\mathtt{(O)\nameAbbrv} \bQ )$ for {\name} embeddings with block rank $R = 2d$ applied to subspaces of dimension 16 for $d \leq 100$, presented in \Cref{fig:r2d_scaling}, confirm this phenomenon. In particular, the injectivity parameter seems asymptotically constant, $\alpha = \cO(1)$, and exhibits relatively small variance. The distortion parameter on the other hand exhibits larger variance (especially for $P=1$). Theory indicates its distribution is quite heavy-tailed, as it behaves conceptually as a product of $d$ chi-squared distributed variables. In all cases, the Orthogonal {\name} variant has better performance than the basic {\name} sketch, and increasing $P$ clearly reduces the embedding distortion. 

\begin{figure}[ht]
    \centering
    \includegraphics[width=\textwidth]{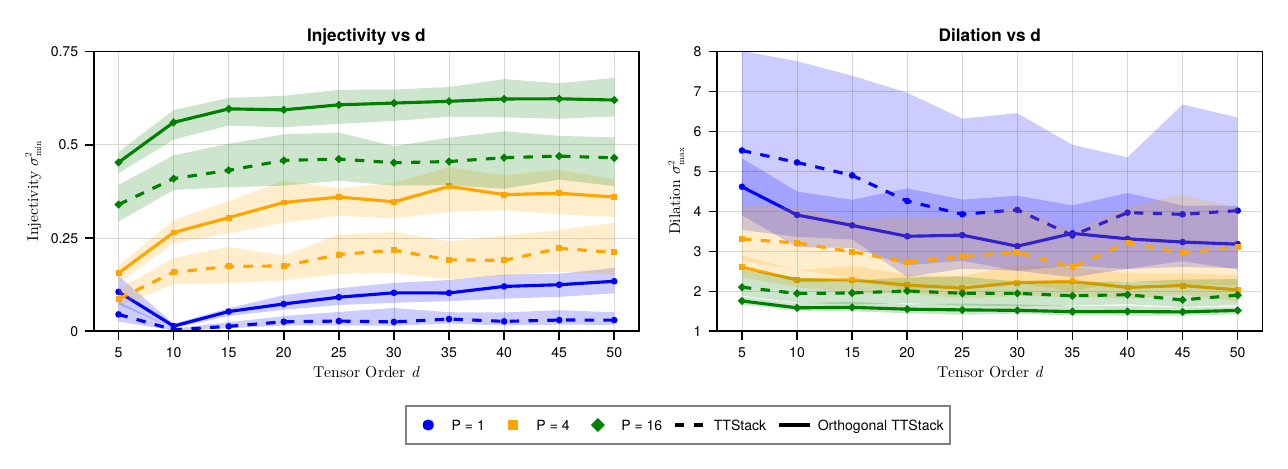}
    \caption{\textbf{Numerical OSE of {\name} Framework on Rank-1 Basis} Empirical injectivity (left) and dilation (right) of the {\name} sketches. The operators are evaluated using embedding dimensions $PR$ with $P \in \{1,4,16\}$ blocks of rank $R = 2d$, applied to a $16$-dimensional target subspace spanned by Kronecker (rank-1) Gaussian TT vectors in $(\mathbb{R}^4)^{\otimes d}$. Markers indicate the median across 100 independent trials, with shaded regions denoting the interquartile range (25th to 75th percentiles).
    }
    \label{fig:r2d_scaling}
\end{figure}

 As a direct consequence of \Cref{thm:OSI}, we establish probabilistic guarantees for the randomized rangefinder implemented with the {\name} sketch in the OSI parameter regime.
\begin{corollary}
\label{thm:rsvd_ttpr}
    Let $\bA \in \bbF^{k \times N}$ be an input matrix, $r \leq k$ the target rank, $\bA_r$ the best rank-$r$ approximation of $\bA$, and $\bOmega_{\mathtt{\nameAbbrv}} \in \bbF^{PR \times N}$ a {\name} sketch realizing an $(\alpha,\delta/2,r)$-OSI of the subspace spanned by the dominant $r$ singular vectors of $\bA$ as prescribed by~\eqref{eq:OSI_condition}.
    Then, the rank-$PR$ \texttt{QB} factorization of $\widehat{\bA} \in \bbF^{k \times N}$ determined by
    \[
        \bQ := \mathrm{Orth}(\bA \bOmega_{\mathtt{\nameAbbrv}}^*) \in \bbF^{k \times PR}, \qquad \widehat{\bA} = \bQ (\bQ^* \bA),
    \]
    satisfies the following bound with probability at least $1-\delta$:
    \[
        \Vert \bA - \widehat{\bA} \Vert_F^2 \leq C_\delta \Vert \bA - \bA_r\Vert_F^2,
    \] 
    for the approximation constant:
    \[
        C_\delta := 1 + \alpha^{-1}  \left ( 1 + \sqrt{\frac{\bigl ( 1 + p_\bbF /R \bigr )^d - 1}{P \delta / 2}}\right ) =\cO \left(  \frac{1 + \sqrt{d/(PR\delta)}}{\alpha} \right) \text{ for } R \geq d.
    \]
\end{corollary}
This result is proven in \Cref{sec:rsvd_ttpr_proof}, and improves on the similar Theorem 2.2 in~\cite{camano2025faster} for generic OSIs with a better and explicit constant, reducing the dependency from $1/\delta$ to $\cO(1 + \sqrt{d/(PR\delta)})$ for $R \geq d$.

\begin{remark}\label{rem:bstt_svd}
    If we require that the sketch $\bOmega_\mathtt{\nameAbbrv}$ satisfies a Strong $(\alpha \sqrt{\varepsilon/r}, \delta/2)$-Johnson-Lindenstrauss Moment property, \Cref{def:SJLM}, then we may achieve an $(1+\varepsilon)$-accurate randomized SVD with probability $1-\delta$, thanks to~\Cref{lem:rsvdfromose}. This increases the sketching rank conditions to
    \[
        R \geq C d \log(2/\delta), \qquad P \geq 16 e^4 r / (\alpha^2 \varepsilon),
    \]
    for a universal constant $C > 0$.
\end{remark}

Finally, we address the application of our framework to operations intrinsic to the tensor train format. An important implication of the property established by \Cref{thm:OSI} is the quality of the {\name} embedding for its application to the Randomized TT-Rounding. The following theorem demonstrates that our sketch produces a rounding approximation that preserves the quasi-optimality of the randomized SVD.

\begin{theorem}\label{thm:rand_round_simple}
     Let $\cA = \cA_1 \bowtie \cdots \bowtie \cA_d$ be a tensor train and $(r_1, \dots, r_{d-1})$ be our target ranks. Then, the rounded tensor $\widetilde{\mathcal{A}}$, computed by Randomize-then-Orthogonalize (\Cref{alg:tt_randrounding}) with sketch ranks fulfilling condition~\eqref{eq:OSI_condition} with \change{$\delta=\delta_{\mathtt{rto}}/(2d-2)$} and $r = \max_{k \in [d-1]} r_k$, satisfies the following with probability larger than $1-\delta_{\mathtt{rto}}$:
    \[
        \Vert \cA - \widetilde{\mathcal{A}} \Vert_F^2 \leq (d-1)\, \change{C_{\delta_{\mathtt{rto}}/(d-1)}} \Vert \cA - \cA_\textrm{best} \Vert^2_F,
    \]
    with $\mathcal{A}_\textrm{best}$ the best approximation of $\cA$ with TT-ranks at most $(r_1, \dots, r_{d-1})$ and \change{$C_{\delta_{\mathtt{rto}}/(d-1)}$} the constant introduced in \Cref{thm:rsvd_ttpr}.
\end{theorem}
This result is proven in \Cref{sec:randroundproof}, with a more precise statement (\Cref{thm:rand_round}) providing explicit constants.
\begin{remark}
    The improvement $C_\delta = \cO(1+\varepsilon)$ can be achieved by considering larger block ranks $R = \cO(d \log(1/\delta))$, see \Cref{rem:bstt_svd}.
\end{remark}

\section{Tensor Train Sketching: Review and Comparative Analysis}
\label{sec:review}
Although the field of randomized numerical linear algebra has produced a rich variety of sketching operators, optimized for properties ranging from sparsity~\cite{optimal_sparsity} to a linear dependence on the tensor order~\cite{ahle2019oblivious}, few are suitable for randomized algorithms such as TT rounding and orthogonalization. To that end, in this short survey we restrict our scope to the limited class of structured sketches that facilitate the support of these algorithms through partial sketches, while referring the corresponding proofs to \Cref{sec:OSE}.

\subsection{Khatri-Rao Random Projections}\label{sec:KR}

 Khatri-Rao sketches serve as a fundamental building block for structured dimensionality reduction even outside the scope of high-dimensional tensor spaces. They provide a mechanism to construct a global sketching operator $\bOmega_\mathtt{KR} \in \bbF^{P \times N}$ from a sequence of smaller, mode-specific sketching maps $\bOmega_k \in \bbF^{P \times n_k}$.

Given two matrices $\bX \in \bbF^{P \times n_1}$, $\bY \in \bbF^{P \times n_2}$, with rows $\bx_1,...,\bx_P$ and $\by_1,...,\by_P$ respectively, define their (row-wise) Khatri-Rao product (or row-wise Kronecker product) as 
\[
    \bX \odot \bY = \begin{bmatrix}
            \bx_1 \otimes \by_1 \\ \cdots \\ \bx_P \otimes \by_P
        \end{bmatrix} \in \bbF^{P \times n_1 n_2} 
\]
where $\otimes$ denotes the standard Kronecker product.

\begin{definition}[Khatri-Rao Sketch]
    Given an isotropic random vector distribution $\bnu_k \in \bbF^{n_k}$ for $k \in [d]$, called the base distribution, the Khatri-Rao sketch matrix $\bOmega_\mathtt{KR} \in \bbF^{P \times N}$ is constructed by taking the Khatri-Rao product of a set of matrices $\bG_k \in \bbF^{P \times n_k}$ for $k \in [d]$ with iid rows $\bg_{jk} \sim {\bnu_k}$ for $j\in[P]$:
    \begin{equation}\label{eq:KRP}
        \bOmega_\mathtt{KR} := \bG_1 \odot \bG_2 \odot \cdots \odot \bG_d = \begin{bmatrix}
            \bg_{1,1} \otimes \dots \otimes \bg_{1,d} \\ \vdots \\ \bg_{P,1} \otimes \dots \otimes \bg_{P,d}
        \end{bmatrix}.
    \end{equation}
\end{definition}

This structure is particularly advantageous for small to moderate regimes because it avoids the explicit storage of the sketching operator $\bOmega_\mathtt{KR}$ through an efficient sequence of tensor contractions. However, the large-scale regime imposes a significant limitation as the required embedding dimension to achieve an OSE is exponential in $d$. Theorem B.1 in \cite{saibaba2025} indeed exhibits that for the standard Gaussian base distribution and a constant $C>1$, the Khatri-Rao sketch entails,
\[
    P \geq 2.6 \varepsilon^{-2}(r + 2rC^d \log^d(8N/\delta))\log(4r\delta).
\]
Furthermore, Theorem 41 in~\cite{AhleKKPVWZ19-arxiv} illustrates that an OSE guarantee for any sketching operator of the form~\eqref{eq:KRP} will inevitably lead to an exponential dependence on $d$. 

The choice of the base distribution critically impacts practical performance, with standard options ranging from vectors with iid Gaussian or Rademacher entries, as well as normalized spherical random vectors. To guarantee a constant OSI bound in the uniform setting ($n_1 = \dots = n_d$ and $\bnu_1 = \ldots = \bnu_d = \bnu$), the required embedding dimension scales exponentially as $P = \cO(C_{\bnu}^d r)$. The exponential base $C_{\bnu}$ is strictly determined by the selected distribution $\bnu$. Notably, for tensors with mode sizes $n_k = 2$, the complex spherical distribution yields the lowest constant of $C_{\bnu} = 4/3$, compared to the common real Gaussian distribution, which yields $C_{\bnu} = 3$.

\subsection{TT Random Projections}

Introduced by Rakhshan and Rabusseau~\cite{rakhshan2020tensorized}, the tensorized random projection map $f_{\mathtt{TT}(R)}$ is constructed by enforcing a low-rank tensor structure on the rows of the random projection matrix.

\begin{definition}[Definition 1 in~\cite{rakhshan2020tensorized}, paraphrased]\label{def:fttr}
    A TT random projection of rank $R$ is a linear
    map $f_{\mathtt{TT}(R)} : \bbF^N \rightarrow \bbF^P$ defined similarly to~\eqref{eq:ttpr}, as the application of a random matrix with tensor-structured rows:
    \[
        f_{\mathtt{TT}(R)}(\bx) = \frac{1}{\sqrt{ P }} \begin{bmatrix}
            \bigl(\cG^{(1,1)} \bowtie \cdots \bowtie \cG^{(1,d)}\bigr)^{\leq 1} \\ \vdots \\ \bigl(\cG^{(P,1)}  \bowtie \cdots \bowtie \cG^{(P,d)} \bigr)^{\leq 1}
        \end{bmatrix} \bx,
   \]
   where for $j \in [P]$, cores $\cG^{(j,1)} \in \mathbb{F}^{1 \times n_1 \times R}$ and $\cG^{(j,d)} \in \mathbb{F}^{R \times n_d \times 1}$ are tensors with iid entries drawn from $\mathcal{N}_\bbF(0, 1/\sqrt{R})$, while cores $\cG^{(j,k)} \in \mathbb{F}^{R \times n_k \times R}$ for $k=2,\ldots,d-1$, are tensors with iid entries drawn from $\mathcal{N}_\bbF(0, 1/R)$. 
\end{definition}

Whereas the Khatri-Rao sketch suffers from an exponentially growing variance, $f_{\mathtt{TT}(R)}$ mitigates this issue through the rank parameter $R$~\cite{rakhshan2020tensorized}. Indeed, in the case $\bbF = \bbR$, we get
\[
    \operatorname{Var}\left( \|f_{\mathtt{TT}(R)}\cX\|_2^2 \right) \leq \frac{1}{P} \left( 3 \left( 1 + \frac{2}{R}\right)^{d-1} - 1 \right) \|\cX\|_F^4.
\]
This variance control implies a distributional Johnson-Lindenstrauss property of the form:
\[
    \P{ \bigl \vert \| f_{\mathtt{TT}(R)}(\bx) \|^2_2 - \|x\|_2^2 \bigr \vert \geq \varepsilon \|x\|_2^2 } < \delta,
\]
under the condition $P = O \bigl ( \varepsilon^{-2}(1+2/R)^d \log^{2d}(1/\delta) \bigr )$ (\cite{rakhshan2020tensorized}, Theorem 2), where the proof depends on the hypercontractivity of Gaussian Chaoses. 

This shows that tensorized random projections may improve exponentially upon the theoretical guarantees of the Khatri-Rao sketch, as choosing $R=\cO(d)$ makes the variance bound independent of $d$. Nevertheless, an exponential logarithmic factor remains.
\begin{remark}
    From \Cref{def:fttr} it is clear that $f_{\mathtt{TT}(R)}$ can be seen as a down-sampled {\name} sketch. The exponential scaling requirement in $P$ may thus be quite pessimistic, with conditions similar to those of \Cref{thm:OSE,thm:OSI} being sufficient, but we do not pursue this analysis in this work.
\end{remark}
\subsection{Random Gaussian TT-tensors}

Another empirically successful tensor train-adapted sketch was proposed in the context of optimizing the TT rounding procedure~\cite{aldaas2023randomized,kressner2023streaming}. Its application consists of the contraction with a random block tensor train $\cG := \cG_1 \bowtie \cdots \bowtie \cG_d$ with given ranks $r_1, \dots, r_{d-1} \geq 1$ and boundary ranks $r_0 = R > 1$, $r_d = 1$.
The execution of this sketch is efficient when $\cX$ is given in tensor train format~\cite{aldaas2023randomized}, and because of identity~\eqref{eq:recursiveTT}, it may be reformulated in the following equivalent way, which is more amenable to analysis.

\begin{definition}\label{def:GTT}[Gaussian TT Sketch]
    Given ranks $\{r_0,\dots,r_d\}$ with $r_0 = R$, $r_1, \dots, r_{d-1} \geq 1$ and $r_d = 1$, the corresponding Gaussian Tensor Train sketch is the random matrix
    \begin{equation}\label{eq:gtt}
        \bOmega_\mathtt{GTT} :=  (\cG_1 \bowtie \cdots \bowtie \cG_d)^{\leq 1} = \bG_1 (\bI_{n_1} \otimes \bG_2) \cdots (\bI_{n_1 \cdots n_{d-1}} \otimes \bG_d) \in \bbF^{R \times N},
    \end{equation}
    where each matrix $\bG_k := \cG_k^{\leq 1} \in \bbF^{r_{k-1}\times n_k r_k}$ has iid entries drawn from $\cN_\bbF(0,1/r_{k-1})$ for $k\in[d]$.
\end{definition}

Although remarkably effective in practice, to the best of our knowledge, no rigorous statements of probabilistic guarantees for this embedding have been published in the literature, although closely related results have been advanced~\cite{ahle2019oblivious,ma2022cost}. We prove the following statements in \Cref{sec:sjlm_ose_gtt}:

\begin{theorem} \label{thm:GTT-OSE}
    The Gaussian Tensor Train sketch $\bOmega_\mathtt{GTT}$ with ranks $(r_0, \dots, r_{d-1})$ satisfies an $(\varepsilon,\delta,r)$-OSE property provided
    \[
        r_i = \cO\left( (d-i) \; \frac{r + \log(1/\delta)}{\varepsilon^2} \right)  \qquad \text{for  } i \in [d-1],
    \]
    and the embedding size satisfies $R = r_0 = \cO \bigl ( d \varepsilon^{-2} (r + \log(1/\delta)) \bigr )$.
\end{theorem}
While the Gaussian TT sketch achieves better probabilistic guarantees and empirical performance than the Khatri-Rao sketch, we note that it incurs a substantially higher computational cost that scales quadratically in $R$. Additionally, it does not allow for easy incremental expansion of the sketch by adding new, independent rows to the sketching matrix as is possible in the Khatri-Rao case~\cite{daas2025adaptive}. 

\begin{theorem}\label{th:gttrsvd}
    Let $\bA \in \bbF^{k \times N}$ and $\varepsilon, \delta \in (0,1]$.
    There exists a constant $c \geq 0$ such that for a Gaussian Tensor Train sketch $\bOmega_\mathtt{GTT}$ with ranks satisfying $R = r_0 \geq d r_{d-1}$, $r_i \geq (d-i) r_{d-1}$ for $i  \in [d-2]$, and
    \begin{equation} \label{eq:randsvd_gtt}
        r_{d-1} \geq c r/\varepsilon \log(2/\delta),
    \end{equation}
    the following holds with probability at least $1-\delta$, 
    \[
        \| \bA - \bQ \bQ^* \bA  \|_F \leq (1+\varepsilon) \| \bA - \bA_r \|_F,
    \]
    where $\bA_r$ is the best rank-$r$ approximation to $\bA$ and $\bQ \bQ^*$ is the orthogonal projector onto the column space of $ \bA \bOmega_\mathtt{GTT}^*$.
\end{theorem}

\section{Applications: TT Rounding}
\label{sec:applications}
Our primary application of interest is the efficient randomization of the TT rounding scheme, which we recall in \Cref{alg:tt_rounding}. The computational bottleneck in \Cref{alg:tt_rounding} is the orthogonalization sweep of the TT cores, which requires the QR decomposition of matrices of size $n_kR_{k-1} \times R_k$. As a way of overcoming this constraint, randomized algorithms~\cite{rakhshan2020tensorized,aldaas2023randomized,kressner2023streaming} have been proposed to accelerate this process. 

\subsection{Algorithms}\label{sec:ttroundalgorithms}
\begin{algorithm}[ht]
  \caption{Deterministic TT-rounding~\cite{oseledets}}
  \label{alg:tt_rounding}
  \begin{algorithmic}[1]
  \Require $(\mathcal{X}_1,\dots,\mathcal{X}_d)$ TT representation with initial TT-ranks $(R_0,\dots,R_d)$, target rank $(r_{0},\dots,r_d)$ 
  \Ensure $(\mathcal{X}_1,\dots,\mathcal{X}_d)$ {right-orthogonal TT representation} with target ranks $(r_{0},\dots,r_d)$
  \Statex
  \Function{TT-Rounding}{$(\mathcal{X}_1,\dots,\mathcal{X}_d),(r_{0},\dots,r_d)$}
    \For{$k=1$ to $d-1$} \Comment{Orthogonalization of the TT cores}
      \State $\bQ_k,\mathbf{R}_k \leftarrow \mathtt{qr}((\mathcal{X}_k)_{i_k\alpha_{k-1}}^{\alpha_k})$
      \State $(\mathcal{X}_k)_{i_k\alpha_{k-1}\alpha_k} \leftarrow (\bQ_k)_{i_k\alpha_{k-1}}^{\alpha_k}$
      \State $(\mathcal{X}_{k+1})_{i_{k+1}\alpha_k\alpha_{k+1}} \leftarrow \sum_{\beta_k} (\mathbf{R}_k)_{\alpha_k\beta_k} (\mathcal{X}_{k+1})_{i_{k+1}\beta_k\alpha_{k+1}}$
    \EndFor
    \For{$k=d$ to $2$} \Comment{Truncation of the TT ranks}
        \State $\bU_k,\Sigma_k,\bV_k^\top = \mathtt{svd}((\mathcal{X}_k)_{\alpha_{k-1}}^{i_k\alpha_{k}},r_{k-1})$ \Comment{Truncated SVD to rank $r_{k-1}$}
        \State $(\mathcal{X}_k)_{\alpha_{k-1}i_k\alpha_{k}} \leftarrow (\bV_k^\top)_{\alpha_{k-1}}^{i_k\alpha_{k}}$
        \State $(\mathcal{X}_{k-1})_{\alpha_{k-2}i_{k-1}\alpha_{k-1}} \leftarrow \sum_{\beta_{k-1}} (\mathcal{X}_{k-1})_{\alpha_{k-2}i_{k-1}\beta_{k-1}}(\bU_k)_{\beta_{k-1}\beta_{k-1}'}(\Sigma_k)_{\beta'_{k-1}} $
    \EndFor
    \State \Return $(\mathcal{X}_1,\dots,\mathcal{X}_d)$
  \EndFunction
  \end{algorithmic}
\end{algorithm} 
 
In the matrix case, where $\bX = \bC_1 \bC_2$ is decomposed by $\bC_1 \in \mathbb{F}^{n_1 \times R}, \bC_2 \in \mathbb{F}^{R \times n_2}$, the original TT rounding algorithm (\ref{alg:tt_rounding}) determines the first (and in this case only) orthogonal factor $\bQ_1$ from the QR factorization of $\bC_1$. 
This is inefficient as the dimension $R$ may be much larger than the numerical rank of the matrix $\bX$.
Randomization is a solution to this issue, and is usually carried out in one of two ways:
\begin{itemize}
    \item As a randomized SVD~\cite{HMT}, where the $Q$ factor is computed from the matrix $\bX \bOmega^*$, for some sketching matrix $\bOmega$;
    \item Or as a Generalized Nyström Approximation~\cite{Nakatsukasa_2020}, by writing $\bX \simeq (\bX \bOmega_2^*) (\bOmega_1 \bX \bOmega_2^*)^\dagger (\bOmega_1 \bX)$, for some sketching matrices $\bOmega_1 , \bOmega_2$.
\end{itemize}
Generalizing these approaches to tensor trains yields \Cref{alg:tt_randrounding} in the first case~\cite{aldaas2023randomized}, and \Cref{alg:stta} in the second case~\cite{kressner2023streaming}. In both scenarios, a central component is the efficient application of the sketch to partial unfoldings of the tensor via \textit{partial contractions} (\Cref{alg:partialcontractions}), which we review next.

\begin{algorithm}[ht]
  \caption{TT-rounding: Randomize-then-Orthogonalize (Alg. 3.1 in~\cite{aldaas2023randomized}),~\cite{Huber_Schneider_Wolf_2017}}
  \label{alg:tt_randrounding}
  \begin{algorithmic}[1]
  \Require $(\mathcal{X}_1,\dots,\mathcal{X}_d)$ TT representation with initial TT-ranks $(R_0,\dots,R_d)$, target rank $(r_{0},\dots,r_d)$ 
  \Ensure $(\mathcal{Y}_1,\dots,\mathcal{Y}_d)$ {left-orthogonal} TT representation with target ranks $(r_{0},\dots,r_d)$
  \Statex
  \Function{randomized-TT-rounding}{$(\mathcal{X}_1,\dots,\mathcal{X}_d),(r_{0},\dots,r_d)$}
    \State Generate random TT sketch $(\mathcal{G}_2,\dots,\mathcal{G}_d)$ with target ranks $(r_{0},\dots,r_d)$ 
    \State $(\mathbf{W}_k)_{2 \leq k \leq d} = \textsc{PartialContractions}((\mathcal{X}_2,\dots,\mathcal{X}_d),(\mathcal{G}_2,\dots,\mathcal{G}_d))$
    \For{$k=1$ to $d-1$}
      \State $(\mathbf{Z}_k)_{i_k\alpha_{k-1}}^{\alpha_k} \leftarrow \sum_{\beta_k} (\mathcal{X}_k)_{i_k\alpha_{k-1}\beta_k} (W_k)_{\beta_k \alpha_k}$ \Comment{$n_k r_{k-1} \times r_{k}$}
      \State $\bQ_k,\mathbf{R}_k \leftarrow \textsc{qr}((\mathbf{Z}_k)_{i_k\alpha_{k-1}}^{\alpha_k})$
      \State $(\mathcal{Y}_k)_{i_k\alpha_{k-1}\alpha_k} \leftarrow (\bQ_k)_{i_k\alpha_{k-1}}^{\alpha_k}$
      \State $(\mathcal{X}_{k+1})_{i_{k+1}\alpha_k\alpha_{k+1}} \leftarrow \sum_{\beta_k} (\mathbf{R}_k)_{\alpha_k\beta_k} (\mathcal{X}_{k+1})_{i_{k+1}\beta_k\alpha_{k+1}}$
    \EndFor
    \State $(\mathcal{Y}_d)_{i_d\alpha_{d-1}\alpha_d} \leftarrow (\mathcal{X}_d)_{i_d\alpha_{d-1}\alpha_d}$ \Comment{$\alpha_d=1$}
    \State \Return $(\mathcal{Y}_1,\dots,\mathcal{Y}_d)$
  \EndFunction
  \end{algorithmic}
\end{algorithm} 

\begin{algorithm}[ht]
  \caption{Streaming TT approximation ~\cite{kressner2023streaming}, (Alg. 3.2 in~\cite{aldaas2023randomized})}
  \label{alg:stta}
    \begin{algorithmic}[1]
    \Require $(\mathcal{X}_1,\dots,\mathcal{X}_d)$ TT representation with initial TT-ranks $(R_0,\dots,R_d)$, target rank $(r_{0},\dots,r_d)$, oversampling $(\ell_0,\dots,\ell_d)$ 
    \Ensure $(\mathcal{Y}_1,\dots,\mathcal{Y}_d)$ TT representation with target ranks $(r_{0},\dots,r_d)$
    \Statex
    \Function{STTA}{$(\mathcal{X}_1,\dots,\mathcal{X}_d),(r_{0},\dots,r_d), (\ell_0,\dots,\ell_d)$}
      \State Generate random TT sketch $(\mathcal{L}_{\leq k})_{0\leq k \leq d}$ with target ranks $(r_{k})_{0 \leq k \leq d}$
      \State Generate random TT sketch $(\mathcal{R}_{>k})_{0 \leq k \leq d}$ with target ranks $(r_{k}+\ell_k)_{0 \leq k \leq d}$
      \For{$k=1$ to $d$}
        \State  $(\mathbf{S}_k)_{\alpha_k\beta_k} \leftarrow (\mathcal{L}_{\leq k})_{\alpha_k}^{i_1\dots i_k} (\mathcal{X})_{i_1\dots i_k}^{i_{k+1}\dots i_d} (\mathcal{R}_{>k})_{i_{k+1}\dots i_d}^{\beta_k}$ \Comment{$r_k \times (r_k + \ell_k)$} 
        \State $(\mathcal{Z}_k)_{i_k \alpha_{k-1}\alpha_k} \leftarrow (\mathcal{L}_{\leq k-1})_{\alpha_{k-1}}^{i_1\dots i_{k-1}} (\mathcal{X})_{i_1\dots i_{k-1};i_k}^{i_{k+1}\dots i_d} (\mathcal{R}_{>k})_{i_{k+1}\dots i_d}^{\beta_k}$  \Comment{$n_k \times r_{k-1} \times (r_k + \ell_k)$}
        \State $(\mathcal{Y}_k)_{i_k\alpha_{k-1}\alpha_k} \leftarrow (\mathcal{Z}_k)_{i_k \alpha_{k-1}\beta_k} (\mathbf{S}_k^\dagger)_{\beta_k \alpha_k}$ \Comment{$n_k \times r_{k-1} \times r_k$}
      \EndFor
      \State \Return $(\mathcal{Y}_1,\dots,\mathcal{Y}_d)$
    \EndFunction
  \end{algorithmic}
\end{algorithm}

\subsection{Application of the {\name} sketch} \label{sec:ttpr_application}

The key insight in the application of the Gaussian TT sketch -- and by extension of {\name} -- is its natural implementation as a recursive sequence of contractions. Given $\bOmega_\mathtt{GTT} = (\cG_1 \bowtie \cdots \bowtie \cG_d )^{\leq 1}$, with uniform ranks $R$ for simplicity, and the vector $\bx$ with TT representation $\cX= \cX_1 \bowtie \cdots \bowtie \cX_d$ and ranks $s_1, \dots, s_{d-1}$, we can formulate $\bOmega_\mathtt{GTT} \bx$ in terms of matrix unfoldings as follows: 

Starting from the last cores, we define $\bW_d :=  \cG_d^{\leq 1} \bigl [ \cX_d^{\leq 1} \bigr ]^T \in \bbF^{R \times s_{d-1}}$, then iterate $\bW_{k-1} := \cG_{k-1}^{\leq 1} \bigl [ (\bW_k \cX_{k-1}^{\leq 1})^{\leq 2}  \bigr ]^T \in \bbF^{R \times s_{k-2}}$ for $k = d, \dots, 2$, and conclude with
\[
    \bOmega_\mathtt{GTT} \bx = \bw_1 := \cG_1^{\leq 1} \bigl [ ( \bW_2 \cX_1^{\leq 1})^{\leq 2} \bigr ]^T \qquad \in \bbF^R.
\]
The corresponding {\name} sketch is then obtained by vertically stacking independent realizations:
\[
    \bW_k = \begin{bmatrix}
        \bW_k^{(1)} \\ \vdots \\ \bW_k^{(P)} 
    \end{bmatrix} \qquad \in \bbF^{PR \times s_{k-1}}, \qquad k \in [d].
\]
This structure inherently enables parallel execution and should seamlessly extend to adaptive algorithms where target ranks are not known a priori and the sketch is incrementally enlarged, analogous to the Khatri-Rao formulation presented in \cite{daas2025adaptive}, although this is left for future work.
Each intermediate matrix $\bW_k$ for $k = 2,\dots,d$ is itself a {\name} sketch of the partial block tensor train obtained by contracting only cores to the right of mode $k$:
\[
    \bW_k = (\cG_k \bowtie \cdots \bowtie \cG_d)^{\leq 1} \bigl [ (\cX_k \bowtie \cdots \bowtie \cX_d)^{\leq 1} \bigr ]^T,
\]
using the notation from \Cref{rem:blocksparseterminology}. 
Therefore the computation captures, at no additional cost, the row space of each partial unfolding $\bX_k^{>} := (\cX_k \bowtie \cdots \bowtie \cX_d)^{\leq 1}$ for $k \in [d]$ in addition to $\bOmega_\mathtt{\nameAbbrv} \bx$ being sketched using matrices $\bOmega^>_{\mathtt{\nameAbbrv},k} := (\cG_k \bowtie \cdots \bowtie \cG_d)^{\leq 1} \in \bbF^{PR \times n_k \dots n_d}$. We note that these matrices are themselves {\name} sketches and inherit their theoretical guarantees (\Cref{thm:OSE,thm:OSI}), although they are clearly not independent of one another.

The intermediate quantities $\{ \bW_k \}_{2 \leq k \leq d}$, called partial sketches, are fundamental to the randomized rounding algorithms reviewed in \cref{sec:ttroundalgorithms}, and thus we formalize the recursive process of getting these partial contractions in \Cref{alg:partialcontractions}.

\begin{algorithm}[ht]
  \caption{Partial Contractions with {\name}}
  \label{alg:partialcontractions}
  \begin{algorithmic}[1]
  \Require $(\cX_1,\dots,\cX_d)$ TT representation of $\bx$ with initial TT-ranks $(\chi_0,\dots,\chi_d)$, {\name} sketch given by $(\cG^{(j,k)})_{j \in [P], k \in [d]}$ with parameter $P$ and block ranks $(r_{0},\dots,r_d)$.
  \Ensure Partial sketches $(\bW_1, \dots, \bW_d)$ with $\bW_k = \bOmega^>_{\mathtt{\nameAbbrv},k} (\bX^>_k)^T \in \bbF^{P r_{k-1} \times r_{k-1}}$.
  \Statex
  \Function{PartialContractions}{$(\cX_1,\dots,\cX_d),(\cG_1,\dots,\cG_d)$}
    \For{$j=1$ to $P$} \Comment{Embarrassingly parallel}
    \State $(\bW_d^{(j)})_{\beta_{d-1}}^{\alpha_{d-1}} \leftarrow \sum_{i_d} (\cG^{(j,d)})_{\beta_{d-1}i_d} (\cX_d)_{\alpha_{d-1}i_d}$ 
    \For{$k=d-1$ to $1$}
        \State $(\bW^{(j)}_k)_{\beta_{k-1}}^{\alpha_{k-1}} \leftarrow  \sum_{i_k,\alpha_k}   \left ( \sum_{\beta_k}  (\cG^{(j,k)})_{\beta_{k-1}i_k\beta_k}(\bW^{(j)}_{k+1})_{\beta_{k}}^{\alpha_{k}}\right ) (\cX_k)_{\alpha_{k-1}i_k \alpha_k}$
    \EndFor
    \EndFor
    \For{$k=1$ to $d$} \Comment{It can be incorporated above}
      \State $\bW_k \leftarrow \textsc{VerticalStack}(\bW_k^{(1)}, \dots, \bW_k^{(P)})$
    \EndFor
    \State \Return $(\bW_1, \bW_2, \dots, \bW_d)$.
  \EndFunction
  \end{algorithmic}
\end{algorithm} 
In general, the cost of \Cref{alg:partialcontractions} is $\cO \bigl ( d  n P R \chi (R + \chi) \bigr )$, given the maximum mode dimension $n = \max_{k \in [d]} n_k$; TT-ranks $\chi = \max_{k \in [d+1]} \chi_k $ for the input vector, and $R = \max_{j \in [d+1]} r_k$ for the sketch. As detailed below in \Cref{sec:ttpr_application}, whenever the cores of the input tensor have additional structure such as block-sparse when considering a linear combination of tensor trains, Kronecker for the Hadamard product of tensors, or as a contraction of two cores in the case of a matrix-vector product, it is not necessary to form explicitely the cores $\cX_k$. 
Instead the tensor contractions $\cG_k^{\leq 1} \bigl [ ( \bW_{k+1} \cX_k^{\leq 1})^{\leq 2} \bigr ]^T$ should be computed efficiently by exploiting this structure, as already noted in~\cite{aldaas2023randomized}, Section 3.4.

%
We briefly review the way to efficiently perform the tensor contractions for the structured cases mentioned previously. 
\subsubsection{Linear combinations}\label{sec:implementLinearCombination}
A linear combination $\bs = \sum_{j=1}^J \alpha_j \by_j$, where each vector $\by_j$ can be expressed in TT format as $\cY_j = \cC_{1,j} \bowtie \cdots \bowtie \cC_{d,j}$ with maximum TT-rank $\chi$, admits a TT decomposition $\cS = \cS_1 \bowtie \cdots \bowtie \cS_d$ defined by
\begin{align*}
    &\cS_1[i_1] = \begin{bmatrix}
        \cC_{1,1}[i_1]  \dots \cC_{1,J}[i_1] 
    \end{bmatrix} \text{ for } i_1 \in [n_1]; \\
    \cS_k[i_k]  = \begin{bmatrix}
        \cC_{k,1}[i_k]  \\ & \ddots \\ & & \cC_{k,J}[i_k] 
    \end{bmatrix}, \  & i_k \in [n_k],\ k = 2 \dots d-1;
    \quad \cS_d[i_d]  = \begin{bmatrix}
        \alpha_1 \cC_{d,1}[i_d]  \\ \vdots \\ \alpha_J \cC_{d,J}[i_d] 
    \end{bmatrix}, \ i_d \in [n_d].
\end{align*}
Explicitly assembling these block-sparse cores of size $J \chi \times n \times J\chi$ and then applying \Cref{alg:partialcontractions} is clearly an ineffective way to compute the partial contractions. Indeed, it can be readily seen~\cite{aldaas2023randomized} that the partial sketches $\bW_k = \bOmega^>_{\mathtt{\nameAbbrv},k} (\bS_k^>)^T$ for $k \geq 2$ can be formed by vertically stacking the independent partial sketches $\bW_{k,j} = \bOmega^>_{\mathtt{\nameAbbrv},k} (\bS^{> }_{k,j})^T$ for $j \in [J]$, obtained by applying \Cref{alg:partialcontractions} independently to each summand TT with the sketch cores:
\[
    \bW_k = \begin{bmatrix}
        \alpha_1 \bW_{k,1} \\ \vdots \\ \alpha_j \bW_{k,J}
    \end{bmatrix} \text{ for } k = d, \dots, 2 \quad \text{ and }\quad  \bOmega_\mathtt{\nameAbbrv} \bs = \bw_1 = \sum_{j=1}^J \alpha_j \bw_{1,j} = \sum_{j=1}^J \alpha_j \bOmega_\mathtt{\nameAbbrv} \by_j,
\]
with a total cost $\cO(JdnPR\chi(R+\chi))$.
Hence, a factor $J$ is saved in both memory and computational cost by not assembling the tensor train of rank $J\chi$ corresponding to the linear combination. 

\subsubsection{Matrix-vector product}\label{sec:implementMatVec}
For the matrix-vector product, the same principle applies. In order to sketch $\bH \by$, where the vector $\by$ is in TT format, and the matrix $\bH$ is the representation of a Tensor Train Operator (TTO), \emph{i.e.}\ a TT-structured linear map acting on TT vectors, there is no need to assemble the corresponding tensor train with multiplied ranks. Instead, the partial contractions computing the sketch can be computed core by core in optimal order, contracting the tensor network from right to left (see \Cref{{subfig:mpo_mps_sketch2}}) to compute partial sketches $\bW_k$, $k = d,\dots,1$, in order to reduce the overall cost. This process is explained in detail in~\cite{Camano_Epperly_Tropp_2025} and represented diagrammatically in \Cref{fig:sketch_mat_vec}.

In the first step, the contraction costs $\mathcal{O}(n P R r_H (\chi+n))$ algebraic operations. The following one costs $\mathcal{O}(n PR r_H\chi (nr_H+R+\chi))$ algebraic operations. Proceeding recursively, the tensor contraction of $\bOmega_\mathtt{\nameAbbrv} \mathbf{H} \mathbf{y}$ costs in total $\cO(dnPRr_H \chi (n r_H +R+\chi))$ algebraic operations. This is significantly better than assembling the tensor train of ranks $R_H \chi$ representing $\mathbf{H} \mathbf{y}$ at a cost of $\cO(d \chi^2 n^2 r_H^2)$, then applying \Cref{alg:partialcontractions} with a cost of $\cO\bigl (d n P R r_H \chi ( R + r_H \chi) \bigr )$.

\begin{figure}
    \centering
    \begin{subfigure}[b]{0.3\textwidth}
        \centering
        \includegraphics[width=\textwidth]{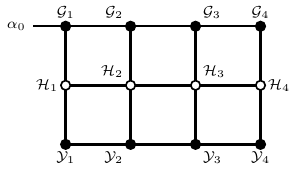}
        \caption{Tensor network representation of $\bOmega_\mathtt{\nameAbbrv}\mathbf{H}\mathbf{y}$}
        \label{subfig:mpo_mps_sketch1}
    \end{subfigure}
    \begin{subfigure}[b]{0.3\textwidth}
        \centering
        \includegraphics[width=0.95\textwidth]{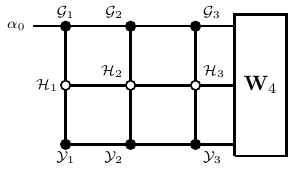}
        \caption{First contraction\\\phantom{blabla}}
        \label{subfig:mpo_mps_sketch2}
    \end{subfigure}
    \begin{subfigure}[b]{0.3\textwidth}
        \centering
        \includegraphics[width=0.6\textwidth]{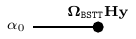}
        \caption{Final result\\\phantom{blabla}}
        \label{subfig:mpo_mps_sketch_final}
    \end{subfigure}
    \caption{Graphical representation of the sketching of the matrix-vector product $\mathbf{H}\mathbf{y}$. 
    The matrix $\mathbf{H}$ has a TT representation $(\cH_k)_{k \in [d]}$ of ranks $r_H$, $\by$ has a TT representation $(\cY_k)_{k \in [d]}$ of ranks $\chi$, $\bOmega_\mathtt{\nameAbbrv} \in \mathbb{F}^{PR \times N}$. 
    The contractions are done from right to left, top to bottom. 
    }
    \label{fig:sketch_mat_vec}
\end{figure}

\subsubsection{Hadamard product}\label{sec:implementHadamard}
Finally, the Hadamard product $\cP = \cY_1 \bullet \dots \bullet \cY_J$, where each factor $\cY_j = \cC_{1,j} \bowtie \cdots \bowtie \cC_{d,j}$ is a tensor train with maximum TT-rank $\chi$, admits a tensor train representation $\cP = \cP_1 \bowtie \cdots \bowtie \cP_d$ defined by
\[
    \cP_k[i_k] = \cC_{k,1}[i_k]  \otimes \dots \otimes \cC_{k,J}[i_k] \qquad  \text{ for }   i_k \in [n_k],\ k \in [d].
\]
As before, explicitly forming these structured blocks of size $\chi^J \times n \times \chi^J$ and then invoking \Cref{alg:partialcontractions} with a cost of $\cO\bigl(dnPR \chi^J (R + \chi^J) \bigr )$ is not efficient. Instead, we notice that for each $i_k \in [n_k]$, we can decompose as a sequence of multiplications the matrix-matrix product with the Kronecker-structured slice:
\[
    \bW_{k+1} (\cC_{k,1}[i_k] \otimes \cdots \cC_{k,J}[i_k])  \ =\  \bW_{k+1} (\cC_{k,1}[i_k] \otimes \bI_{\chi^{J-1}}) \cdots (\bI_{\chi^{J-1}} \otimes \cC_{k,J}[i_k]).
\]
Going left to right, each matrix-matrix product in this sequence can be seen as contraction over a single index of dimension $\chi$, with the overall cost of this operation batched over $i_k \in [n_k]$ being $\cO(nJ \chi^{J+1} )$ (see \Cref{fig:hadamard_sketch}). The last step is to perform the product $\cG_k^{\leq 1} \bigl (\bW_{k+1} (\cC_{k,1} \otimes \cdots \cC_{k,J}) \bigr )^{\leq 2}$ with cost $\cO(nR^2 \chi^J)$, for an overall cost of the partial contractions of $\cO(dnPR \chi^J (J \chi + R ))$.

\begin{remark}
 More elaborate randomized schemes exist to further reduce the cost of sketching Hadamard products, see~\cite{sun2024hatt,Meng_Khoo_Li_Stoudenmire_2026}.
 A thorough comparison with these approaches is future work, although the approach presented here has the advantage that it is easy to implement and trivially generalizes to any number of factors.
\end{remark}


%

\begin{figure}[ht]
    \centering
    \begin{subfigure}[b]{0.45\textwidth}
        \centering
        \includegraphics[width=\textwidth]{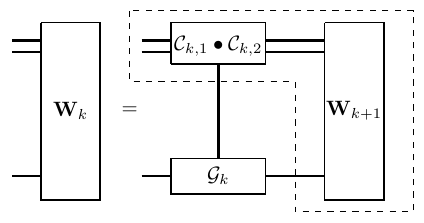}
        \caption{Partial contraction $\mathbf{W}_k$\\\phantom{blabla}}
    \end{subfigure}
    \begin{subfigure}[b]{0.45\textwidth}
        \centering
        \includegraphics[width=0.7\textwidth]{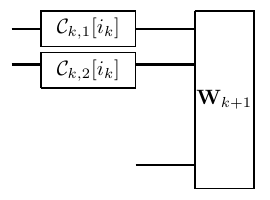}
        \caption{Diagrammatic representation of the critical tensor contraction}
    \end{subfigure}
    \caption{Graphical representation of the partial contraction between a Hadamard product $\mathcal{Y}_1 \bullet \mathcal{Y}_2$ and a TT sketch. 
    The figure on the right highlights the dashed section on the left, illustrating the benefit of not assembling the intermediate TT core $\mathcal{C}_{k,1} \bullet \mathcal{C}_{k,2}$.}
    \label{fig:hadamard_sketch}
\end{figure}

\begin{figure}[bh!]
    \centering
    \includegraphics[width=\textwidth]{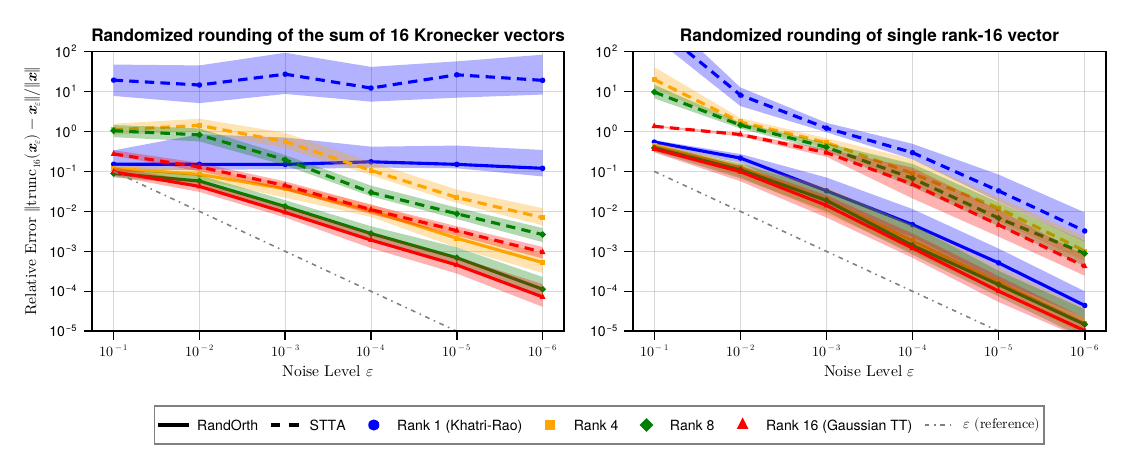}
    \caption{\textbf{Numerical benchmark of randomized TT-rounding on a synthetic test case.} Empirical evaluation of \Cref{alg:tt_randrounding,alg:stta} applied to a perturbed input tensor $\cX = \cX_s + \varepsilon \cX_{10}$, where all constituent vectors reside in the high-dimensional space $\mathbb{R}^{4^{50}}$. The baseline signal $\cX_s$ corresponds to either the sum of 16 Gaussian Kronecker TT vectors (left) or a single rank-16 Gaussian TT vector (right), which is then perturbed by the rank-10 Gaussian TT vector $\cX_{10}$ at variable noise levels $\varepsilon = 10^{-s}$ for $s \in [6]$. We employ a fixed embedding dimension $PR=16$ (no oversampling) across block ranks $R=\{1,4,8,16\}$. Markers track the median of 10 independent trials, with shaded regions denoting the interquartile range (25th to 75th percentiles).}
    \label{fig:TT_Rounding}
\end{figure}

\subsection{Numerical Experiments}

In this section, we present numerical tests of the {\name} sketch first in synthetic examples, then in the context of TT-rounding for compressing Hadamard products of tensors, and finally showcase a tensor train-adapted sketched Rayleigh-Ritz eigensolver applied to an example from quantum chemistry. All results presented in this section are implemented using real arithmetic ($\bbF = \bbR$).

\subsubsection{A Synthetic Example}

We apply both \Cref{alg:tt_randrounding,alg:stta} with orthogonal {\name} sketches to compress a synthetic perturbed low-rank tensor $\cX = \cX_s + \varepsilon \cX_{10}$, where $\cX_{10}$ is a rank-10 Gaussian TT perturbation with variable noise level $\varepsilon = 10^{-s}$. We consider two structures for the dominant term $\cX_s$: a normalized sum of 16 rank-1 Gaussian Kronecker TT vectors, and a single rank-16 Gaussian TT vector. In both cases we fix the embedding dimension $PR = 16$ and vary the block rank $R \in \{1, 4, 8, 16\}$; numerical results are reported on \Cref{fig:TT_Rounding}. In all cases, increasing the rank $R$ increases the accuracy of the randomized rounding algorithms; in particular, in the first case where summands of the input tensor $\cX_s$ have Kronecker structure, \Cref{alg:tt_randrounding} (Randomize-then-Orthogonalize) implemented with the Khatri-Rao sketch ($R=1$) struggles to achieve even one digit of accuracy. By comparison, \Cref{alg:stta} (Streaming Tensor Train Approximation) fails completely. 
For both algorithms, increasing the rank of the sketch vectors modestly restores accuracy to the rounding operation. By comparison, all randomized approaches perform much better when the input tensor $\cX_s$ is a random rank-16 tensor train, and in particular rounding approaches depending on the Khatri-Rao sketch ($R=1$) have comparable accuracy (losing only about an order of magnitude), albeit with increased variance. For larger ranks such as $R=16$, the algorithms perform quite similarly in either case.

Our analysis sheds some light on this behavior: by \Cref{thm:OSI}, increasing $R$ at fixed $PR$ reduces the subspace coherence constant $C_\bQ(R)$, improving the injectivity parameter and, via \Cref{thm:rand_round_simple}, the rounding error bound. The two choices of $\cX_s$ probe how the TT-rank structure of the dominant component interacts with this block-rank dependence, with the worst case $C_\bQ(R) = (1+2/R)^d$ achieved when $\cX_s$ belongs to a subspace including vectors with Kronecker structure.

\subsubsection{Hadamard Product}
The Quantized Tensor Train (QTT) format~\cite{Oseledets_2010} provides a tensor train representation of functions discretized on dyadic grids, where the tensor order $d$ corresponds to the resolution $2^d$ of the grid. Many classes of functions: polynomials, exponentials, and trigonometric series among them, admit QTT representations with ranks bounded by small constants independent of $d$~\cite{Oseledets_2013,Lindsey_2023}, making this format well-suited to high-resolution function approximation.

Despite this representational efficiency, performing pointwise algebraic operations within the QTT format remains a significant computational bottleneck. We apply \Cref{alg:tt_randrounding} with Orthogonal {\name} sketch to compress a three-term Hadamard product, and report the accuracy and computation time as functions of the target rank.
\begin{figure}[b!]
    \centering
    \includegraphics[width=0.9\linewidth]{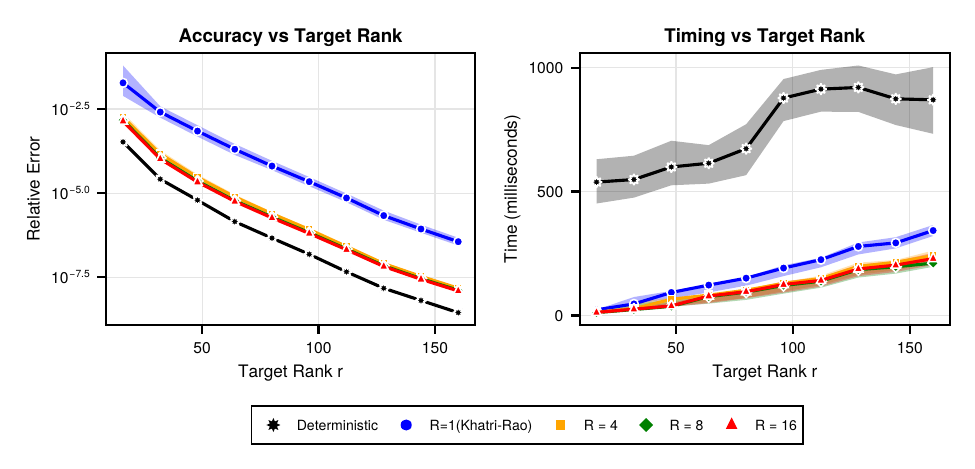}
    \caption{\textbf{Randomized rounding of Hadamard products of QTT functions.} Empirical performance of Orthogonal {\name} sketches in the randomized rounding of a three-term QTT Hadamard product. The sketches are applied via \Cref{alg:tt_randrounding} to the functions defined in \eqref{eq:QTT}, evaluated on the discretized domain $[0,1]^3$ over a $(2^{20})^3$ grid. We compare embedding dimensions $PR=r$ across varying block ranks $R\in\{1,4,8,16\}$, where markers indicate the median across 1000 independent trials, and shaded regions denote the interquartile range (25th to 75th percentiles).}

    \label{fig:hadamard}
\end{figure}
Namely, we compute the product $f_1 f_2 f_3$ of QTT discretizations of the functions
\begin{equation} \label{eq:QTT}
    \begin{aligned}
        f_1(x,y,z) &= \frac{1}{10}e^{-2(x^2 + y^2 + z^2)-(x+y-2z)^2} + e^x \\
        f_2(x,y,z) &= \frac{1}{10}\cos(2^{16}(x+y-2z)) + e^{-x} \\
        f_3(x,y,z) &= \frac{1}{10}\cos\left (\frac{2^{14}}{\sqrt{5}}(x+y-2z) \right) + 1,
    \end{aligned}
\end{equation}
 which are interpolated by the 2-site tensor cross interpolation (TCI) algorithm~\cite{Learning_with_TT_cross}, also known as TT-cross~\cite{QTT,TT_cross}. Although the QTT discretizations of $f_2$ and $f_3$ possess exact closed-form representations with small ranks, their pointwise product $f$ exhibits exact TT-ranks exceeding $800$, even though it admits highly accurate approximations at much lower ranks, and we demonstrate the effectiveness of the approach presented in \Cref{sec:implementHadamard} in \Cref{fig:hadamard}. \Cref{alg:tt_randrounding} yields computational speedups of up to two orders of magnitude over the deterministic approach, \Cref{alg:tt_rounding}, across all block rank configurations. Notably, the Khatri-Rao sketch ($R=1$) exhibits somewhat slower execution times because its formulation cannot leverage the highly optimized BLAS3 matrix-matrix operations available on modern hardware, although a more careful and specialized implementation may eliminate this gap. 

Beyond execution speed, the empirical accuracy gap strictly corroborates the overwhelming orthogonality phenomenon. Because the dominant component of the target function $f$ is the constant factor $1$, which is representable by a rank-1 tensor train on any grid, the Khatri-Rao sketch is far less accurate, as discussed in \Cref{sec:KR}. Increasing the block ranks allows to overcome this geometric bottleneck and gain a full order of magnitude in accuracy without oversampling, with choices $R=4,16,32$ exhibiting comparable behavior on this example.
 
\subsubsection{Ground-state energy of a quantum chemistry Hamiltonian}

In this experiment, we apply the Orthogonal {\name} sketch to the computation of the electronic ground-state energy of the lithium hydride (LiH) molecule in the \textsc{ccpvdz} basis. 

The electronic Hamiltonian is written in second quantization~\cite{Szalay_Pfeffer_Murg_Barcza_Verstraete_Schneider_Legeza_2015}, where it is naturally represented as a TTO. For the LiH molecule, the Hamiltonian corresponds to a symmetric matrix of size $2^{38} \times 2^{38}$. Starting from a Hartree-Fock initial guess $\mathcal{V}_0$, which has TT rank 1, we construct a Krylov basis by repeatedly applying the TTO and rounding the result via \Cref{alg:tt_randrounding} to maintain a bounded TT rank. The Orthogonal {\name} sketch plays the dual role of providing the random projections needed for randomized rounding of the new basis vector $\bb_j = \bH \bb_{j-1} - \sum_{s<j} \alpha_s \bb_s$ at each step, and also enabling a sketched Rayleigh-Ritz approach following~\cite{nakatsukasa2024fast}.
\begin{algorithm}[ht!]
  \caption{Sketched Rayleigh-Ritz eigensolver, adapted from~\cite{nakatsukasa2024fast}}
  \label{alg:sketched_rr}
  \begin{algorithmic}[1]
  \Require Matrix $\bH = \cH_1 \bowtie \cdots \bowtie \cH_d  \in \bbF^{N \times N}$ in TTO format, initial vector $\bv$ in TT format $\cV= \cV_1 \bowtie \cdots \bowtie \cV_d \in \bbF^N$, sketch parameters $(P,R)$, basis dimension $d$, partial orthogonalization truncation parameter $k$, number of steps $m$, target ranks $(r_0,\dots,r_d)$
  \Ensure Approximate eigenpairs $(\lambda_j, \bv_j)_{1 \leq j \leq m}$ estimating the $m$ lowest eigenpairs of $\cH$ and estimated residual norms
  \State
  \Function{Sketched-Rayleigh-Ritz}{}
    \State Draw sketch $\bOmega_\mathtt{\nameAbbrv} = (\cG_1 \bowtie \cdots \bowtie \cG_d )^{\leq 1}\in \ \bbF^{PR \times N}$ with parameters $(P,R)$
    \State $\bOmega \bv,\ (\bW_k^{(1)})_{2 \leq k \leq d} = \textsc{PartialContractions}((\cV_2,\dots,\cV_d),(\cG_2,\dots,\cG_d))$
    \State $\bb_1 = \cB^{(1)}_1 \bowtie \cdots \bowtie \cB^{(1)}_d \leftarrow \bv / \Vert \bOmega\bv\Vert_2$ \Comment{Sketch-normalized starting vector}
    \State $\bOmega \bb_1 \leftarrow \bOmega \bv / \Vert \bOmega\bv\Vert_2$
    \For{$j = 2$ to $m$}
      \State $(\bOmega \bH \bb_{j-1}, \bZ_1, \dots, \bZ_d) \leftarrow $
      \State \hspace{.5in} \textsc{PartialContractions}($(\cH_1 \,$\rotatebox{90}{$\bowtie$}$\,\cB^{(j-1)}_1,\dots,\cH_d\,$\rotatebox{90}{$\bowtie$}$\,\cB^{(j-1)}_d)$, $(\cG_1,\dots,\cG_d))$\footnotemark[1]
      \State $\balpha  \leftarrow [ \bOmega \bb_{j-k} \ldots \bOmega \bb_{j-1}]^\dagger (\bOmega \bH \bb_{j-1})$ \\
      \Comment{Sketched truncated Gram-Schmidt coefficients, $\bb_{-s} = 0$ for $s \geq 0$}
      \State $\bv_j \leftarrow \textsc{randomized-TT-rounding}(\bH \bb_{j-1} - \sum_{s=j-k}^{j-1} \alpha_s \bb_s)$ \\
      \Comment{Partial sketches are already computed}
      \State $\widetilde{\bv}_j = \cV^{(j)}_1\bowtie \cdots\bowtie\cV^{(j)}_d \leftarrow \textsc{TT-rounding}(\bv_j)$ \Comment{Enforce desired ranks}
      \State $\bOmega \widetilde{\bv}_j ,\ (\bW_k^{(j)})_{2 \leq k \leq d} \leftarrow \textsc{PartialContractions}((\cV^{(j)}_2,\dots,\cV^{(j)}_d),(\cG_2,\dots,\cG_d))$
      \State $\bb_j = \cB^{(j)}_1 \bowtie \cdots \bowtie \cB^{(j)}_d \leftarrow \widetilde{\bv}_j / \Vert \bOmega\widetilde{\bv}_j \Vert_2$ \Comment{Sketch-normalized basis vector}
      \State $\bOmega \bb_j \leftarrow \bOmega \widetilde{\bv}_j / \Vert \bOmega\widetilde{\bv}_j\Vert_2$
    \EndFor
    \State $\bC \leftarrow [ \bOmega \bb_1, \ldots, \bOmega \bb_{m}]$, \quad $\bD \leftarrow \bigl[\bOmega \bH \bb_1, \ldots, \bOmega \bH \bb_m\bigr]$ \Comment{$PR \times m$ sketched matrices}
    \State $\bC^\dagger \bD \by_j = \lambda_j \by_j$ for $j \in [m]$ \Comment{Solve Ritz eigenvalue problem}
    \State $r_j \leftarrow \Vert \bD \by_j - \lambda_j \bC \by_j \Vert_2$ \Comment{Residual estimates}
    \State $\bx_j \leftarrow \sum_{j=1}^m y_j \bb_j$ \Comment{Ritz vectors}
    \State \Return $(\lambda_j, \bx_j, r_j)_{j=1}^m$
  \EndFunction
  \end{algorithmic}
\end{algorithm}
\footnotetext[1]{For $\mathcal{A} \in \mathbb{F}^{R_1 \times n \times n \times R_2}$, $\mathcal{B} \in \mathbb{F}^{r_1 \times n \times r_2}$, we define the contraction $\mathcal{A}\,$\rotatebox{90}{$\bowtie$}$\,\mathcal{B} \in \mathbb{F}^{R_1r_1 \times n \times R_2r_2}$ as:
\[
   (\mathcal{A}\,\text{\rotatebox{90}{$\bowtie$}}\,\mathcal{B})_{\alpha_1\beta_1,i,\alpha_2\beta_2} = \sum_{j=1}^n \mathcal{A}_{\alpha_1,i,j,\alpha_2} \mathcal{B}_{\beta_1,j,\beta_2}, \qquad \text{for} \, \alpha_1 \in [R_1], \alpha_2 \in [R_2], \beta_1 \in [r_1], \beta_2 \in [r_2], i \in [n].
\]
}
Notably, due to the large rank $\cO(d^2)$ of the quantum chemistry Hamiltonian~\cite{Bachmayr_Gotte_Pfeffer_2022}, randomized TT rounding is crucial for the efficient evaluation of this linear combination.

Given the sketched basis matrix $\bC = [\bOmega_{O\nameAbbrv}\,\bb_1 \cdots \bOmega_{O\nameAbbrv}\,\bb_m]$ and the sketched action \linebreak $\bD = [\bOmega_{O\nameAbbrv}\,\bH \bb_1 \cdots \bOmega_{O\nameAbbrv}\,\bH\bb_m]$, the Ritz matrix is obtained as the least-squares solution $\bM = \bC^\dagger\bD$ via a QR decomposition of $\bC$, and then diagonalized to yield the Ritz pairs and in particular an approximation $\lambda_1$ to the lowest eigenvalue of $\bH$ and corresponding approximate eigenvector $\widehat{\bx} = \sum_{j=1}^m y_j \bb_j$ with $\bM \by = \lambda_1 \by$.  In order to control the conditioning of the basis, we additionally implement a sketched Gram-Schmidt algorithm, meaning that our basis vectors are (partially) \textit{sketch-orthogonal}, $(\bOmega_{O\nameAbbrv} \bb_i)^* (\bOmega_{O\nameAbbrv} \bb_j) \approx \delta_{ij}$, with inaccurate evaluations due to the TT-rounding steps necessary to keep ranks under control. This ensures good conditioning of the basis when $\bOmega_{O\nameAbbrv}$ is an OSE for the (approximate) Krylov subspace. The same sketch is reused to (1) round matrix-vector products and linear combinations, (2) compute the coefficients needed to sketch-orthogonalize the Krylov basis, and (3) form the sketched problem. The resulting algorithm is summarized in \Cref{alg:sketched_rr}.

On this quantum chemistry example, we observe that the basis remains well-conditioned with $\cO(1)$ condition number, however convergence tends to stagnate after rapid progress is observed over the first few iterations. The convergence of Rayleigh-Ritz is known to be complex in general~\cite{nakatsukasa2024fast} and we leave to future work a rigorous study of \Cref{alg:sketched_rr} and its convergence properties, taking into account the inexact basis construction induced by the rounding operations necessary to avoid explosion of the TT-ranks. In this work, we simply restart the algorithm every 10 iterations to avoid stagnation, which also allows to grow the target ranks along outer iterations.
\begin{figure}[ht!]
  \centering
  \includegraphics[width=\linewidth]{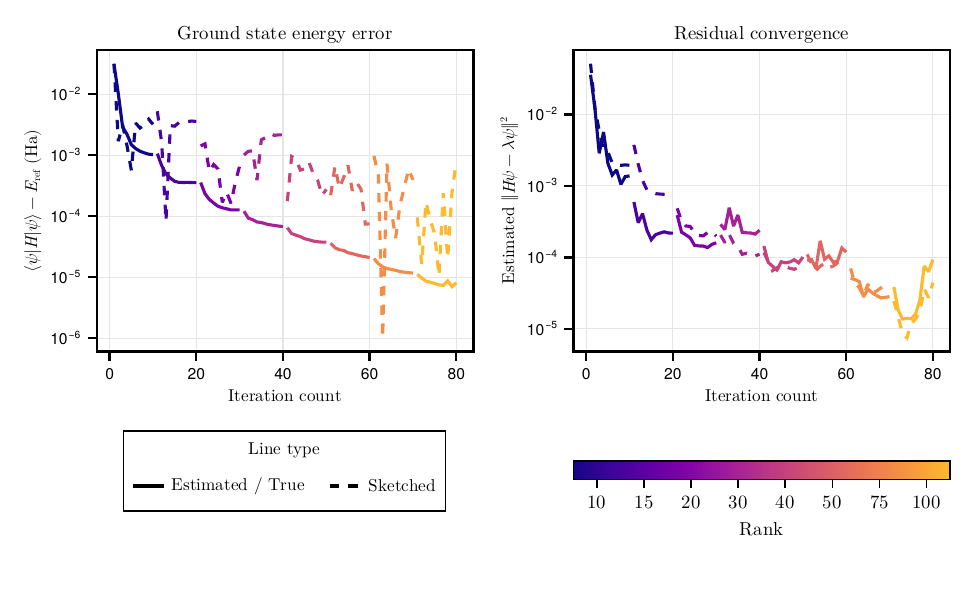}
  \caption{\textbf{Sketched Rayleigh-Ritz convergence for the LiH ground state.}
    Each panel shows convergence as a function of Krylov iterations with restart every 10 iterations and increasing basis TT-ranks, with lighter shading indicating higher rank.
    \textbf{Left:} Ground-state energy error relative to a reference energy, $\vert \lambda_1 - E_{\mathrm{ref}}\vert$ ("Sketched") or exact Rayleigh quotient $\langle \widehat{\bx} | \bH | \widehat{\bx} \rangle$ ("True").
    \textbf{Right:} Squared residual $\|\bOmega (\bH \widehat{\bx} - \lambda \widehat{\bx})\|^2$ ("Sketched") or randomized estimate of the true residual $\| \bH \widehat{\bx} - \lambda_1 \widehat{\bx} \|^2$ ("Estimated").
    }
  \label{fig:LiH_convergence}
\end{figure}

Convergence results, relative to the reference FCI energy computed using PySCF~\cite{sun2020recent}, are reported on \Cref{fig:LiH_convergence}. For each quantity of interest, namely eigenvalue error and eigenvector residual, we compare the quantities obtained from the sketched problem at each iteration, $\lambda_1 - E_\mathrm{ref}$ and $\| \bD \by - \lambda_1 \bC \by \|^2_2$, and the true residuals $ \langle \widehat{\bx} | \bH | \widehat{\bx} \rangle - E_{ref}$ and $\| \bH \widehat{\bx} - \lambda_1 \widehat{\bx} \|^2_2$. Our observation is that the approximate eigenvalue $\lambda_1$ is quite noisy and inaccurate, however evaluation of the true Rayleigh quotient shows a fairly steady convergence behavior, with 5 digits of accuracy obtained within 80 iterations. On the other hand, the sketched residual remains a good approximation of the true residual, as expected from the OSE / OSI guarantees of the {\name} sketch.

\section{Oblivious Subspace Embedding (OSE) property for {\name}}\label{sec:OSE}

This section begins the technical development of this paper, where we establish OSE and OSI guarantees by rigorous probabilistic techniques. To provide a pedagogical exposition, we provide full statements for the key results needed for the analysis of tensor-formatted sketch matrices such as $\bOmega_\mathtt{GTT}$ and $\bOmega_{\mathtt{\nameAbbrv}}$.

\subsection{Prerequisites: (Strong) JLM Property}
\label{sec:johnson_lindenstrauss_moment_property}

A powerful framework for establishing OSE guarantees is the \textit{Strong Johnson-Lindenstrauss Moment Property}~\cite{strong_jlm,AhleKKPVWZ19-arxiv}. By enforcing a tight control of the moments of the estimator, rather than of its tail behavior, one can often derive bounds in a more straightforward manner, particularly in structured settings where direct probability bounds are difficult to compute.

\begin{definition}[Strong JLM Property]\label{def:SJLM}
    A random matrix $\bOmega \in\bbF^{r \times n}$ has the Strong $(\varepsilon, \delta)$-JLM (Strong JLM) Property if for any unit vector $\bx \in \bbF^{n}$ and every integer $2 \leq t \leq \log(1/\delta)$, the following holds:
    \[
        \left\| \| \bOmega \bx\|_2^2 - 1\right\|_{L^t} \leq \frac{\varepsilon}{e} \sqrt{\frac{t}{\log(1/\delta)}} \qquad \text{and} \qquad \mathbb{E}[\| \bOmega \bx\|_2^2]=1
    \]
    where $\|Z\|_{L^t} := (\mathbb{E}[|Z|^t])^{1/t}$.
\end{definition}

\begin{remark}
    The Strong JLM property implies a distributional guarantee with the same parameters by Markov's inequality. Indeed, taking $t=\log(1/\delta)$ we have
    \[
        \mathbb{P}\Big\lbrace \big\lvert   \| \bOmega\bx \|_2^2 -1 \big\rvert \geq \varepsilon \big\rbrace \leq \frac{\mathbb{E}\big\lVert \| \bOmega\bx \|_2^2 -1 \big\rVert^t}{\varepsilon^t} \leq e^{-\log(1/\delta)} = \delta.
    \]
\end{remark}

As an example, we provide a fully explicit condition to realize a Strong JLM property using a simple Gaussian sketch, with its corresponding proof found in \Cref{app:sjlm_gaussian}. 
\begin{lemma}
    \label{lem:Gaussian_SJLM}
    Given a unit vector $\bx \in \bbR^n$ and a random Gaussian matrix $\bG \in \bbF^{R \times n}$ with iid $\mathcal{N}_\bbF \left(0,\frac{1}{R}\right)$ entries:
    \[
        \bigl \| \| \bG \bx \|_2^2 - 1 \bigr \|_{L^t} \leq 2 \left ( \sqrt{\frac{t}{R}} + \frac{t}{R} \right ) \qquad \text{  for  } t \geq 2.
    \] 
    In particular, $\bG$ satisfies the Strong $(\varepsilon,\delta)$ JLM property provided that 
    \[
        R \geq 8 \max(e/\varepsilon, 1)^2 \log(1/\delta).
    \]
\end{lemma}
To turn the Strong JLM property into an OSE guarantee, one possible route is the approximate matrix multiplication property, which is a condition evaluating how closely the product of sketched matrices approximates the product of the original ones.
\begin{importedlemma}\label{lem:appmult}(\cite{cohen2016optimal,ahle2019oblivious}, paraphrased)
Provided that $\bOmega$ satisfies 
the Strong $(\varepsilon, \delta)$-JLM property, then it also satisfies the $(\varepsilon, \delta)$-Approximate Matrix Multiplication (AMM) property:

For any matrices $\bA,\bB$ with $N$ rows, the following holds:
    \begin{equation}\label{eq:amm2}
        \P{ \| (\bOmega \bA)^* (\bOmega \bB) - \bA^* \bB \|_F > \varepsilon  \| \bA \|_F \| \bB \|_F }< \delta.
    \end{equation}
\end{importedlemma}
\begin{proof}
    By Lemma 4.2 in~\cite{ahle2019oblivious} we have the moment bound for $t = \log(1/\delta)$:
    \[
        \bigl \| \| (\bOmega \bA)^* (\bOmega \bB) - \bA^* \bB \|_F \bigr \|_{L^t} \leq \varepsilon/e \| \bA \|_F \| \bB \|_F= \varepsilon \delta^{1/t} \| \bA \|_F \| \bB \|_F,
    \]
    hence by the Markov bound:
    \[
        \P{ \| (\bOmega \bA)^* (\bOmega \bB) - \bA^* \bB \|_F > \varepsilon \| \bA \|_F \| \bB \|_F  } < \frac{(\varepsilon \delta^{1/t} \| \bA \|_F \| \bB \|_F)^t}{(\varepsilon \| \bA \|_F \| \bB \|_F)^t} = \delta.
    \]
\end{proof}
Now, to make the link with the OSE property, pick $\bA = \bB = \bQ$ with $r$ orthonormal columns, and by \Cref{eq:amm2}, we have 
\[
    \P{ \| (\bOmega \bQ)^* (\bOmega \bQ) - \bI_r \|_2 > r \varepsilon } \leq \P{ \| (\bOmega \bQ)^* (\bOmega \bQ) - \bI_r \|_F > \varepsilon  \| \bQ \|_F \| \bQ \|_F }< \delta.
\]
A better scaling with respect to the embedding dimension $r$ can be obtained by using a different bound on $\| (\bOmega \bQ)^* (\bOmega \bQ) - \bI_r \|_2$. 
By Theorem~6 in~\cite{cohen2016optimal}, which is based on an $\varepsilon$-net argument, for any $\bQ \in \mathbb{F}^{N \times r}$ with orthonormal columns, there exists a set $X \subset \mathbb{F}^r$ of cardinality $9^r$ of unit norm vectors such that 
\[
    \| (\bOmega \bQ)^* (\bOmega \bQ) - \bI_r \|_2 \leq 2 \sup_{\bx \in X} \big\lvert \| \bOmega\bQ\bx \|_2^2 -1 \big\rvert.
\]
Using this result, if $\bOmega \in \mathbb{F}^{m \times N}$ is a random matrix satisfying a Strong $(\varepsilon/2,\delta/9^r)$-JLM property, we have that for $t = \log(9^r/\delta)$
\begin{align*}
    \mathbb{P}\big\lbrace \| (\bOmega \bQ)^* (\bOmega \bQ) - \bI_r \|_2 \geq \varepsilon \big\rbrace & \leq \frac{\mathbb{E} \big[   \| (\bOmega \bQ)^* (\bOmega \bQ) - \bI_r \|_2^t \big] }{\varepsilon^t}  \leq \frac{2^t}{\varepsilon^t} \mathbb{E}\big[ \sup_{\bx \in X}\big\lvert \| \bOmega\bQ\bx \|^2_2 -1 \big\rvert^t \big] \\
    & \leq \frac{2^t}{\varepsilon^t} \sum_{ \bx \in X }^{ } \mathbb{E}\big[ \big\lvert \| \bOmega\bQ\bx \|_2^2 - 1 \big\rvert^t  \big] \leq \frac{2^t}{\varepsilon^t} \left[ 9^r \left(\frac{\varepsilon}{2e}\right)^t \right] = \delta.
\end{align*}

We summarize both results in the next corollary.
\begin{corollary}\label{cor:sjlm_ose}
    For $\varepsilon \in(0,1]$ and $\delta \in (0, 1/2)$, any random matrix $\bOmega$ with either the Strong $(\varepsilon/2, \delta/9^{r})$ or $(\varepsilon/r, \delta)$-JLM property realizes an $(\varepsilon, \delta,r)$-Oblivious Subspace Embedding.
\end{corollary}

\subsection{Properties of Gaussian Tensor-Train Sketches}
\label{sec:sjlm_ose_gtt}

We now recall some results from~\cite{ahle2019oblivious} which are essential to our analysis.

\begin{importedlemma}[Lemma 4.5 in~\cite{ahle2019oblivious}]\label{lem:sjlm_kron}
    If the random matrix $\bOmega \in \bbF^{t \times n}$ satisfies the Strong $(\varepsilon, \delta)$-JLM property, then for any $s \in \mathbb{N}$, the matrix $(\bI_s \otimes \bOmega) \in \bbF^{ts \times ns}$ also satisfies the Strong $(\varepsilon, \delta)$-JLM property.
\end{importedlemma}

The proof of this lemma relies on the observation that if $\bOmega$ satisfies the Strong JLM property, then for $\bx \in \bbF^{ns}$, 
\[
    (\bI_s \otimes \bOmega) \bx = \begin{bmatrix}
    \bOmega \bx_1 \\ \vdots \\ \bOmega \bx_s,
    \end{bmatrix}, \qquad \text{where } \bx =  \begin{bmatrix}
    \bx_1 \\ \vdots \\ \bx_s
    \end{bmatrix} \quad \text{ with } \bx_i \in \bbF^n \text{ for } i \in [s],
\]
by subadditivity of the norm, $(\bI_s \otimes \bOmega)$ satisfies the Strong JLM property for the same parameters. The proof of the next result is more involved and relies on a decoupling argument.

\begin{importedlemma}[Lemma 4.13 in~\cite{ahle2019oblivious}]
\label{lem:sjlm_product}
    There exists a positive constant $L$ such that, given $\varepsilon, \delta \in (0,1]$, for any independent random matrices $\bOmega_i$ each with the Strong $(\varepsilon/(L \sqrt{ i }), \delta)$-JLM property with appropriate dimensions for $i \in [d]$ , the product $\bOmega_1 \cdots \bOmega_d$ has the Strong $(\varepsilon, \delta)$-JLM property.   
\end{importedlemma}

It is then enough to use the Strong JLM property of a random Gaussian matrix (\Cref{lem:Gaussian_SJLM}) to prove a subspace embedding property for Gaussian Tensor Train sketches. 

\begin{proposition}\label{prop:sjlm_gtt}
    Let $\bOmega_\mathtt{GTT} := \bG_1 (\bI_{n_1} \otimes \bG_2) \cdots (\bI_{n_1 \dots n_{d-1}} \otimes \bG_d)$ be a Gaussian TT sketch \eqref{eq:gtt} with ranks $r_0 = R$, $ r_1, \dots, r_{d-1} \geq 1$, and $r_d=1$, that is $\bG_j \in \bbF^{r_{j-1} \times n_j r_j}$ for $j \in [d]$ with iid $\cN_{\bbF}(0,1/r_{j-1})$ entries, and let $\bx$ be an input vector in $\bbF^N$. Given $\varepsilon, \delta \in (0,1]$ and provided
    \begin{equation} \label{eq:sjlm_gtt}
        R \geq d r_{d-1}, \quad r_j \geq (d-j) r_{d-1} \ \text{for } j \in [d-2], \quad \text{ and } \quad r_{d-1} \geq  8 \bigl ( L e /\varepsilon  \bigr )^2 \log(1/\delta),
    \end{equation}
    with $L \geq 1$ the universal constant from \Cref{lem:sjlm_product},
    the sketch matrix $\bOmega_\mathtt{GTT}$ satisfies the Strong $(\varepsilon, \delta)$-JLM Property.
\end{proposition}
\begin{proof}
    By \Cref{lem:Gaussian_SJLM}, condition~\eqref{eq:sjlm_gtt} ensures that the Gaussian matrices $\bG_j$ satisfy the Strong $(\varepsilon/L \sqrt{d+1-j}, \delta)$-JLM Property and by \Cref{lem:sjlm_kron}, so do matrices $\bM_j = \bI_{n_1 \cdots n_{j-1}} \otimes \bG_j$. Finally, by \Cref{lem:sjlm_product}, the product $\bOmega = \bM_1 \cdots \bM_d$ satisfies the Strong $(\varepsilon, \delta)$-JLM Property. 
\end{proof}
\begin{remark}
    The condition $\varepsilon \leq 1$, while seemingly innocuous, is a genuine restriction as seen in the proof of this result~\cite{AhleKKPVWZ19-arxiv} which limits the applicability of this framework to the regime $R = \cO(d \log(1/\delta))$, and in particular, does not allow to investigate the OSI regime $R = \cO(d)$ with the Strong JLM technique.
\end{remark}

\begin{proof}[Proof of \Cref{thm:GTT-OSE} (OSE with Gaussian TT sketches)]
    By combining \Cref{prop:sjlm_gtt} and \Cref{cor:sjlm_ose}, we have that $\bOmega_\mathtt{GTT}$ satisfies an $(\varepsilon,\delta,r)$-OSE under the rank conditions $R \geq d r_{d-1}  $ and $r_i \geq (d-i) r_{d-1}$ for $i  \in [d-2]$, with $r_{d-1} \geq 32 (Le/\varepsilon)^2 \bigl ( r \log(9) + \log(1/\delta) \bigr )$.
\end{proof}

The final step in our analysis is to translate the OSE/OSI condition into probabilistic guarantees for the randomized SVD. This is achieved with the help of the following:

\begin{importedlemma}[Lemma 4.2 from~\cite{woodruff2014sketching}, paraphrased]\label{lem:rsvdfromose}
Let $\bOmega \in \bbF^{k \times N}$ satisfy an $(\alpha_{\mathtt{osi}},\delta/2,r)$-OSI condition as well as the $(\varepsilon_{\mathtt{amm}},\delta/2)$-AMM property~\eqref{eq:amm2}.
For any fixed matrix $\bA$, the sketched matrix $\bA\bOmega ^*$ has full rank almost surely, and admits the following with probability at least $1-\delta$:
\[
    \| \bA - \bQ \bQ^* \bA  \|_F^2 \leq \Big(1+\frac{{r}\varepsilon_{\mathtt{amm}}^2}{\alpha_{\mathtt{osi}}^2}\Big) \| \bA - \bA_r \|_F^2,
\]
where $\bA_r$ is the best rank-$r$ approximation to $\bA$ and $\bQ \bQ^*$ is the orthogonal projector onto the column space of $ \bA \bOmega^*$.
\end{importedlemma}
\begin{proof}
    Let $\bU,\bSigma,\bV$ be an SVD of $\bA$. 
    Let 
    \[
        \bSigma = \begin{bmatrix}
        \bSigma_1 & \\
        & \bSigma_2
        \end{bmatrix}, \quad \bV = \begin{bmatrix}
        \bV_1 & \bV_2
        \end{bmatrix},
    \]
    with $\bSigma_1 \in \mathbb{F}^{r \times r}$ and $\bV_1 \in \mathbb{F}^{N \times r}$.
   We have the standard error bound~\cite{HMT}:   
   \[
    \| \bA - \bQ \bQ^* \bA  \|_F^2 = \| \bSigma_2 \|^2_F + \| \bSigma_2 (\bV_2^* \bOmega^*)(\bV_1^* \bOmega^*)^\dagger \|_F^2,
   \] 
   and we note $(\bV_1^* \bOmega^*)^\dagger = \bOmega \bV_1 (\bV_1^* \bOmega^* \bOmega \bV_1)^{-1}$.
  By assumption, $\| (\bV_1^* \bOmega^* \bOmega \bV_1)^{-1} \|_2  = \sigma_{min}^{-2}(\bOmega \bV_1) \leq 1/\alpha_\mathtt{osi}$ with probability at least $1-\delta/2$.
    In this event we have:
    \begin{align*}
        \| \bSigma_2 (\bV_2^* \bOmega^*)(\bV_1^* \bOmega^*)^\dagger \|_F &\leq 1/\alpha_\mathtt{osi} \| \bSigma_2 (\bV_2^* \bOmega^*) \bOmega \bV_1 \|_F \leq  \varepsilon_{\mathtt{amm}}/\alpha_\mathtt{osi} \| \bSigma_2 \bV_2^*\|_F \| \bV_1 \|_F \\
        &\leq \sqrt{r}\varepsilon_{\mathtt{amm}}/\alpha_\mathtt{osi}  \| \bA - \bA_r \|_F.
    \end{align*}
\end{proof}
This result implies that to achieve an $(1+\varepsilon)$-accurate randomized SVD, it is enough for the sketch to be an $(\alpha_\mathtt{osi}, \delta/2,r)$-OSI (or, a fortiori, an OSE) with $\alpha_\mathtt{osi} > 0$, but not necessarily very small, and simultaneously satisfy an AMM property with $\varepsilon_\mathtt{amm} = \alpha_\mathtt{osi} \sqrt{\varepsilon/r}$. Both of which can be realized through a Strong JLM property via \Cref{cor:sjlm_ose} and \Cref{lem:appmult}.

We now can conclude the result of \Cref{th:gttrsvd}.

\begin{proof}[Proof of \Cref{th:gttrsvd} (Gaussian TT Sketch for RSVD)]
    From \Cref{lem:rsvdfromose}, \Cref{prop:sjlm_gtt}, \Cref{thm:GTT-OSE}, and \Cref{lem:appmult}, the sketch $\bOmega_\mathtt{GTT}$ allows to compute an $(1+\varepsilon)$-accurate low rank approximation of rank $r$, provided the sketch ranks satisfy $R = r_0 \geq d r_{d-1}  $ and $r_i \geq (d-i) r_{d-1}$ for $i  \in [d-2]$, with
    \[
       r_{d-1} \geq  \begin{cases}
            32 (2Le)^2 \bigl (r \log(9) + \log(2/\delta) \bigr ), & \text{(   $(1/2,\delta/2,r)$-OSE condition),} \\
            32 (Le)^2 r/\varepsilon \log(2/\delta), & \text{   (approximate matrix multiplication, $\varepsilon_\mathtt{amm} = 1/2 \sqrt{\frac{\varepsilon}{r}}$).}
        \end{cases}
    \]
    Pick $c = 32 (2 Le)^2 \cdot (2 \log(9)+1)$ and conclude.
\end{proof}

\subsection{Properties of {\name} sketches}\label{subsec:bstt}

Now knowing that each row block of the {\name} sketch is indeed a subspace embedding for large enough $R$, we can proceed to showing that this unifying framework is also an OSE, where the number of blocks $P$ helps relax the constraints on the ranks $R$. To this effect, we show the following result:

\begin{lemma}\label{lem:sjlm_stack}
Let $\bOmega_1, \dots, \bOmega_P $ be iid realizations of a random matrix with the Strong $\bigl (  \frac{\sqrt{P}\varepsilon}{4e^2}, \delta \bigr)$-JLM property, then the random matrix $\bOmega = \frac{1}{\sqrt{P}} \begin{bmatrix}
    \bOmega_1 \\ \vdots  \\  \bOmega_P
    \end{bmatrix}$ has the Strong $(\varepsilon, \delta)$-JLM property.
\end{lemma}
The proof is postponed to \Cref{app:sjlm_ttpr}. As a direct corollary of \Cref{lem:sjlm_stack} and \Cref{prop:sjlm_gtt} we have:
\begin{proposition}\label{prop:sjlm_sbtt}
    Consider the {\name} sketch \eqref{eq:ttpr} with parameters $P,R$:
    \[
        \bOmega_{\mathtt{\nameAbbrv}} = \frac{1}{\sqrt{P}} \begin{bmatrix}
            \bOmega_{\mathtt{GTT},1} \\ \vdots \\ \bOmega_{\mathtt{GTT},P}
        \end{bmatrix},
    \]
    where $\bOmega_{\mathtt{GTT},i}$ for $i \in [P]$ are iid realizations of a Gaussian TT sketch matrix \eqref{eq:gtt} with uniform ranks $r_0, \dots, r_{d-1} = R$. Then, the {\name} sketch matrix $\bOmega_{\mathtt{\nameAbbrv}}$ satisfies the Strong $(\varepsilon, \delta)$-JLM property provided
    \begin{equation} \label{eq:sjlm_sbtt}
        R \geq 8 ( L e /\varepsilon_1)^2 d  \log(1/\delta), \qquad P \geq 16 e^4 / \varepsilon_2^2,
    \end{equation}
    where $L \geq 1$ is the universal constant from \Cref{lem:sjlm_product} and  $\varepsilon_1 \varepsilon_2 = \varepsilon$ with $\varepsilon_1 \in (0,1]$ and $\varepsilon_2 > 0$.
\end{proposition}
We now have the elements to prove the main results of this section.
\begin{proof}[Proof of \Cref{thm:OSE} (OSE with {\name} Sketches)]
    By combining \Cref{prop:sjlm_sbtt} and \Cref{cor:sjlm_ose}, we have that $\bOmega_{\mathtt{\nameAbbrv}}$ enjoys an $(\varepsilon,\delta,r)$-OSE guarantee when it satisfies either a Strong $(\varepsilon/2, \delta/9^{r})$-JLM property, which translates to conditions:
    \[
        R \geq 32 (Le)^2 d \bigl(r \log(9) + \log(1/\delta) \bigr), \qquad P \geq 16 e^4 / \varepsilon^2,
    \]
    or a Strong $(\varepsilon/r, \delta)$-JLM property with conditions:
        \[
        R \geq 32 (Le)^2 d \log(1/\delta), \qquad P \geq 16 e^4 r^2 / \varepsilon^2.
    \]
\end{proof}

\begin{proof}[Proof of \Cref{cor:ttprsvd} (RSVD with {\name})]
    From \Cref{prop:sjlm_sbtt}, \Cref{lem:appmult}, \Cref{thm:OSE}, and \Cref{lem:rsvdfromose}, the $\bOmega_\mathrm{\nameAbbrv}$ sketch allows to compute an $(1+\varepsilon)$-accurate low-rank approximation of rank $r$ provided:
    \[
       \begin{cases}
            R \geq   32 (Le)^2 d \bigl(r \log(9) + \log(2/\delta) \bigr), \ P \geq 64 e^4, & \text{   ( $(1/2,\delta/2,r)$-OSE),} \\
            R \geq   32 (Le)^2 r d \log(1/\delta), \quad P \geq 16 e^4 / \varepsilon, &    (\text{AMM with }\varepsilon_\mathtt{amm}=1/2\sqrt{\frac{\varepsilon}{r}}).
        \end{cases}
    \]
    Pick $c = 32 (Le)^2 \cdot ( \log(9)+1)$ and conclude.
\end{proof}
\begin{remark}
    Note that the scaling $R = \cO(d \log(1/\delta))$ and $P = \cO(r/\varepsilon)$ is not enough to achieve the $(\varepsilon,\delta,r)$-OSE property.
\end{remark}

\section{Oblivious Subspace Injection (OSI) property for {\name}}\label{sec:OSI}

In this section we present
the proof of the OSI for {\name} based on the Gaussian comparison technique introduced in~\cite{tropp2025comparison} which estimates the minimum eigenvalue of random positive definite matrices. 

\subsection{Moments computation}
Given a random test matrix with iid block rows $\bOmega = \dfrac{1}{\sqrt{P}} \begin{bmatrix}
    \bOmega_1^\top & \cdots & \bOmega_P^\top
\end{bmatrix}^\top$ and a fixed matrix $\bQ \in \bbF^{N \times r}$ with orthonormal columns, we introduce the random positive semi-definite (p.s.d) matrix
\[
    \bY := (\bOmega \bQ)^* (\bOmega \bQ) = \frac{1}{P} \sum_{k=1}^P \bW_j \text{ with iid and p.s.d.  } \bW_j := (\bOmega_j \bQ)^* (\bOmega_j \bQ).
\]
The objective is to construct a Gaussian matrix $\bZ \in \bbF^{r \times r}$ whose statistics are determined by the summands:
\[
    \E{\bZ} = \E{\bW_1}  \text{ and }   \Var{\Tr{\bS_r\bZ}} \geq \frac{1}{P} \E{\Tr{\bS_r \bW_1}^2} \quad \forall\, \bS_r \in \bbF^{r \times r} \text{  self-adjoint,}
\]
allowing for a stochastic lower bound of $\lambda_{min}(\bY)$ by $\lambda_{min}(\bZ)$ which can be studied by specialized methods for Gaussian matrices. To apply this program to the {\name} sketch, we need to estimate the 4th-order moments $\E{\Tr{\bS_r \bW_1}^2}$ when $\bOmega_1$ is a Gaussian TT with uniform rank $R \geq 1$. 
\commentOut{
\begin{remark}
    {\color{red} In the matrix case $d=2$, it's important to decompose the matrix $\bY$ as an outer sum of rank 1 terms instead of a single object of the form $\bOmega\bOmega^\top$ (Wishart example in~\cite{tropp2025comparison}. This explains why the GTT sketch looks so weak from this angle - all we can show is that it doesn't explode, but it also doesn't improve as $R \to \infty$ the way {\name} does when $P \to \infty$. Is there a way to decompose GTT in smaller iid pieces, leading to a better estimate with this comparison technique? I'm not so sure.}
\end{remark}}

Given the fixed self-adjoint matrix $\bS := \bQ \bS_r \bQ^*$, let us bound from below the moments
\[
    \Mom{\bOmega_1^* \bOmega_1}{\bS} := \E{\Tr{\bOmega_1 \bS \bOmega_1^*}^2} = \E{\Tr{\bS_r \bW_1}^2}.
\]
\noindent
We seek to exploit the formulation \eqref{eq:recursiveTT} of the Gaussian TT as a product of successive embeddings:
\[
    \bOmega_1 = \bM_1 \cdots \bM_d \text{ where } \bM_j = \bI_{n_1 \cdots  n_{j-1}} \otimes \bG_j,
\]
where $\bG_j \in \bbF^{ R \times n_j R }$ for $j \in [d-1]$, and $\bG_d \in \bbF^{ R \times n_d}$ are random matrices with iid $\mathcal{N}_\bbF(0,1/R)$-distributed entries. 
Introduce the random self-adjoint matrices 
\begin{equation}\label{eq:Sk}
    \bS_k := \bM_k \cdots \bM_d \bS \bM_d^* \cdots \bM_k^*  \in \bbF^{n_1 \cdots n_{k-1} R \times n_1 \cdots n_{k-1} R} \quad \text{for} \quad k \in [d]
\end{equation}
and $\bS_{d+1} = \bS$. We define the following random variables for $k \in [d+1]$:
\begin{equation}\label{eq:towerrule}
   m_1 :=  \Tr{\bS_1}^2, \quad \quad m_k := \E{ \Tr{\bM_1 \cdots \bM_{k-1} \bS_{k} \bM_{k-1}^* \cdots \bM_1^*}^2 \ \vert \ \bS_{k} }, 
\end{equation}
so that we have the relation $m_k = \E{ m_{k-1} \ \vert \ \bS_{k}}$ for $k \geq 2$ where the expectation is taken solely over $\bG_{k-1}$, with $m_{d+1} = \Mom{\bOmega_1^* \bOmega_1}{\bS}$. To bound $m_k$ from above, let us introduce some useful notation and proceed with explicit computations.

\subsubsection{Partial traces}

\noindent
Following~\cite{meyer2023hutchinson} we define the \textit{partial trace} operators:

\begin{definition}
    Let $\bS \in \bbF^{N \times N}$. Any choice of a subset of indices $\cI \subset [d]$ decomposes the overall system into two subsystems of respective sizes
    \begin{equation}\label{def:subsystemsize}
        n_\cI = \prod_{i \in \cI} n_i \qquad \text{and} \qquad n_{\cI^c} = \prod_{i=1, i \notin \cI}^d n_i,
    \end{equation}
    such that $n_\cI n_{\cI^c} = N$.
    \begin{enumerate}
        \item Let $i \in [d]$. The partial trace of $\bS$ with respect to the $i$\textsuperscript{th} subsystem is
        \begin{equation}
            \tr{\{i\}}{\bS} := \sum_{j=1}^{n_i} \bE_{\{i\},j}^\top \bS \bE_{\{i\},j} \in \bbF^{n_{\{i\}^c} \times n_{\{i\}^c}},
        \end{equation}
        where $\bE_{\{i\},j} := \bI_{n_1 \dots n_{i-1}} \otimes \be_j \otimes \bI_{n_{i+1} \dots n_d}$, and $\be_j \in \bbF^{n_j}$ are standard basis vectors.
        \item Let $(i_1, \dots, i_{\vert \cI \vert})$ be the indices in $\cI$, sorted so that $i_1 < \dots < i_{\vert \cI \vert}$. Then, the partial trace of $\bS$ with respect to the subsystem $\cI$ is
        \begin{equation}\label{def:partialtrace}
            \tr{\cI}{\bS} := \tr{\{i_1\}}{ \tr{\{i_2\}}{\cdots \left ( \tr{\{i_{\vert \cI \vert}\}}{\bS}\right )} } \in \bbF^{n_{\cI^c} \times n_{\cI^c}}
        \end{equation}
    \end{enumerate}
\end{definition}

\noindent 
Partial trace operators satisfy a number of properties~\cite{meyer2023hutchinson}, in particular trace and positivity preservation: $\Tr{\tr{\cI}{\bS}} = \Tr{\bS}$ for any $\bS$ and $\tr{\cI}{\bS} \geq 0$ if $\bS \geq 0$; as well as the following Cauchy-Schwarz-like property:
\begin{lemma}\label{lem:partialtrace_cs} Let $\cI \subset [d]$ and $\ba,\bb$ be arbitrary vectors in $\bbF^N$, then
\[
 \Vert \tr{\cI}{\ba \bb^*}\Vert^2_F \leq \Vert \tr{\cI}{\ba \ba^*}\Vert_F \Vert \tr{\cI}{\bb \bb^*}\Vert_F.
\]
\end{lemma}
\begin{proof}
    Let $\bP_\cI$ represent a coordinate permutation that brings indices $\{1, \dots, \vert \cI \vert\}$ in the multi-index to positions $\{i_1 , \dots , i_{\vert \cI \vert} \}$ listed in $\cI$. Write $\bP_\cI^\top \ba = \sum_{j=1}^{n_\cI} \be_j \otimes \ba_j$ and $\bP_\cI^\top \bb = \sum_{j=1}^{n_\cI} \be_j \otimes \bb_j$ such that
\begin{align*}
    \Vert \tr{\cI}{\ba \bb^*}\Vert_F^2 &= \Tr{\left (\sum_{i=1}^{n_\cI} \ba_i \bb_i^*\right)\left (\sum_{j=1}^{n_\cI} \ba_j \bb_j^*\right)^*} \\
    &= \sum_{i,j=1}^{n_\cI}(\ba_j^* \ba_i)  (\bb_i^* \bb_j )\\
   \text{(Cauchy-Schwarz)} \qquad &\leq \left ( \sum_{i,j=1}^{n_\cI} (\ba_i^* \ba_j) (\ba_j^* \ba_i ) \right )^{1/2}\left ( \sum_{i,j=1}^{n_\cI} (\bb_i^* \bb_j) (\bb_j^* \bb_i)  \right )^{1/2} \\
    & \hspace{-.75in}=  \left ( \Tr{(\sum_{i=1}^{n_\cI} \ba_i\ba_i^* )(\sum_{j=1}^{n_\cI} \ba_j \ba_j^* ) } \right )^{1/2} \left ( \Tr{(\sum_{i=1}^{n_\cI} \bb_i\bb_i^* )(\sum_{j=1}^{n_\cI} \bb_j \bb_j^* ) } \right )^{1/2}\\
    &= \Vert \tr{\cI}{\ba \ba^*}\Vert_F \Vert \tr{\cI}{\bb \bb^*}\Vert_F.
\end{align*}
\end{proof}

The following technical result is essential to our moment bounds, with a detailed proof provided in \Cref{app:moment_tensor}.
\begin{lemma}\label{lem:moment_tensor}
    Let $k \in [d]$, let random matrices $\bS_k$, $\bS_{k+1}$ be as in~\eqref{eq:Sk}, and let $\mathcal{I} \subseteq [k-1]$. Define constants $p_\bbR = 2$ and $p_\bbC = 1$ depending on the scalar field. Then we have the inequalities (equalities in the complex case):
    \begin{align*}
        &\E{ \Vert \tr{\cI\cup\{k\}}{\bS_k}\Vert^2_F\ \vert \ \bS_{k+1}} \leq \Vert \tr{\cI \cup \{k,k+1\}}{\bS_{k+1}}\Vert^2_F + \frac{p_\bbF}{R} \Vert \tr{\cI}{\bS_{k+1}}\Vert^2_F, \\
        &\E{ \Vert \tr{\cI}{\bS_k}\Vert^2_F\ \vert \ \bS_{k+1}} \leq \frac{1}{R} \Vert \tr{\cI \cup \{k,k+1\}}{\bS_{k+1}}\Vert^2_F + \left(1 + \frac{p_\bbF-1}{R} \right)\Vert \tr{\cI}{\bS_{k+1}}\Vert^2_F.
    \end{align*}
\end{lemma}

\subsubsection{Moment bounds for Gaussian TT sketches}
Bringing together the estimates of \Cref{lem:moment_tensor} with the tower rule~\eqref{eq:towerrule} leads to the following moment bound:
\begin{proposition}\label{prop:momentbound}
    There exists positive coefficients $\left \{\gamma_\cI \right \}_{\cI \subseteq [d]}$ satisfying:
    \[
        \sum_{\cI \subseteq [d]} \gamma_\cI = \left ( 1+\frac{p_\bbF}{R} \right)^d \qquad \text{where } p_\bbF = \begin{cases} 2 & \text{if } \bbF = \bbR \\ 1 & \text{if } \bbF = \bbC\end{cases},
    \]
    such that the moments of the Gaussian TT sketch satisfy the upper bound:
    \[
        \Mom{\bOmega_1^* \bOmega_1}{\bS} \leq \sum_{\cI \subseteq [d]} \gamma_\cI \Vert \tr{\cI}{\bS} \Vert^2_F,
    \]
    which is an equality in the complex case. Notably, we have the \change{particular case $\gamma_{[d]} = 1$}.
\end{proposition}
\begin{proof}
    We proceed by recurrence on the random variables $m_k$ defined above~\eqref{eq:towerrule}. Let us assume at step $k \in [d]$ we have established that
    \[
        m_k \leq \sum_{\cI \subseteq [k]} \gamma_\cI^{(k)} \Vert \tr{\cI}{\bS_k} \Vert^2_F
    \]
    with equality in the complex case, for some positive coefficients $\{\gamma_\cI^{(k)}\}$ satisfying
    \[
        \sum_{\cI \subseteq [k]} \gamma_\cI^{(k)} = \left ( 1+\frac{p_\bbF}{R} \right)^{k-1},
    \]
    which is trivially satisfied for $k=1$ with $m_1 = \Tr{\bS_1}^2 = \Vert \tr{\{1\}}{\bS_1} \Vert^2_F$ and thus coefficients $\gamma^{(1)}_{\{1\}} = 1$, $\gamma^{(1)}_{\emptyset} = 0$. 
    Let us estimate $m_{k+1} = \E{m_k \ \vert\ \bS_{k+1}}$ by decomposing:
    \begin{align*}
        \sum_{\cI \subseteq [k]}& \gamma_\cI^{(k)} \E{\Vert \tr{\cI}{\bS_k} \Vert^2_F \ \vert\ \bS_{k+1}} 
         = \\ &\sum_{\cI \subseteq [k-1]} \gamma_\cI^{(k)} \E{\Vert \tr{\cI}{\bS_k} \Vert^2_F \ \vert\ \bS_{k+1}} + \gamma_{\cI\cup\{k\}}^{(k)} \E{\Vert \tr{{\cI}}{\tr{\{k\}}{\bS_k}} \Vert^2_F \ \vert\ \bS_{k+1}}
    \end{align*}
    We now use \Cref{lem:moment_tensor} to find the upper bound in the real case (and equality in the complex case):
    \begin{align*}
        m_{k+1} \leq  \sum_{\cI \subseteq [k-1]} & \left ( \frac{1}{R}\gamma_\cI^{(k)} + \gamma_{\cI\cup\{k\}}^{(k)} \right ) \left \Vert \tr{\cI\cup\{k\}}{\tr{\{k+1\}}{\bS_{k+1}}} \right \Vert_F^2 \\
        & \qquad + \left ( \left ( 1+ \frac{p_\bbF - 1}{R} \right )\gamma_\cI^{(k)} + \frac{p_\bbF}{R}\gamma_{\cI\cup\{k\}}^{(k)} \right ) \left \Vert \tr{\cI}{\bS_{k+1}} \right \Vert_F^2 
    \end{align*}
    and thus the induction hypothesis is satisfied for $k+1$ with the new coefficients defined for $\cI \subseteq [k-1]$:
    \[
        \gamma_\cI^{(k+1)} := \left ( 1+ \frac{p_\bbF - 1}{R} \right )\gamma_\cI^{(k)} + \frac{p_\bbF}{R}\gamma_{\cI\cup\{k\}}^{(k)} \quad \text{and} \quad \gamma_{\cI\cup\{k,k+1\}}^{(k+1)} := \frac{1}{R}\gamma_\cI^{(k)} + \gamma_{\cI\cup\{k\}}^{(k)},
    \]
    and all others set to zero. In particular, we note that $\gamma_\cI^{(k+1)} + \gamma_{\cI\cup\{k,k+1\}}^{(k+1)} = (1+p_\bbF/R) \left (\gamma_\cI^{(k)} + \gamma_{\cI\cup\{k\}}^{(k)} \right )$ and verify the identities:
    \[
        \sum_{\cI \subseteq [k+1]} \gamma_\cI^{(k+1)} = \left ( 1+ \frac{p_\bbF}{R} \right ) \sum_{\cI \subseteq [k]}\gamma_\cI^{(k)}  = \left ( 1+ \frac{p_\bbF}{R} \right )^k, \quad \gamma_{\{k\}}^{(k+1)} = \gamma_{[k]}^{(k+1)} = 0, \quad \gamma_{[k+1]}^{(k+1)} = 1.
    \]
    By induction, the hypothesis holds for $k=d+1$ which is the desired result, as $m_{d+1} = \Mom{\bOmega_1^* \bOmega_1}{\bS}$, albeit with the slight modification that we define coefficients for all $\cI \subseteq \{1,\dots,d-1\}$:
    \[
        \gamma_\cI := \left ( 1+ \frac{p_\bbF - 1}{R} \right )\gamma_\cI^{(d)} + \frac{p_\bbF}{R}\gamma_{\cI\cup\{d\}}^{(d)} \quad \text{and} \quad \gamma_{\cI\cup\{d\}} := \frac{1}{R}\gamma_\cI^{(d)} + \gamma_{\cI\cup\{d\}}^{(d)},
    \]
    and check in particular $\gamma_{[d]} = 1$. Using the recurrence relations, we compute finally $\gamma_\emptyset^{(1)} = 0$, $\gamma_\emptyset^{(2)} = \frac{p_\bbF}{R}$, then for $k \geq 1$:
    $
        \gamma_\emptyset^{(k+1)} = \left (1 + \frac{p_\bbF-1}{R} \right ) \gamma_\emptyset^{(k)} = \left (1 + \frac{p_\bbF-1}{R} \right )^{k-2} \frac{p_\bbF}{R}
    $.
\end{proof}

\subsection{Gaussian Comparison Model}\label{sec:gaussiancomparison}
Let us now construct Gaussian matrices that have variance equal to $\Vert \tr{\cI}{\bS} \Vert_F^2$ for each $\cI \subseteq [d]$. Given $\bPhi_{\cI^c}$ of size $n_{\cI^c} \times n_{\cI^c}$ in the GUE ensemble, set
\begin{equation}\label{eq:gaussian_comparison}
    \bX_\cI =  \bP_\cI (\bI_{n_\cI} \otimes \bPhi_{\cI^c}) \bP_\cI^\top
\end{equation}
where $\bP_\cI$ represents a coordinate permutation that brings indices $\{1, \dots, \vert \cI \vert\}$ in the multi-index to positions $\{i_1 , \dots , i_{\vert \cI \vert} \}$ listed in $\cI$. When $\cI$ is the full set $[d]$, we take $\bX_{\cI} = \Phi_\cI \bI$ with $\Phi_\cI$ a scalar Gaussian variable.
\begin{lemma}
    \label{lem:weak_variance_gaussian}
    Gaussian matrices $\bX_\cI$ of the form \eqref{eq:gaussian_comparison} satisfy 
    \[
        \Var{\bX_\cI}(\bS) = \Vert \tr{\cI}{\bS} \Vert_F^2 \qquad \text{and} \qquad \sigma_*^2(\bX_\cI) = 1,
    \]
    where $\sigma_*^2(\bX) := \max_{\Vert \bu \Vert_2 = 1} \Var{\bu^* \bX \bu}$ is the weak variance~\cite{tropp2025comparison}.
\end{lemma}
\begin{proof}
    First, we observe $\Tr{\bS \bX_\cI} =\Tr{\bP_\cI^\top\bS \bP_\cI (\bI_{n_\cI} \otimes \bPhi_{\cI^c})}$.
    Writing $\bP_\cI^\top\bS \bP_\cI = \sum_{ij} \bE_{ij} \otimes \bS_{ij}$, we then compute
    \begin{align*}
        \Tr{\bP_\cI^\top\bS \bP_\cI (\bI_{n_\cI} \otimes \bPhi_{\cI^c})} &= \sum_{i,j=1}^{n_\cI} \Tr{(\bE_{ij} \otimes \bS_{ij}) (\bI_{n_\cI} \otimes \bPhi_{\cI^c})} \\ &=  \sum_{i,j=1}^{n_\cI} \Tr{\bE_{ij}} \Tr{\bS_{ij} \bPhi_{\cI^c}} \\
        &= \Tr{ \left ( \sum_{j=1}^{n_\cI} \bS_{jj}  \right ) \bPhi_{\cI^c}} \\
        &= \Tr{ \tr{[\vert\cI\vert]}{\bP_\cI^\top\bS \bP_\cI}\bPhi_{\cI^c}} \\
        &= \Tr{ \tr{\cI}{\bS}\bPhi_{\cI^c}}.
    \end{align*}
    Thus by the standard property of the GUE~\cite{tropp2025comparison}:
        \begin{align*}
        \Var{\bX_\cI}(\bS) = \Var{\bPhi_{\cI^c}}( \tr{\cI}{ \bS }) = \Vert \tr{\cI}{ \bS }\Vert^2_F.
    \end{align*}
    For the weak variance, write $\bP_\cI^\top\bu = \sum_{j=1}^{n_\cI} \be_j \otimes \bu_j$ with $\Vert \bu \Vert^2 = \sum_{j=1}^{n_\cI} \Vert \bu_j \Vert^2 = 1$, then
    \begin{align*}
        \E{\vert \bu^* \bX_\cI \bu \vert^2} &= \sum_{i,j=1}^{n_\cI} \E{(\bu_i^* \bPhi_{\cI^c} \bu_i) \overline{(\bu_j^* \bPhi_{\cI^c} \bu_j)}} \\
        &=\sum_{i,j=1}^{n_\cI} \vert \bu_i^* \bu_j \vert^2 \leq \sum_{i,j=1}^{n_\cI} \Vert \bu_i \Vert^2 \Vert \bu_j \Vert^2 = 1.
    \end{align*}
    The bound is attained for $\bu_1 = \dots = \bu_j$, so we conclude 
    \[
        \sigma_*^2(\bX_\cI) = \max_{\Vert \bu \Vert = 1} \Var{\bu^* \bX_\cI \bu} = 1.
    \]
\end{proof}
\noindent
Given a subspace encoded by the range of an orthogonal matrix $\bQ$, the Gaussian comparison model \change{for the Gaussian sketch} takes the form
\begin{equation}\label{eq:gaussianmodel}
    \bX = \bI_N + \sum_{\cI \subseteq [d]} \change{\sqrt{\gamma_\cI} }\bX_\cI \qquad \text{and} \qquad \bQ^* \bX \bQ = \bI_r + \sum_{\cI \subseteq [d]} \change{\sqrt{\gamma_\cI}} \bQ^*\bX_\cI \bQ.
\end{equation}

\change{The model~\eqref{eq:gaussianmodel} is the comparison object for the eigenvalue bound of~\cite{tropp2025comparison}, whose hypotheses require it to share the mean of the sketch Gram matrix $\bOmega_1^* \bOmega_1$ and to dominate its second moment, that is, $\E{\bX} = \E{\bOmega_1^* \bOmega_1}$ and $\Var{\bX}(\bS) \geq \Mom{\bOmega_1^* \bOmega_1}{\bS}$ for every Hermitian $\bS$. The means agree, as the $\bX_\cI$ are centered and mutually independent while $\E{\bOmega_1^* \bOmega_1} = \bI_N$. For the second moment, independence and \Cref{lem:weak_variance_gaussian} give
\[
    \Var{\bX}(\bS) = \sum_{\cI \subseteq [d]} \gamma_\cI \, \Var{\bX_\cI}(\bS) = \sum_{\cI \subseteq [d]} \gamma_\cI \, \| \tr{\cI}{\bS} \|_F^2 \;\geq\; \Mom{\bOmega_1^* \bOmega_1}{\bS},
\]
the inequality being exactly \Cref{prop:momentbound}, with equality when $\bbF = \bbC$. The comparison is carried out in \Cref{sec:osi_proof}.}

\begin{definition}\label{def:entanglementI}
    Let $\bQ \in \bbF^{N \times r}$ be an arbitrary matrix with orthonormal columns, $\cI \subseteq [d]$ a subset of indices. Let us define the constant
    \[
        C_{\bQ,\cI} = \max_{\Vert \bu \Vert = 1}  \Vert \tr{\cI}{(\bQ \bu) (\bQ \bu)^*} \Vert_F, \qquad \text{with } \frac{1}{\sqrt{n_{\cI^c}}} \leq C_{\bQ,\cI} \leq 1,
    \]
    which measures the entanglement between the elements of the subsystem indexed by $\cI$ and its complement, in the subspace $\bQ$.
\end{definition}
The above bounds on $C_{\bQ,\cI}$ follow directly from the observation that since $\tr{\cI}{(\bQ \bu) (\bQ \bu)^*}$ is symmetric positive definite of size $n_{\cI^c}$ by $n_{\cI^c}$,
\[
    \frac{1}{\sqrt{n_{\cI^c}}}\Tr{\tr{\cI}{(\bQ \bu) (\bQ \bu)^*}} \leq \Vert \tr{\cI}{(\bQ \bu) (\bQ \bu)^*} \Vert_F \leq \Tr{\tr{\cI}{(\bQ \bu) (\bQ \bu)^*}},
\]
and $\Tr{\tr{\cI}{(\bQ \bu) (\bQ \bu)^*}} = \Tr{(\bQ \bu) (\bQ \bu)^*} = \Vert \bQ \bu \Vert^2_2 = 1$.
\begin{remark}\label{rem:kronmax}
    In particular, the constant $C_{\bQ, \cI}$ is maximized, $C_{\bQ, \cI} = 1$, whenever the subspace spanned by columns of $\bQ$ contains a vector of Kronecker type, i.e. there exists a unit vector $\bu_1 \in \bbF^r$ such that $\bP_\cI^T \bQ \bu_1 = \bu_{\cI} \otimes \bu_{\cI^c}$. Indeed, in this case we may explicitly compute the partial trace $\tr{\cI}{\bQ \bu_1 (\bQ \bu_1)^*} = \| \bu_{\cI} \|^2_2 (\bu_{\cI^c} \bu_{\cI^c}^*)$, and thus 
    \[
        \| \tr{\cI}{\bQ \bu_1 (\bQ \bu_1)^*} \|^2_F = \| \bu_{\cI} \|^2_2  \| \bu_{\cI^c} \|^2_2 = \| \bQ \bu_1 \|^2_2 = 1.
    \]
\end{remark}

\change{
\begin{definition}\label{def:entanglement}
    Given an arbitrary matrix $\bQ \in \bbF^{N \times r}$ with orthonormal columns, we define the subspace entanglement measure
    \begin{equation}
        C_\bQ(R) := \sqrt{\sum_{\cI \subsetneq [d] } \gamma_\cI(R) C_{\bQ,\cI}^2}  = \sqrt{\sum_{\cI \subseteq [d] } \gamma_\cI(R) C_{\bQ,\cI}^2 - 1},
    \end{equation}
    where $C_{\bQ,\cI}$ is defined in \Cref{def:entanglementI} and $\gamma_\cI(R)$ in \Cref{prop:momentbound} and we recall $\gamma_{[d]}(R) = C_{\bQ,[d]} = 1$.
    In particular, note 
    \[
        \sum_{\cI \subsetneq [d] } \frac{\gamma_\cI(R)}{n_{\cI^c}} \leq C^2_\bQ(R) \leq \left( 1 + \frac{p_\bbF}{R} \right)^d - 1 \quad \implies \quad C_\bQ(R) = \cO \left( \sqrt{ \frac{d p_\bbF} {R}} \right),
    \]
    where the latter estimate holds for $R \geq d$ and the upper bound is achieved when the column space of $\bQ$ contains a vector of Kronecker type $\bu = \bu_1 \otimes \dots \otimes \bu_d$.
\end{definition}
}
\begin{lemma}\label{lem:mineig}
    Let $\bQ \in \bbF^{N \times r}$ be an arbitrary matrix with orthonormal columns, \change{and $\bX = \bI_N + \sum_{\cI \subseteq [d]} \sqrt{\gamma_\cI} \bX_\cI$ the Gaussian random matrix defined above in \Cref{eq:gaussianmodel}}.  Then, the weak variance satisfies $\change{\sigma_*^2 (\bQ^* \bX \bQ) \leq 1+C_\bQ^2}$ and its lowest eigenvalue
    \[
        \change{\E{\lambda_{min}(\bQ^* \bX \bQ)} \geq 1 - 2 C_\bQ \sqrt{2r}.}
    \]
\end{lemma}

\begin{proof}
Compute the weak variance explicitly. Given a unit vector $\bu \in \bbF^{r}$,
\begin{align*}
    \E{\vert \bu^* \bQ^* \bX_\cI \bQ \bu \vert^2} &= \E{ \left \vert \Tr{\bX_\cI \bQ \bu (\bQ \bu)^*} \right \vert^2 } \\
    & = \E{ \left \vert \Tr{ \bPhi_{\cI^c} \tr{\cI}{(\bQ \bu) (\bQ \bu)^*}} \right \vert^2}\\
    &= \Vert \tr{\cI}{(\bQ \bu) (\bQ \bu)^*} \Vert^2_F,
\end{align*}
so $\sigma_*^2(\change{\bQ^* \bX_\cI \bQ}) = \max_{\Vert \bu \Vert = 1} \Var{\bu^* \bQ^* \bX_\cI \bQ \bu} = C_{\bQ,\cI}^2$. \change{By independence of the $\bX_\cI$, it follows $\sigma_*^2(\bQ^* \bX \bQ) = \sigma_*^2 \change {\left ( \bI_r + \sum_{\cI \subseteq [d]} \sqrt{\gamma_\cI} \bQ^*\bX_\cI\bQ \right ) \leq \sum_{\cI \subseteq [d]} \gamma_\cI C_{\bQ,\cI}^2 = 1+C^2_\bQ}$.} Next, we follow the standard study of the extremal eigenvalues of a Gaussian matrix. The first trivial bound is given by observing that $\lambda_{min}(\bQ^* \bX_\cI \bQ) \geq \lambda_{min}(\bX_\cI)  \geq \lambda_{min}(\bPhi_{\cI^c}) \geq - 2 \sqrt{n_{\cI^c}}$.
To improve on this exponential scaling, consider the real-valued empirical process indexed by a unit vector $\bu \in \bbC^r$:
\begin{align*}
    X_{\bu} = 1 - \bu^* \change{\bX} \bu = \change{ - \sqrt{\gamma_{[d]}} \Phi_{[d]} - \sum_{\cI \subsetneq [d]} \sqrt{\gamma_\cI} }\Tr{\bX_{\cI} (\bQ \bu) (\bQ \bu)^*} & \\
    = \change{ - \sqrt{\gamma_{[d]}} \Phi_{[d]} -\sum_{\cI \subsetneq [d]} \sqrt{\gamma_\cI} } \Tr{ \bPhi_{\cI_c} \tr{\cI}{(\bQ \bu) (\bQ \bu)^*}},
\end{align*}
\change{where we have used $\bPhi_{[d]} = \Phi_{[d]} \bI$ with $\Phi_{[d]}$ a scalar Gaussian variable and $\bu^* \bQ^* \bQ \bu = 1$.}
    We have that $\lambda_{min}(\bQ^* \change{\bX} \bQ) = - \sup_{\| \bu \|=1} X_{\bu}$.
We now study increments of this empirical process, using the standard property of the GUE~\cite{tropp2025comparison} \change{and independence of the $\bPhi_{\cI^c}$:}
\begin{align*}
    \Var{X_{\bu}-X_{\bu'}} \hspace{-.5in}&\hspace{.5in}=  \change{\sum_{\cI \subsetneq [d]} \gamma_\cI \;}  \E{\Tr{\bPhi_{\cI^c} \tr{\cI}{(\bQ\bu)(\bQ\bu)^*-(\bQ\mathbf{u'})(\bQ\mathbf{u'})^*} }^2}  \\ 
    &= \change{\sum_{\cI \subsetneq [d]} \gamma_\cI \;}\| \tr{\cI}{(\bQ\bu)(\bQ\bu)^*-(\bQ\bu')(\bQ\bu')^*}  \|_F^2 \\
    &= \change{\sum_{\cI \subsetneq [d]} \gamma_\cI \;}\Big\| \frac{1}{2} \tr{\cI}{(\bQ(\bu-\bu'))(\bQ(\bu+\bu'))^* \change{+}(\bQ(\bu+\bu'))(\bQ(\bu-\bu'))^*} \Big\|_F^2 \\
    & \leq \change{\sum_{\cI \subsetneq [d]} \gamma_\cI \;}\change{\left ( \frac{1}{2} \| \tr{\cI}{\bQ(\bu-\bu')(\bQ(\bu+\bu'))^* } \|_F^2 + \frac{1}{2} \| \tr{\cI}{(\bQ(\bu+\bu'))(\bQ(\bu-\bu'))^* } \|_F^2 \right )} \\
    & \leq \change{\sum_{\cI \subsetneq [d]} \gamma_\cI \;}\| \tr{\cI}{(\bQ(\bu+\bu'))(\bQ(\bu-\bu'))^* } \|_F^2.
\end{align*}
By \Cref{lem:partialtrace_cs}, we have 
\[
    \Var{X_{\bu}-X_{\bu'}} \leq \change{\sum_{\cI \subsetneq [d]} \gamma_\cI \;} \| \tr{\cI}{\bQ(\bu+\bu')(\bQ(\bu+\bu'))^*} \|_F \| \tr{\cI}{\bQ(\bu-\bu')(\bQ(\bu-\bu'))^*}  \|_F,
\]
so by definition of $C_{\bQ,\cI}$ \change{and $C_\bQ$:}
\[
    \Var{X_{\bu}-X_{\bu'}} \leq \change{\Big ( \sum_{\cI \subsetneq [d]} \gamma_\cI C_{\bQ,\cI}^2 \Big )} \| \bu+\bu' \|^2 \| \bu-\bu' \|^2 \leq 4 \change{C_{\bQ}^2} \| \bu-\bu' \|^2.
\]
Define the centered Gaussian process 
\[
    Y_{\bu} = \change{- 2 \sqrt{2} \,C_{\bQ} \,\Re{\mathbf{g}^* \bu}},
\]
where $\mathbf{g}$ is a complex Gaussian vector $\mathcal{N}_\bbC(0,\bI_r)$. 
Then we have 
\[
    \Var{Y_{\bu}-Y_{\bu'}} = 4 \change{C_{\bQ}^2} \| \bu-\bu' \|^2.
\]
Finally, by the Sudakov-Fernique inequality, we conclude 
\begin{align*}
    \E{\lambda_{min}(\bQ^* \change{\bX} \bQ)} = 1 - \E{\sup_{\bu=1} X_{\bu} } &\geq 1 -\E{ \sup_{\bu=1} Y_{\bu}} \\
    &= \change{1 - 2 \sqrt{2}\,C_{\bQ} \;  \E{\Vert \bg \Vert} \geq 1 - 2 C_{\bQ} \;\sqrt{2r}}.
\end{align*}
\end{proof}

\subsection{OSI property for the {\name} Sketch}\label{sec:osi_proof}

\begin{proof}[Proof of \Cref{thm:OSI} (OSI with {\name} sketches)]
    Fix an arbitrary $\bQ \in \bbF^{N \times r}$ with orthonormal columns.
    Using the same notation as in \Cref{sec:gaussiancomparison}, the Gaussian comparison model to the {\name} sketch writes:
    \[
        \bZ = \frac{1}{P} \sum_{i=1}^P \bQ^* \bX_{i} \bQ \sim \bI_r + \frac{1}{\sqrt{P}} \sum_{\cI \subseteq [d]} \change{\sqrt{\gamma_\cI}} \bQ^*\bX_{\cI} \bQ,
    \]
    where $\bX_{i}$ are iid copies of the Gaussian comparison model \eqref{eq:gaussian_comparison} for the the Gaussian TT sketch. From Theorem 2.3 in reference~\cite{tropp2025comparison}, for $\bY := (\bOmega \bQ)^* (\bOmega \bQ)$ where $\bOmega$ is the {\name} sketch matrix,
    \[
        \E{\lambda_{min}(\bY)} \geq \E{\lambda_{min}(\bZ)} - \sqrt{2\sigma_*^2(\bZ)\log(2r/\delta)}
    \]
    and for any $t>0$,
    \[
        \P{\lambda_{min}(\bY) < \E{\lambda_{min}(\bZ)} - t} \leq 2 r e^{-t^2/(2\sigma_*^{2}(\bZ))}.
    \]
    \change{Further, since $\bZ - \bI_r$ has the law of $\frac{1}{\sqrt{P}}(\bQ^* \bX \bQ - \bI_r)$, \Cref{lem:mineig} gives:}
    \[
    \change{\sigma_*^2(\bZ) \leq \frac{1 + C_\bQ^2}{P}
    \quad \text{and} \quad
    \E{\lambda_{min}(\bZ)} \geq 1 -
        2  C_\bQ  \sqrt{\frac{2r}{P}}.}
    \]
    Now setting $t =  \sqrt{1+C_\bQ^2} \sqrt{2\log(2r/\delta)/P}$, we obtain with probability $1-\delta$ that:
    \[
        \lambda_{min}(\bY) \geq 1 - \frac{1}{\sqrt{P}} \left ( 
        2   C_\bQ  \sqrt{2r} + \sqrt{1+C_\bQ^2} \sqrt{2\log(2r/\delta)} \right ).
    \]
\change{It remains to translate the requirement
\[
    \frac{1}{\sqrt{P}} \left (
        2 C_\bQ \sqrt{2r} + \sqrt{1 + C_\bQ^2} \sqrt{2\log(2r/\delta)} \right ) \leq \varepsilon
\]
into a condition on the number of blocks $P$. This inequality is equivalent to
\[
    P \geq \frac{1}{\varepsilon^2} \left ( 2 C_\bQ \sqrt{2r} + \sqrt{1 + C_\bQ^2} \sqrt{2\log(2r/\delta)} \right )^2,
\]
and applying the elementary inequality $(a+b)^2 \leq 2a^2 + 2b^2$ with $a = 2 C_\bQ \sqrt{2r}$ and $b = \sqrt{1 + C_\bQ^2} \sqrt{2\log(2r/\delta)}$, so that $a^2 = 8 C_\bQ^2 r$ and $b^2 = 2(1+C_\bQ^2)\log(2r/\delta)$, yields the sufficient condition
\[
    P \geq \frac{1}{\varepsilon^2} \left ( 16 C_\bQ^2 r + 4(1+C_\bQ^2) \log(2r/\delta) \right ) = \frac{4}{\varepsilon^2} \left ( 4 C_\bQ^2 r + (1+C_\bQ^2) \log(2r/\delta) \right ),
\]
which is precisely condition~\eqref{eq:OSI_condition}.}

\end{proof}

\subsection{Randomized SVD guarantee with a {\name} sketch }\label{sec:rsvd_ttpr_proof}
 In this section, we establish probabilistic guarantees (\Cref{thm:rsvd_ttpr}) for a \texttt{QB} factorization (which may be further manipulated to construct a compact \texttt{SVD}) in the particular instances of the {\name} sketch.
As a first step, we show the following improvement on Lemma 2.1 in~\cite{camano2025faster} with constant $C_\delta = \cO(\alpha^{-1} \sqrt{d/(PR\delta})$ instead of $\cO(1/\alpha \delta)$:
\begin{lemma}\label{lem:rsvd_term}
    Fix two matrices with orthonormal columns $\bQ \in \bbF^{N \times r}$ and $\bQ_\perp \in \bbF^{N \times s}$ whose ranges are mutually orthogonal: $\bQ^* \bQ_\perp = {\boldsymbol 0}$, and pick $\delta >0$. Choose an arbitrary matrix $\bB \in \bbF^{t \times s}$, and draw a random {\name} sketch matrix $\bOmega \in \bbF^{N \times PR}$ realizing an $(\alpha,\delta,r)$-OSI of the subspace spanned by the $r$ columns of $\bQ$ as determined by~\eqref{eq:OSI_condition}.
    
    With probability at least $1-2\delta$, the matrix $\bOmega \bQ$ has full column rank, and
    \[
        \| \bB (\bOmega\bQ_\perp )^* \bigl ((\bOmega\bQ )^* \bigr )^\dagger \|^2_F \leq \left ( 1 + \sqrt{\frac{\bigl ( 1 + p_\bbF /R \bigr )^d - 1}{P \delta / 2}}\right ) \| \bB \|^2_F / \alpha.
    \]
    
\end{lemma}
\begin{proof}
    The standard operator-norm bound for the Frobenius norm yields:
    \[
        \| \bB (\bOmega \bQ_\perp)^* \bigl ((\bOmega\bQ)^* \bigr )^\dagger \|^2_F \leq \| \bB (\bOmega \bQ_\perp)^* \|^2_F  \| (\bOmega\bQ)^\dagger \|^2_2 \leq \frac{\| \bB (\bOmega \bQ_\perp)^* \|^2_F}{\sigma_{min}^2(\bOmega\bQ)}.
    \]
    \Cref{thm:OSI} ensures that ${\sigma_{min}^2(\bOmega\bQ)} \geq \alpha$ with probability at least $1 - \delta$. To bound the numerator, we note that we may compute the expectation $\mathbb{E}{\| \bB (\bOmega \bQ_\perp)^* \|^2_F} = \| \bB \|^2_F$ since $\mathbb{E}[{\bOmega^* \bOmega}] = \bI$~\cite{camano2025faster}. Furthermore, let us write the variance of $\| \bB (\bOmega \bQ_\perp)^* \|^2_F = \frac{1}{P} \sum_{j=1}^P \| \bB (\bOmega_j \bQ_\perp)^* \|^2_F$, with $\bOmega_j$ iid copies of the Gaussian TT sketch. The variance of the iid variables $X_j = \| \bB (\bOmega_j \bQ_\perp)^* \|^2_F$ writes as:
    \[
        \E{ (X_j - \| \bB \|^2_F)^2} = \E{\Tr{\bOmega_j^* \bOmega_j (\bB \bQ_\perp^*)^* (\bB \bQ_\perp^*)}^2} - \| \bB \|^4_F = \Mom{\bOmega_j^* \bOmega_j}{\bS}  - \| \bB \|^4_F,
    \]
    where we have introduced the symmetric matrix $\bS := \bQ_\perp \bB^* \bB \bQ_\perp^*$. Hence, by \Cref{prop:momentbound}, we know that
    \[
        \Mom{\bOmega_j^* \bOmega_j}{\bS} \leq \sum_{\cI \subset [d]} \gamma_\cI(R) \| \tr{\cI}{\bS} \|^2_F,
    \]
    and since $\bS$, and therefore $\tr{\cI}{\bS}$, are positive semi-definite for any $\cI \subset [d]$, we note that $\| \tr{\cI}{\bS} \|^2_F \leq \Tr{\tr{\cI}{\bS}}^2 = \Tr{\bS}^2 = \| \bB \|^4_F$. Consequently
    \[
        \Var{\| \bB (\bOmega_j \bQ_\perp)^* \|^2_F} / \| \bB \|^4_F \leq \sum_{\cI \subset [d]} \gamma_\cI(R) - 1 \leq \left( 1 + \frac{p_\bbF}{R} \right)^d - 1,
    \]
    and $\Var{\| \bB (\bOmega \bQ_\perp)^* \|^2_F} \leq \frac{\| \bB \|^4_F}{P} \left ( \bigl ( 1 + \frac{p_\bbF}{R} \bigr )^d - 1 \right ) $.
    Now by the Chebyshev inequality, with probability at least $1-\delta$ we have
    \[
        \| \bB (\bOmega \bQ_\perp)^* \|^2_F \leq  \| \bB \|^2_F  + \sqrt{\frac{\bigl ( 1 + \frac{p_\bbF}{R} \bigr )^d - 1}{P \delta}}  \| \bB \|^2_F,
    \]
    and we conclude the argument by using the union bound to control numerator and denominator simultaneously with probability at least $1 - 2\delta$.
\end{proof}
In turn, this allows to bound the error from the randomized SVD implemented with a {\name} sketch.

\begin{proof}[Proof of \Cref{thm:rsvd_ttpr} (RSVD with {\name} OSI)]

    As with Theorem 2.2 in~\cite{camano2025faster}, the proof follows from the standard RSVD error bound: For a given matrix $\bA \in \bbF^{k \times N}$ and {\name} sketch $\bA\bOmega^*$, we compute the rank-$PR$ randomized factorization $\widehat{\bA} := \bQ (\bQ^*\bA)$ where $\bQ \bQ^*$ is the orthogonal projection onto the column space of $\bA\bOmega^*$. Write the partitioned SVD 
    \[
        \bA = \bU \begin{bmatrix}
            \bSigma_1 \\ & \bSigma_2 
        \end{bmatrix} \begin{bmatrix}
            \bV_1^* \\ \bV_2^*
        \end{bmatrix}
    \]
    with $\bSigma_1 \in \bbF^{r \times r}$ and $\bV_1 \in \bbF^{N \times r}$ containing the top $r$ singular values and right singular vectors of $\bA$, then (Theorem 9.1 in~\cite{HMT}) $\| \bA - \widehat{\bA} \|^2_F \leq \| \bSigma_2\|^2_F + \| \bSigma_2 (\bV_2^* \bOmega^*)^* \bigl ( ( \bV_1^* \bOmega^*) \bigr)^\dagger\|^2_F$. Finally, by \Cref{lem:rsvd_term}, we bound the excess term $\| \bSigma_2 (\bV_2^* \bOmega^*)^* \bigl ( ( \bV_1^* \bOmega^*) \bigr)^\dagger\|^2_F$ with probability at least $1-\delta$ and conclude.
\end{proof}

\subsection{Analysis of Randomize-then-Orthogonalize}\label{sec:randroundproof}

We conclude this section by providing a probabilistic guarantee for Randomize-then-Orthogonalize (\Cref{alg:tt_randrounding}), which we note is equivalent to the TT-SVD algorithm~\cite{oseledets} where each classical SVD step is instead realized using a randomized SVD implemented with an appropriate sketch. 

\begin{theorem}\label{thm:rand_round}
     \change{Let $\cA \in \mathbb{F}^{n_1 \times \dots \times n_d}$ be a tensor, $(r_k)_{k \in [d-1]} \in \mathbb{N}^{d-1}$ be the target ranks and $r = \max_{k \in [d-1]} r_k$, fix $\delta_\mathtt{rto}>0$,
    let $\mathbf{\Omega}_{\mathtt{\nameAbbrv}} = \mathcal{G}_2 \bowtie \cdots \bowtie \mathcal{G}_d \in \mathbb{F}^{PR \times n_2 \times \dots \times n_d}$ be a {\name} sketch with 
    \[
        R = p_{\mathbb{F}} d, \quad \text{and} \quad P \geq \frac{4}{(1-\alpha)^2} \left ( 4 (e-1) \; r +  e \log(4(d-1)r/\delta_\mathtt{rto}) \right ).
    \]
     Let $\widehat{\cA}$ be the output of Randomize-then-Orthogonalize (\Cref{alg:tt_randrounding}) applied to $\cA$ with target ranks $(r_k)_{k \in [d-1]}$.
    Then, with probability larger than $1-\delta_{\mathtt{rto}}$, we have
    \[
        \Vert \cA - \widehat{\cA} \Vert_F^2  \leq  (d-1) C_{\delta_{\mathtt{rto}}/(d-1)}  \Vert \cA - \cA_\textrm{best} \Vert^2_F,
    \]
    where $\cA_\textrm{best}$ is the best approximation of $\cA$ with TT-ranks at most $(r_1, \dots, r_{d-1})$ and $C_\delta$ is the constant from \Cref{thm:rsvd_ttpr}:
    \[
        C_{\delta} := 1 + \alpha^{-1}  \left ( 1 + \sqrt{\frac{\bigl ( 1 + p_\bbF /R \bigr )^d - 1}{P \delta / 2}}\right ) \quad =\cO \left(  \frac{1 + \sqrt{d/(PR\delta)}}{\alpha} \right).
    \]}
\end{theorem}

\begin{proof}
\change{Let $\delta = \delta_{\mathtt{rto}}/(d-1)$.
    By \Cref{thm:OSI}, for $k=1,\dots,d-1$, the sketch matrix $\mathbf{\Omega}_k = (\mathcal{G}_{k+1} \bowtie \cdots \bowtie \mathcal{G}_d)^{\leq 1}$ satisfies an $(\alpha,\delta/2,r_k)$-OSI of the subspace spanned by the dominant $r_k$ right singular vectors of $\mathcal{A}^{\leq k}$. Since these target subspaces are fixed (determined by $\mathcal{A}$ alone), the bounds below hold by a union bound, irrespective of any statistical dependence between the $\mathbf{\Omega}_k$. Let $\bP_k =  \bQ_k\bQ_k^*$ be the orthogonal projector onto the column space of $\mathcal{A}^{\leq k} \mathbf{\Omega}_k^*$. Then by the union bound and \Cref{thm:rsvd_ttpr}, with probability $1-(d-1)\delta$, we have
    \[
        \| \bP_k \mathcal{A}^{\leq k} - \mathcal{A}^{\leq k} \|^2_F \leq C_{\delta} \| \mathcal{A}^{\leq k} - \mathcal{A}^{\leq k}_{r_k} \|^2_F \qquad \text{for all }k = 1, \dots, d-1,
    \]
    where $\mathcal{A}^{\leq k}_{r_k}$ is the best rank-$r_k$ approximation  of $\mathcal{A}^{\leq k}$. In the rest of the proof, we fix $\cG_2, \dots, \cG_d$ such that these bounds hold.
    Now let $\widetilde{\mathcal{A}}_1 = \mathcal{A}$, and define recursively tensors $\widetilde{\mathcal{A}}_k \in \mathbb{F}^{n_1 \times \dots \times n_d}$ by
    \[
      (\widetilde{\mathcal{A}}_{k+1})^{\leq k} = \widetilde{\bP}_k (\widetilde{\mathcal{A}}_{k})^{\leq k}  \quad \text{for } k = 1 \dots d-1, 
    \]
    where $\widetilde{\bP}_{k} = \widetilde{\mathbf{Q}}_k\widetilde{\mathbf{Q}}_k^*$ is the orthogonal projector onto the column space of $\widetilde{\mathcal{A}_k}^{\leq k} \mathbf{\Omega}_k^*$. 

    By definition, we have that $\widehat{\mathcal{A}} = \widetilde{\mathcal{A}}_d$, and thus $\mathcal{A}-\widetilde{\mathcal{A}}_d = \sum_{k=1}^{d-1} \widetilde{\mathcal{A}}_{k} - \widetilde{\mathcal{A}}_{k+1}$. Now,
    \begin{align*}
        \| \mathcal{A}-\widetilde{\mathcal{A}}_d \|_F^2 &= \Big\| \sum_{k=1}^{d-1} \widetilde{\mathcal{A}}_{k} - \widetilde{\mathcal{A}}_{k+1} \Big\|_F^2
        = \Big\| \sum_{k=1}^{d-1} (\widetilde{\mathcal{A}}_{k} - \widetilde{\mathcal{A}}_{k+1})^{\leq 1} \Big\|_F^2 \\
        &= \| \widetilde{\mathcal{A}}_{1} - \widetilde{\mathcal{A}}_{2} \|_{F}^2 + \Big\| \sum_{k=2}^{d-1} \widetilde{\mathcal{A}}_{k} - \widetilde{\mathcal{A}}_{k+1} \Big\|_F^2,
    \end{align*}
    since the range of $(\widetilde{\mathcal{A}}_{k} - \widetilde{\mathcal{A}}_{k+1})^{\leq 1}$ for $k = 2, \dots, d-1$ is spanned by the range of $\widetilde{\bQ}_1$, and $\widetilde{\mathcal{A}}_{1}^{\leq 1} - \widetilde{\mathcal{A}}_{2}^{\leq 1} = (\mathbf{I}_{n_1} - \widetilde{\bP}_1) \widetilde{\mathcal{A}}_1^{\leq 1}$.
    By repeating the same argument, we have by iteration
    \[
        \| \mathcal{A}-\widetilde{\mathcal{A}}_d \|_F^2 = \sum_{k=1}^{d-1} \big\|  \widetilde{\mathcal{A}}_{k} - \widetilde{\mathcal{A}}_{k+1} \big\|_F^2.
    \]
    Since the Frobenius norm of a tensor coincides with that of any of its unfoldings, each summand satisfies $\| \widetilde{\mathcal{A}}_{k} - \widetilde{\mathcal{A}}_{k+1} \|_F = \| \widetilde{\mathcal{A}}_{k}^{\leq k} - \widetilde{\mathcal{A}}_{k+1}^{\leq k} \|_F$, which we now bound.
    Note for $k = 1,\dots,d$ the identity:
    \[
        (\widetilde{\mathcal{A}}_k)^{\leq k} = \widehat{\bP}_{k-1} \cA^{\leq k} \qquad \text{with} \qquad \widehat{\bP}_{k-1} := (\widetilde{\bP}_{k-1} \otimes \bI_{n_{k}}) \cdots (\widetilde{\bP}_1 \otimes \bI_{n_2 \dots n_{k}}),
    \] 
    an orthogonal projection constructed by recursive application of the randomized SVD at previous steps. In particular, we have $(\widetilde{\mathcal{A}}_k)^{\leq k} \bOmega_k^* = \widehat{\bP}_{k-1} \cA^{\leq k} \bOmega_k^*$. Since columns of $\bQ_k$ form an orthogonal basis for the range of $\cA^{\leq k} \bOmega_k^*$, the column spaces of $\widehat{\bP}_{k-1}\bQ_k$ and $\widehat{\bP}_{k-1}\cA^{\leq k} \bOmega_k^*$ are the same and $\widetilde{\bP}_k$ is an orthogonal projector onto the column space of $\widehat{\bP}_{k-1}\bQ_k$ as well, yielding the identity $(\bI - \widetilde{\bP}_k) \widehat{\bP}_{k-1} \bP_k = \mathbf{0}$. Now we write
    \begin{align*}
        \widetilde{\cA}_k^{\leq k} - \widetilde{\cA}_{k+1}^{\leq k}
         = (\bI - \widetilde{\bP}_k) \widetilde{\cA}_k^{\leq k}
         = (\bI - \widetilde{\bP}_k) \widehat{\bP}_{k-1} \cA^{\leq k}
         = (\bI - \widetilde{\bP}_k)  \widehat{\bP}_{k-1} (\bI - \bP_k) \cA^{\leq k}.
    \end{align*}
    Hence, since $(\bI - \widetilde{\bP}_k)$ and $\widehat{\bP}_{k-1}$ are orthogonal projections, and therefore contractions, we obtain
    \[
        \| \widetilde{\cA}_{k+1}^{\leq k} - \widetilde{\cA}_k^{\leq k} \|^2_F \leq \| (\bI - \bP_k) \cA^{\leq k} \|^2_F \leq C_\delta \| \mathcal{A}^{\leq k} - \mathcal{A}^{\leq k}_{r_k} \|_F^2,
    \]
where we recall $\mathcal{A}^{\leq k}_{r_k}$ is the best rank-$r_k$ approximation  of $\mathcal{A}^{\leq k}$.
Finally, we see that  $\| \mathcal{A}^{\leq k} - \mathcal{A}^{\leq k}_{r_k} \|_F^2 \leq \| \mathcal{A}-\mathcal{A}_{\mathrm{best}} \|_F^2$, so
\[
    \| \cA - \widetilde{\cA}_d \|^2_F \leq C_\delta (d-1) \| \cA - \cA_\mathrm{best} \|_F^2.
\]}
%
\end{proof}

\section{Conclusions/future perspectives}

In this work, we introduced the {\name} sketch, a novel tensor sketching framework that naturally unifies and extends the Khatri-Rao and Gaussian TT approaches. We established theoretical embedding guarantees exhibiting linear scaling with respect to the tensor order, and as a direct consequence, we derived quasi-optimal error bounds for the randomized TT Rounding scheme and empirically demonstrated its computational efficiency across a range of relevant applications.

Currently, the proposed sketch relies on a Gaussian base distribution. In regimes characterized by large block sizes $R$ or mode dimensions $n_k$, it is natural to explore computationally accelerated alternatives, such as the Fast Johnson-Lindenstrauss Transform \cite{ailon2009fast} or SparseStack \cite{kane2014sparser,camano2025faster}, which achieve comparable performance in the matrix setting. Future work will need to investigate how transitioning to these structured distributions impacts both the theoretical embedding guarantees and the practical efficiency of the tensor sketch.

Furthermore, extending this framework to broader tensor network architectures, particularly Tree Tensor Networks (TTNs) utilized in the Multi-Configurational Time-Dependent Hartree (MCTDH) method \cite{Meyer_Manthe_Cederbaum_1990}, represents a promising research direction. Adapting our randomized approach to these topologies could facilitate the solution of high-dimensional problems that fall beyond the traditional scope of DMRG-style algorithms for such networks.

Finally, we hope the {\name} sketch can be efficiently used for quantum chemistry applications, where it holds significant potential, since inherent physical symmetries induce block-sparse structures within the TT cores \cite{Bachmayr_Gotte_Pfeffer_2022}. The primary theoretical challenge in this domain lies in determining how to optimally align the {\name} sketch with the natural block sparsity of these tensors.


\bibliographystyle{plain}
\bibliography{references}

\appendix

\section{Gaussian Strong JLM Property} \label{app:sjlm_gaussian}

\begin{proof}[Proof of \Cref{lem:Gaussian_SJLM} (Strong JLM for Gaussian matrices)]
    Given a unit vector $\bx \in \bbF^n$ and a random matrix $\bG \in \bbF^{R \times n}$ with iid $\mathcal{N}_\bbF(0,\frac{1}{R})$ entries, the random variable $x = R \Vert \bG \bx \Vert^2$ follows a chi-squared distribution with parameter $R$. The classical Laurent-Massart bounds~\cite{laurent2000adaptive} yield
    \[
        \P{x - R > 2 \sqrt{Ru} + 2u} < e^{-u}, \qquad \P{R - x > 2 \sqrt{Ru}} < e^{-u}.
    \]
    To evaluate moments of $x-R$ for $t \geq 2$, we use the integrated tail formula and split the right and left tail components:
    \begin{align*}
        \E{ \vert x - R \vert^t} &= \int_0^\infty t s^{t-1} \P{ \vert x-R \vert > s} ds \\
        &= \int_0^\infty t s^{t-1} \P{  x-R  > s} ds + \int_0^\infty t s^{t-1} \P{  R-x  > s} ds = m_>^t + m_<^t.
    \end{align*}
    For the right, sub-gamma tail, we perform the substitution $s(u) = 2 \sqrt{Ru} + 2u$. Since $s(u)$ is strictly increasing, we obtain after integration by parts:
    \begin{align*}
        m_>^t &= \int_0^\infty t s'(u) s(u)^{t-1} \P{  x-R  > s(u)}  du  \leq \int_0^\infty s(u)^t e^{-u} du,
    \end{align*}
    which we further split using the Minkowski inequality and compute using the Gamma function:
    \begin{align*}
        m_>  & \leq \left ( \int_0^\infty \bigl ( 2 \sqrt{Ru} + 2u \bigr )^t e^{-u} du \right )^{1/t} \\
        & \leq \left ( \int_0^\infty \bigl ( 2 \sqrt{Ru}  )^t e^{-u} du \right )^{1/t} + \left ( \int_0^\infty \bigl (  2u \bigr )^t e^{-u} du \right )^{1/t} \\
        &= 2 \sqrt{R} \bigl ( \Gamma(t/2+1) \bigr )^{1/t} + 2 \bigl ( \Gamma(t+1) \bigr )^{1/t}.
    \end{align*}
    Using the upper bounds $\bigl ( \Gamma(t+1) \bigr )^{1/t} \leq t/\sqrt{2}$ and $\bigl ( \Gamma(t/2+1) \bigr )^{2/t} \leq t/2$ for $t \geq 2$, yields $m_> \leq \sqrt{2 R t} + \sqrt{2} t$. We similarly estimate the left sub-gaussian tail moment $m_< \leq \sqrt{2 R t}$. Finally, 
    \[
        \E{ \vert x - R \vert^t} \leq m_>^t + m_<^t \leq 2 \max (m_>, m_<)^t \leq 2 \left ( \sqrt{2} t + \sqrt{2 Rt}\right ) ^ t.
    \]
    Thus, for all $t \geq 2$, $\E{ \vert \Vert \bG \bx \Vert^2 - 1 \vert^t }^{1/t} \leq 2^{1/t} \sqrt{2} \bigl ( \sqrt{t/R} + t/R \bigr ) \leq 2 \bigl ( \sqrt{t/R} + t/R \bigr )$.
    Next, we check that provided $R \geq 8 \max(e/\varepsilon, 1)^2 \log(1/\delta)$, for $2 \leq t \leq \log(1/\delta)$:
    \[
        \E{ \vert \Vert \bG \bx \Vert^2 - 1 \vert^t }^{1/t} \leq \frac{\varepsilon}{\sqrt{2}e} \sqrt{\frac{t}{\log(1/\delta)}} \left ( 1 + \sqrt{t/R} \right )  \leq  \frac{1+1/\sqrt{8}}{\sqrt{2}}\frac{\varepsilon}{e} \sqrt{\frac{t}{\log(1/\delta)}},
    \]
    and we conclude by observing that $1+1/\sqrt{8} < \sqrt{2}$.
\end{proof}

\section{Proof of \Cref{lem:moment_tensor}}
\label{app:moment_tensor}
We prove the real and complex cases presented together in \Cref{lem:moment_tensor}, as \Cref{lem:moment_real_tensor,lem:moment_complex_tensor} respectively.

\subsection{Real case}
\begin{lemma}\label{lem:moment_real_matrix}
    Let $\bG \in \bbR^{R \times n}$ be a real Gaussian matrix with $\mathcal{N}_\bbR(0,1/R)$-iid entries and $\bA,\bB$ be fixed $n \times n$ real matrices (not necessarily symmetric). Then
    \begin{enumerate}
        \item $\E{\Tr{\bG\bA\bG^\top}\Tr{\bG\bB\bG^\top}} = \Tr{\bA} \Tr{\bB} + \frac{2}{R} \Tr{\hat{\bA} \hat{\bB}}$ \\
        where $\hat{\bC} := \frac{1}{2} (\bC+\bC^\top)$,
        \item $\E{\Tr{ \bG\bA\bG^\top (\bG \bB \bG^\top)^\top}} = \frac{1}{R} \left ( \Tr{\bA}\Tr{\bB} +  \Tr{\bA\bB} \right )+ \Tr{ \bA \bB^\top} $.
    \end{enumerate}
\end{lemma}
\begin{proof}
    For the first computation, let us write
    \begin{align*}        \E{\Tr{\bG\bA\bG^\top}\Tr{\bG\bB\bG^\top}} &= \Tr{(\bA \otimes \bB) \E{(\bG^\top \bG) \otimes (\bG^\top \bG)}}.
    \end{align*}
    Thanks to Iserles' theorem, the 4th order moment $\E{(\bG^\top \bG) \otimes (\bG^\top \bG)}$ may be decomposed as three terms, which we can order as
    \begin{align*}
        \E{(\bG^\top \bG) \otimes (\bG^\top \bG)} &= \E{\bG^\top \bG} \otimes \E{\bG^\top \bG)} + \E{\bG^\top \otimes \bG^\top} \E{\bG \otimes \bG} \\ & \qquad \qquad + \E{(\bG^\top \otimes \bI) \E{\bG \otimes \bG^\top} (\bI \otimes \bG)} \\
        &= \bI \otimes \bI + \frac{1}{R} \bE
    \end{align*}
    where we have used the identity $\E{\bG^\top \bG} = \bI_n$ and let $\bE$ have entries $E_{ii'}^{jj'} = \delta_{ii'} \delta_{jj'} + \delta_{ij'}\delta_{i'j}$ after a straightforward computation. Then, we get:
    \begin{align*}
    \E{\Tr{\bG\bA\bG^\top}\Tr{\bG\bB\bG^\top}} &= \Tr{\bA \otimes \bB} + \frac{1}{R} \Tr{(\bA \otimes \bB) \bE} \\&= \Tr{\bA}\Tr{\bB} + \frac{1}{R} \left (\Tr{\bA\bB^\top} + \Tr{\bA\bB }\right ).
    \end{align*}
    For the Hilbert-Schmidt scalar product computation, we proceed similarly, introducing iid copies $\bG_1$, $\bG_2$ of $\bG$:
    \begin{align*}        
        \hspace{.5in}&\hspace{-.5in} \E{ \Tr{ \bG\bA\bG^\top \bG\bB^\top\bG^\top} } \\
        &= \Tr{\bA\E{\bG^\top\bG} \bB^\top \E{\bG^\top\bG}} \\ 
        & \qquad + \Tr{\E{\bG\bA\bG^\top} \E{\bG\bB^\top\bG^\top}} + \E{ \Tr{ \bG_1\bA\bG_2^\top \bG_1\bB^\top\bG_2^\top} } \\
        &= \Tr{\bA \bB^\top} + \Tr{\frac{1}{R^2}  \Tr{\bA}\Tr{\bB}\bI } + \frac{1}{R} \Tr{\bA \E{\bG_2^\top \bG_2}\bB} \\
        &= \Tr{\bA \bB^\top} + \frac{1}{R}  \Tr{\bA}\Tr{\bB} + \frac{1}{R} \Tr{\bA\bB},
    \end{align*}
    where we have used the identities $\E{\bG\bA\bG^\top} = \frac{1}{R} \Tr{\bA} \bI$ and $\E{\bG\bB^\top\bG^\top} = \frac{1}{R} \Tr{\bB} \bI$ as well as 
    \[
        \E{\Tr{\bG_1 \bC \bG_1 \bD} \ \vert \ \bC, \bD} = \frac{1}{R} \Tr{\bC \bD^\top}, 
    \]
    where we let $\bC = \bA \bG_2^\top$, $\bD = \bB^\top \bG_2^\top$.
\end{proof}

\begin{lemma}\label{lem:moment_real_tensor}
    Let $k \in \{ 1,\dots,d\}$  and $\cI \subseteq \{ 1,\dots, k-1\}$. Then for $\bbF = \bbR$ we have:
    \begin{enumerate}
    \item    
    $
        \E{ \Vert \tr{\cI}{\tr{\{k\}}{\bS_k}}\Vert^2_F\ \vert \ \bS_{k+1}} \leq \Vert \tr{\cI \cup \{k,k+1\}}{\bS_{k+1}}\Vert^2_F + \frac{2}{R} \Vert \tr{\cI}{\bS_{k+1}}\Vert^2_F.
    $        
    \item  
    $
        \E{ \Vert \tr{\cI}{\bS_k}\Vert^2_F\ \vert \ \bS_{k+1}} \leq \frac{1}{R} \Vert \tr{\cI \cup \{k,k+1\}}{\bS_{k+1}}\Vert^2_F + \left ( 1 + \frac{1}{R} \right ) \Vert \tr{\cI}{\bS_{k+1}}\Vert^2_F.
    $
    \end{enumerate}
    In particular we let $\tr{\{d+1\}}{\bS_{d+1}} = \bS_{d+1} = \bS$.
\end{lemma}
\begin{proof}
Let us write $\bS_k = \sum_{i,j=1}^{n_1 \ldots n_{k-1}} \bE_{ij} \otimes \bS_{k,ij}$ and $\bS_{k+1} = \sum_{i,j=1}^{n_1 \ldots n_{k-1}} \bE_{ij} \otimes \bS_{k+1,ij}$ where $\bE_{ij} \in \bbR^{n_1 \ldots n_{k-1} \times n_1 \ldots n_{k-1}}$ is the matrix with its only nonzero entry equal to $1$ at position $(i,j)$, $\bS_{k, ij} \in \bbF^{R \times R}$ and $\bS_{k+1, ij} \in \bbF^{n_kR \times n_kR}$. Then we can factor the trace over the sub-index $k$:
\[
    \tr{\cI}{\tr{\{k\}}{\bS_k}} = \sum_{ij} \tr{\cI}{\bE_{ij}} \Tr{\bS_{k,ij}} = \sum_{ij} \Tr{\bG_k \bS_{k+1,ij} \bG_k^\top } \tr{\cI}{\bE_{ij}} 
\]
such that, using \Cref{lem:moment_real_matrix} to compute the expectation:
\begin{align*}
    \hspace{.4in}&\hspace{-.4in} \E{\Vert \tr{\cI}{\tr{\{k\}}{\bS_k}} \Vert^2} = \E{\Tr{\tr{\cI}{\tr{\{k\}}{\bS_k}}^2}} \\
    &= \sum_{iji'j'} \E{\Tr{\bG_k \bS_{k+1,ij} \bG_k^\top } \Tr{\bG_k \bS_{k+1,i'j'} \bG_k^\top }} \Tr{ \tr{\cI}{\bE_{ij}} \tr{\cI}{\bE_{i'j'}}} \\
    &= \sum_{iji'j'} \left ( \Tr{\bS_{k+1,ij}}\Tr{\bS_{k+1,i'j'}} + \frac{2}{R} \Tr{\hat{\bS}_{k+1,ij} \hat{\bS}_{k+1,i'j'}} \right ) \\
    & \qquad \qquad \qquad \qquad \times \Tr{ \tr{\cI}{\bE_{ij}} \tr{\cI}{\bE_{i'j'}}} \\
    &= \Tr{ \tr{\cI \cup \{k,k+1\}}{\sum_{ij} \bE_{ij} \otimes \bS_{k+1,ij} } \tr{\cI \cup \{k,k+1\}}{\sum_{i'j'} \bE_{i'j'} \otimes \bS_{k+1,i'j'}}} \\
    & \quad + \frac{2}{R}  \Tr{ \tr{\cI}{\sum_{ij} \bE_{ij} \otimes \hat{\bS}_{k+1,ij} } \tr{\cI}{\sum_{i'j'} \bE_{i'j'} \otimes \hat{\bS}_{k+1,i'j'}}}\\
    & = \Vert \tr{\cI \cup \{k,k+1\}}{\bS_{k+1}} \Vert^2_F + \frac{2}{R} \Vert \tr{\cI}{\hat{\bS}_{k+1}} \Vert^2_F,    \end{align*}
    where we have introduced the partially symmetrized matrix $\hat{\bS}_{k+1} := \sum_{ij} \bE_{ij} \otimes \hat{\bS}_{k+1,ij}$. Next, we estimate $\left \Vert \tr{\cI}{\hat{\bS}_{k+1}} \right \Vert_F^2 \leq \left \Vert \tr{\cI}{\bS_{k+1}} \right \Vert_F^2$ such that:
    \begin{align*}
    \E{\Vert \tr{\cI}{\tr{\{k\}}{\bS_k}} \Vert^2} &\leq \Vert \tr{\cI \cup \{k,k+1\}}{\bS_{k+1}} \Vert^2_F + \frac{2}{R} \Vert \tr{\cI}{\bS_{k+1}} \Vert^2_F,
\end{align*}
with equality when $\bS_{k+1} = \hat{\bS}_{k+1}$. Similarly, expanding $\Vert \tr{\cI}{\bS_k} \Vert^2$ using the Hilbert-Schmidt inner product:
\begin{align*}
    \hspace{.5in}&\hspace{-.5in} \E{\Vert \tr{\cI}{\bS_k} \Vert^2} \\
    & = \E{\Tr{\tr{\cI}{\bS_k}\tr{\cI}{\bS_k}^\top}} \\
    & = \sum_{iji'j'} \E{\Tr{\bG_k \bS_{k+1,ij} \bG_k^\top (\bG_k \bS_{k+1,i'j'} \bG_k^\top)^\top }} \Tr{ \tr{\cI}{\bE_{ij}} \tr{\cI}{\bE_{i'j'}^\top}} \\
    & = \sum_{iji'j'} \left [ \Tr{\bS_{k+1,ij} \bS_{k+1,i'j'}^\top} + \frac{1}{R}  \Tr{\bS_{k+1,ij} \bS_{k+1,i'j'}} \right. \\
    &  \qquad \qquad \qquad \qquad \left. + \frac{1}{R}\Tr{\bS_{k+1,ij}}\Tr{\bS_{k+1,i'j'}} \right ]  \Tr{ \tr{\cI}{\bE_{ij}} \tr{\cI}{\bE_{i'j'}^\top}} \\
    & = \sum_{iji'j'} \left [ \left ( 1 - \frac{1}{R} \right ) \Tr{\bS_{k+1,ij} \bS_{k+1,i'j'}^\top} + \frac{2}{R}  \Tr{ \hat{\bS}_{k+1,ij} \hat{\bS}_{k+1,i'j'}} \right. \\
    &   \qquad \qquad \qquad \qquad \left.+ \frac{1}{R}\Tr{\bS_{k+1,ij}}\Tr{\bS_{k+1,i'j'}} \right ]  \Tr{ \tr{\cI}{\bE_{ij}} \tr{\cI}{\bE_{i'j'}^\top}} \\
    & = \left ( 1 - \frac{1}{R} \right ) \left \Vert \tr{\cI}{\bS_{k+1}} \right \Vert^2_F + \frac{2}{R} \left \Vert \tr{\cI}{\hat{\bS}_{k+1}} \right \Vert_F^2 + \frac{1}{R}  \Vert \tr{\cI \cup \{k,k+1\}}{\bS_{k+1}} \Vert^2_F,
    \end{align*}
    and finally
    \begin{align*}
    \E{\Vert \tr{\cI}{\bS_k} \Vert^2} &\leq \left (1+\frac{1}{R} \right )\Vert \tr{\cI}{\bS_{k+1}}\Vert^2_F  + \frac{1}{R}  \Vert \tr{\cI \cup \{k,k+1\}}{\bS_{k+1}} \Vert^2_F,
\end{align*}
with equality when $\bS_{k+1} = \hat{\bS}_{k+1}$.
\end{proof}
\begin{remark}
    In the Khatri-Rao case ($R=1$), all upper bounds can be formulated in terms of the partially symmetrized matrices only~\cite{meyer2023hutchinson}, and then be equalities. This is not the case for $R>1$, but the bounds are sharp nevertheless.
\end{remark}

\subsection{Complex case}
\begin{lemma}\label{lem:moment_complex_matrix}
    Let $\bG \in \bbC^{R \times n}$ be a complex Gaussian matrix with iid $\mathcal{N}_\bbC(0,1/R)$-distributed entries and $\bA,\bB$ be fixed $n \times n$ complex matrices (not necessarily Hermitian). Then
    \begin{enumerate}
        \item $\E{\Tr{\bG\bA\bG^*}\overline{\Tr{\bG\bB\bG^*}}} = \Tr{\bA} \overline{\Tr{\bB}} + \frac{1}{R} \Tr{\bA \bB^*}$,
        \item $\E{\Tr{\bG\bA\bG^* (\bG\bB\bG^*)^*}} = \frac{1}{R} \Tr{\bA} \overline{\Tr{\bB}} + \Tr{\bA \bB^*}$.
    \end{enumerate}
\end{lemma}
\begin{proof}
    The proof is similar to the real case: for the first computation, let us write
    \begin{align*}        
        \E{\Tr{\bG\bA\bG^*} \overline{\Tr{\bG\bB\bG^*}}} &= \Tr{(\bA \otimes \bB^*) \E{(\bG^* \bG) \otimes (\bG^* \bG)}}.
    \end{align*}
    Thanks to Iserles' theorem, the 4th order moment $\E{(\bG^\top \bG) \otimes (\bG^\top \bG)}$ may be decomposed as three terms, which we can order as
    \begin{align*}
        \E{(\bG^* \bG) \otimes (\bG^* \bG)} &= \E{\bG^* \bG} \otimes \E{\bG^* \bG} + \E{\bG^* \otimes \bG^*} \E{\bG \otimes \bG} \\ & \qquad \qquad + \E{(\bG^* \otimes \bI) \E{\bG \otimes \bG^*} (\bI \otimes \bG)} \\
        &= \bI \otimes \bI + \frac{1}{R} \widetilde{\bE}
    \end{align*}
    where we have used the identities $\E{\bG^* \bG} = \bI_n$, $\E{\bG \otimes \bG} = \E{\bG^* \otimes \bG^*} = \boldsymbol{0}$ and let $\widetilde{\bE}$ have entries $E_{ii'}^{jj'} = \delta_{ij'}\delta_{i'j}$ after a straightforward computation. We then get:
    \begin{align*}
    \E{\Tr{\bG\bA\bG^*}\overline{\Tr{\bG\bB\bG^*}}} &= \Tr{\bA \otimes \bB^*} + \frac{1}{R} \Tr{(\bA \otimes \bB^*) \widetilde{\bE}} \\&= \Tr{\bA}\overline{\Tr{\bB}} + \frac{1}{R} \Tr{\bA\bB^*}.
    \end{align*}
    For the Hilbert-Schmidt inner product computation, we proceed similarly applying Iserles' theorem:
    \begin{align*}        
        \E{ \Tr{ \bG\bA\bG^* \bG\bB^*\bG^*} } &= \Tr{\bA\E{\bG^*\bG} \bB^* \E{\bG^*\bG}} + \Tr{\E{\bG\bA\bG^*} \E{\bG\bB^*\bG^*}} \\
        &= \Tr{\bA \bB^*} + \frac{1}{R}  \Tr{\bA} \overline{\Tr{\bB}},
    \end{align*}
    where we have used the identities $\E{\bG\bA\bG^*} = \frac{1}{R} \Tr{\bA} \bI$, $\E{\bG\bB^*\bG^*} = \frac{1}{R} \overline{\Tr{\bB}} \bI$.
\end{proof}

\begin{lemma}\label{lem:moment_complex_tensor}
    Let $k \in \{ 1,\dots,d\}$  and $\mathcal{I} \subseteq \{ 1,\dots, k-1\}$. Then for $\bbF = \bbC$ we have:
    \begin{enumerate}
    \item    
    $
        \E{ \Vert \tr{\cI\cup\{k\}}{\bS_k}\Vert^2_F\ \vert \ \bS_{k+1}} = \Vert \tr{\cI \cup \{k,k+1\}}{\bS_{k+1}}\Vert^2_F + \frac{1}{R} \Vert \tr{\cI}{\bS_{k+1}}\Vert^2_F.
    $        
    \item   
    $
        \E{ \Vert \tr{\cI}{\bS_k}\Vert^2_F\ \vert \ \bS_{k+1}} = \frac{1}{R} \Vert \tr{\cI \cup \{k,k+1\}}{\bS_{k+1}}\Vert^2_F + \Vert \tr{\cI}{\bS_{k+1}}\Vert^2_F.
    $
    \end{enumerate}
\end{lemma}
\begin{proof}
We follow the same steps as in the real case. Using \Cref{lem:moment_complex_matrix} to compute the expectation:
\begin{align*}
    & \E{\Vert \tr{\cI}{\tr{\{k\}}{\bS_k}} \Vert^2} \\ 
    & \qquad = \E{\Tr{\tr{\cI}{\tr{\{k\}}{\bS_k}}\tr{\cI}{\tr{\{k\}}{\bS_k}}^*}} \\
    &\qquad= \sum_{iji'j'} \E{\Tr{\bG_k \bS_{k+1,ij} \bG_k^* } \Tr{\bG_k \bS_{k+1,i'j'}^* \bG_k^* }} \Tr{ \tr{\cI}{\bE_{ij}} \tr{\cI}{\bE_{i'j'}^\top}} \\
    &\qquad= \sum_{iji'j'} \left ( \Tr{\bS_{k+1,ij}} \overline{\Tr{\bS_{k+1,i'j'}}} + \frac{1}{R} \Tr{\bS_{k+1,ij} \bS_{k+1,i'j'}^*} \right ) \\
    &\qquad \qquad \qquad \qquad \qquad \times \Tr{ \tr{\cI}{\bE_{ij}} \tr{\cI}{\bE_{i'j'}^\top}} \\
    &\qquad = \Vert \tr{\cI \cup \{k,k+1\}}{\bS_{k+1}} \Vert^2_F + \frac{1}{R} \Vert \tr{\cI}{\bS_{k+1}} \Vert^2_F.
\end{align*}
Similarly, expanding $\Vert \tr{\cI}{\bS_k} \Vert^2$ using the Hilbert-Schmidt inner product:
\begin{align*}
    & \E{\Vert \tr{\cI}{\bS_k} \Vert^2} \\ 
    &\qquad = \E{\Tr{\tr{\cI}{\bS_k}\tr{\cI}{\bS_k}^*}} \\
    &\qquad = \sum_{iji'j'} \E{\Tr{\bG_k \bS_{k+1,ij} \bG_k^* (\bG_k \bS_{k+1,i'j'} \bG_k^*)^* }} \Tr{ \tr{\cI}{\bE_{ij}} \tr{\cI}{\bE_{i'j'}^\top}} \\
    &\qquad = \sum_{iji'j'} \left ( \Tr{\bS_{k+1,ij} \bS_{k+1,i'j'}^*} + \frac{1}{R}\Tr{\bS_{k+1,ij}} \overline{\Tr{\bS_{k+1,i'j'}}} \right )\\
    &\qquad \qquad \qquad \qquad \qquad \times \Tr{ \tr{\cI}{\bE_{ij}} \tr{\cI}{\bE_{i'j'}^\top}} \\
    &\qquad = \left \Vert \tr{\cI}{\bS_{k+1}} \right \Vert^2_F  + \frac{1}{R}  \Vert \tr{\cI \cup \{k,k+1\}}{\bS_{k+1}} \Vert^2_F.
    \end{align*}
\end{proof}

\commentOut{
    
    The randomized TT algorithms are tested on different particular states 
    \begin{itemize}
    \item on $A + \varepsilon B + \varepsilon^2 C$ where $A,B \in \mathbb{R}^{n^d}$ are random normalized tensors of TT ranks $30$ and $C \in \mathbb{R}^{n^d}$ is a random tensor of TT rank $150$; 
    \item on random Slater determinants states coming from \[|\Psi\rangle = \det(\psi_1,\dots,\psi_N) = \sum_{ 1 \leq i_1 < \cdots < i_N \leq d}^{ } C_{i_1,i_2,\dots,i_N} \det(\phi_{i_1},\phi_{i_2},\dots,\phi_{i_N}) \] 
    \end{itemize}
    
    The second example has the property of having a controlled decay of the singular values with a smoother distribution than in the first example. 
    In the particle number representation, the tensor representing $|\Psi\rangle$ is of size $2^d \times 2^d$ and one can show that the singular values $(\sigma_j)_{1 \leq j\leq 2^{\min (k,d-k,N) }}$ of the matricization $\Uppsi_{\mu_1 \dots \mu_k}^{\mu_{k+1} \dots \mu_d}$ follow $\sigma_j^2 \sigma_{2^{\min (k,d-k,N)}-j}^2 = C$ for some constant $C$ that is explicit and for all $1 \leq j \leq  2^{\min (k,d-k,N)}$~\cite[Theorem 2.1]{dupuy2021inversion}.
    
    
    In all the tests below, $n=50$.
    
    \begin{figure}
    \centering
    \begin{subfigure}[b]{0.45\textwidth}
    \includegraphics[width=\textwidth]{img/perturbed_exacterror_ℓ=0_ε=0.01.png}
    \caption{Error of the TT-SVD}
    \label{subfig:exact_error_perturbed}
    \end{subfigure}
    \hfill
    \begin{subfigure}[b]{0.45\textwidth}
    \includegraphics[width=\textwidth]{img/perturbed_randroundingerror_ℓ=0_ε=0.01.png}
    \caption{Error of the randomized TT-SVD}
    \label{subfig:randrounding_error_perturbed}
    \end{subfigure}
    \begin{subfigure}[b]{0.45\textwidth}
    \includegraphics[width=\textwidth]{img/perturbed_randortherror_ℓ=0_ε=0.01.png}
    \caption{Error of the RandOrth TT-SVD}
    \label{subfig:randorth_error_perturbed}
    \end{subfigure}
    \hfill
    \begin{subfigure}[b]{0.45\textwidth}
    \includegraphics[width=\textwidth]{img/perturbed_stta_ℓ=0_ε=0.01.png}
    \caption{Error of the STTA}
    \label{subfig:STTA_error_perturbed}
    \end{subfigure}
    \caption{Error in log-scale of the different algorithms. No oversampling, $\varepsilon = 0.01$}
    \label{fig:perturbed}
    \end{figure}
    
    \begin{figure}
    \centering
    \begin{subfigure}[b]{0.45\textwidth}
    \includegraphics[width=\textwidth]{img/perturbed_exacterror_ℓ=0_ε=0.001.png}
    \caption{Error of the TT-SVD}
    \label{subfig:exact_error_perturbed}
    \end{subfigure}
    \hfill
    \begin{subfigure}[b]{0.45\textwidth}
    \includegraphics[width=\textwidth]{img/perturbed_randroundingerror_ℓ=0_ε=0.001.png}
    \caption{Error of the randomized TT-SVD}
    \label{subfig:randrounding_error_perturbed}
    \end{subfigure}
    \begin{subfigure}[b]{0.45\textwidth}
    \includegraphics[width=\textwidth]{img/perturbed_randortherror_ℓ=0_ε=0.001.png}
    \caption{Error of the RandOrth TT-SVD}
    \label{subfig:randorth_error_perturbed}
    \end{subfigure}
    \hfill
    \begin{subfigure}[b]{0.45\textwidth}
    \includegraphics[width=\textwidth]{img/perturbed_stta_ℓ=0_ε=0.001.png}
    \caption{Error of the STTA}
    \label{subfig:STTA_error_perturbed}
    \end{subfigure}
    \caption{Error in log-scale of the different algorithms. No oversampling, $\varepsilon = 0.001$}
    \label{fig:perturbed}
    \end{figure}
    
    One can see the exponential dependence in the STTA algorithm.
    
    
    \begin{figure}
      \centering
      \begin{subfigure}[b]{0.45\textwidth}
        \includegraphics[width=\textwidth]{img/slater_ℓ=10_rks_32_N=[6, 7, 8, 9, 10].png}
        \caption{Target TT rank = 32}
        \label{subfig:slater_rk32}
      \end{subfigure}
      \hfill
      \begin{subfigure}[b]{0.45\textwidth}
        \includegraphics[width=\textwidth]{img/slater_ℓ=10_rks_48_N=[6, 7, 8, 9, 10].png}
        \caption{Target TT rank = 48}
        \label{subfig:slater_rk48}
      \end{subfigure}
      \begin{subfigure}[b]{0.45\textwidth}
        \includegraphics[width=\textwidth]{img/slater_ℓ=10_rks_64_N=[6, 7, 8, 9, 10].png}
        \caption{Target TT rank = 64}
        \label{subfig:slater_rk64}
      \end{subfigure}
      \caption{Slater state with oversampling $\ell=10$}
      \label{fig:slater_oversampling10}
    \end{figure}
    
    \begin{figure}
      \centering
      \begin{subfigure}[b]{0.45\textwidth}
        \includegraphics[width=\textwidth]{img/slater_ℓ=0.5_rks_32_N=[6, 7, 8, 9, 10].png}
        \caption{Target TT rank = 32}
        \label{subfig:slater_rk32}
      \end{subfigure}
      \hfill
      \begin{subfigure}[b]{0.45\textwidth}
        \includegraphics[width=\textwidth]{img/slater_ℓ=0.5_rks_48_N=[6, 7, 8, 9, 10].png}
        \caption{Target TT rank = 48}
        \label{subfig:slater_rk48}
      \end{subfigure}
      \begin{subfigure}[b]{0.45\textwidth}
        \includegraphics[width=\textwidth]{img/slater_ℓ=0.5_rks_64_N=[6, 7, 8, 9, 10].png}
        \caption{Target TT rank = 64}
        \label{subfig:slater_rk64}
      \end{subfigure}
      \caption{Slater state with oversampling $\ell=r_\mathrm{max}/2 $}
      \label{fig:slater_oversampling10}
    \end{figure}

    
    \begin{lemma}[Johnson-Lindenstrauss lemma]
      \label{lem:JL-lemma}
      Let $\Omega \in \mathbb{R}^{r \times N} = \frac{1}{\sqrt{r}} \begin{bmatrix}
      \omega_{1}^\top \\ \vdots \\ \omega_{r}^\top
      \end{bmatrix}$ be a matrix with $\omega_{ij} \sim \mathcal{N}(0,1)$. 
      Let $A$ be a set of cardinality $k$. 
      Then $\Omega$ is an $\varepsilon$-embedding with probability $1-\delta$ if $r \geq 4\varepsilon^{-2} \log \big( k^2/\delta \big) $.
    \end{lemma}
    
    \begin{lemma}[Johnson-Lindenstrauss lemma for $I_N \otimes \Omega$]
      \label{lem:JL-lemma-IoxOmega}
      Let $\Omega \in \mathbb{R}^{r \times n} = \frac{1}{\sqrt{r}} \begin{bmatrix}
      \omega_{1}^\top \\ \vdots \\ \omega_{r}^\top
      \end{bmatrix}$ be a matrix with $\omega_{ij} \sim \mathcal{N}(0,1)$. 
      Let $x \in \mathbb{R}^{nN}$.
      Then $I_N \otimes \Omega$ is an $\varepsilon$-embedding of $x$ with probability $1-\delta$ if $r \geq 4 \varepsilon^{-2} \log \big( N^2/\delta \big) $.
    \end{lemma}
    
    \begin{proof}
      We have 
      \[
        I_N \otimes \Omega = \begin{bmatrix}
        \Omega & &  \\
        & \ddots & \\
        & & \Omega
        \end{bmatrix}
      \]
      and writing $x = \begin{bmatrix}
      x_1 \\ \vdots \\ x_N
      \end{bmatrix}$ where $(x_i) \in \mathbb{R}^n$, we have that 
      \[
        \big\vert\| (I_N \otimes \Omega)x \|^2 - \| x \|^2\big\vert = \Big\vert\sum_{ k=1 }^{ N} \Big( \frac{1}{r} \sum_{ j=1 }^{r } |\omega_j^\top x_k|^2 - \| x_k \|^2  \Big) \Big\vert \leq \varepsilon \| x \|^2,
      \]
      with probability $1-\delta$ by \Cref{lem:JL-lemma}.
    \end{proof}
    
    Using an $\varepsilon$-net argument, we can prove that $I_N \otimes \Omega$ is an $(\varepsilon, \delta, r)$ OSE if $r \geq Cr\varepsilon^{-2} \log\big(N^2/\delta\big)$, where $C$ is a numerical constant. 
    
    \begin{lemma}[Product of $\varepsilon$-embeddings]
      \label{lem:JL-product}
      Let $\varepsilon,\delta>0$. Let $x \in \mathbb{R}^N$.
      Let $(\Omega_1,\dots,\Omega_d)$ be matrices such that $\Omega_k \in \mathbb{R}^{r_{k-1} \times r_k}$ with $r_d = N$. 
      Suppose that for each $1 \leq k \leq d$, $\Omega_k$ is an $(\frac{\varepsilon}{d(1+\frac{d+1-k}{d}\varepsilon)})$-embedding with probability $1-\frac{\delta}{d}(1-\frac{d+1-k}{d}\delta)$. 
      Then $\Omega_1\cdots\Omega_d \in \mathbb{R}^{r_0 \times N}$ is an $\varepsilon$-embedding for $x$ with probability $1-\delta$.
    \end{lemma}
    
    \begin{proof}
      We prove this lemma by iteration.
      More precisely, we will show by iteration that 
      \[
        \mathbb{P}\bigg( \Big|\| \Omega_k \cdots \Omega_dx \| - \| x \|\Big| \geq \varepsilon_{d+1-k} \| x \|\bigg) \leq \delta_{d-k+1},
      \]
      with $\delta_{d+1-k} = \frac{d+1-k}{d}\delta$ and $\varepsilon_{d+1-k} = \frac{d+1-k}{d}\varepsilon$.
      For $k=d$, the statement is true.
      
      Suppose that it is true for $k=2,\dots,d$. We have 
      \begin{align}
        \mathbb{P}\bigg( \Big|\| \Omega_1 \cdots \Omega_dx \| - \| x \|\Big| \geq \varepsilon \| x \|\bigg) &=\mathbb{P}\Bigg( \bigg\lbrace \Big|\| \Omega_1 \cdots \Omega_dx \| - \| x \|\Big| \geq \varepsilon \| x \|\bigg\rbrace \bigcap \bigg\lbrace \Big|\| \Omega_2 \cdots \Omega_dx \| - \| x \|\Big| \geq \varepsilon_{d-1} \| x \|\bigg\rbrace \Bigg)  \\
        & \quad \quad +\mathbb{P}\Bigg( \bigg\lbrace \Big|\| \Omega_1 \cdots \Omega_dx \| - \| x \|\Big| \geq \varepsilon \| x \|\bigg\rbrace \bigcap \bigg\lbrace \Big|\| \Omega_2 \cdots \Omega_dx \| - \| x \|\Big| < \varepsilon_{d-1} \| x \|\bigg\rbrace \Bigg) \\
        &\leq \delta_{d-1} + \frac{ \mathbb{P} \Bigg(\Big| \| \Omega_1 \cdots \Omega_dx \| - \| x \|\Big| \geq \varepsilon \| x \|\ \bigg| \ \Big| \| \Omega_2 \cdots \Omega_dx \| - \| x \|\Big| < \varepsilon_{d-1} \| x \| \Bigg) }{\mathbb{P} \bigg( \Big|\| \Omega_2 \cdots \Omega_dx \| - \| x \|\Big| < \varepsilon_{d-1} \| x \| \bigg)} .
      \end{align}
      If $\Big\vert\| \Omega_2 \cdots \Omega_dx \| - \| x \|\Big| < \varepsilon_{d-1} \| x \|$, then we have 
      \[
        \| \Omega_2 \cdots \Omega_dx \| \leq (1+\varepsilon_{d-1}) \| x \|.
      \]
      Thus if $\Big|\| \Omega_1 \cdots \Omega_dx \| - \| x \|\Big| \geq \varepsilon \| x \|$, we have that  
      \[
        \Big|\| \Omega_1 \cdots \Omega_dx \| - \| \Omega_2 \cdots \Omega_dx \| \Big| \geq (\varepsilon-\varepsilon_{d-1}) \| x \| \geq \frac{\varepsilon-\varepsilon_{d-1}}{1+\varepsilon_{d-1}} \| \Omega_2 \cdots \Omega_dx \|= \frac{\varepsilon}{d(2-\frac{\delta}{d})}\| \Omega_2 \cdots \Omega_dx \|.
      \]
      Since $\mathbb{P} \bigg( \Big|\| \Omega_2 \cdots \Omega_dx \| - \| x \|\Big| < \varepsilon_{d-1} \| x \| \bigg) \geq 1-\delta_{d-1}$, we have 
      \[
        \mathbb{P}\bigg( \Big|\| \Omega_1 \cdots \Omega_dx \| - \| x \|\Big| \geq \varepsilon \| x \| \bigg) \leq \delta_{d-1} + \frac{\mathbb{P}\bigg( \Big|\| \Omega_1 \cdots \Omega_dx \| - \| \Omega_2 \cdots \Omega_dx \| \Big| \geq \frac{\varepsilon}{d(2-\frac{\delta}{d})} \| \Omega_2 \cdots \Omega_dx \|\bigg) }{1-\delta_{d-1}} .
      \]
      By assumption on $\Omega_1$, we have 
      \[
        \mathbb{P}\bigg( \Big|\| \Omega_1 \cdots \Omega_dx \| - \| \Omega_2 \cdots \Omega_dx \| \Big| \geq \frac{\varepsilon}{d(2-\frac{\delta}{d})} \| \Omega_2 \cdots \Omega_dx \|\bigg) \leq \frac{\delta}{d}(1-\frac{d-1}{d}\delta). 
      \]
      Thus 
      \[
        \mathbb{P}\bigg( \Big|\| \Omega_1 \cdots \Omega_dx \| - \| x \|\Big| \geq \varepsilon \| x \| \bigg) \leq \delta_{d-1} + \frac{\frac{\delta}{d}(1-\frac{d-1}{d}\delta)}{1-\delta_{d-1}} \leq \delta.
      \]
    \end{proof}
    
    \begin{proposition}
      Let $d,r\in \mathbb{N}$. Let $(n_1,\dots,n_d) \in \mathbb{N}^d$ and $n = \max n_k$.
      Let $\varepsilon,\delta>0$.
      Assume that $r_k \geq C \varepsilon^{-2}d^4 \log(dn/\delta)$ and $\Omega_k \in \mathbb{R}^{r_{k-1} \times n_k \times r_k}$ such that $\Omega_k^{\leq 1} \in \mathbb{R}^{r_{k-1} \times n_k r_k}$ has iid $\mathcal{N}(0,\frac{1}{r_{k-1}})$ entries. 
      Then $(\Omega_1 \bowtie \cdots \bowtie \Omega_d)^{\leq 1} \in \mathbb{R}^{r_0 \times n_1 \dots n_d}$ is an $(\varepsilon,\delta,r)$ OSE.
    \end{proposition}
    
    \begin{proof}
      We have that 
      \[
        (\Omega_1 \bowtie \cdots \bowtie \Omega_d)^{\leq 1} = \Omega_1^{\leq 1} \big( I_{n_1} \otimes \Omega_2^{\leq 1} \big) \cdots \big( I_{n_1\cdots n_{d-1}} \otimes \Omega_d^{\leq 1} \big).
      \]
      For all $1 \leq k \leq d-1$, $I_{n_1\cdots n_k} \otimes \Omega_k$ is an $\cO(\varepsilon/d)$-embedding with probability $1-\cO(\delta/d)$ by \Cref{lem:JL-lemma-IoxOmega}. 
      By \Cref{lem:JL-product}, we deduce that $(\Omega_1 \bowtie \cdots \bowtie \Omega_d)^{\leq 1}$ satisfies a Johnson-Lindenstrauss embedding property. 
      Using an $\varepsilon$-net argument, we deduce that $(\Omega_1 \bowtie \cdots \bowtie \Omega_d)^{\leq 1}$ satisfies an $(\varepsilon,\delta,r)$-OSE.
    \end{proof}

} 

\section{Proof of \Cref{lem:sjlm_stack}}\label{app:sjlm_ttpr}


\begin{proof}[Proof (Strong JLM for concatenation of matrices)]
Assume $\bOmega_i$ for $i \in [P]$ are iid realizations of a random matrix distribution satisfying the $(\varepsilon_0, \delta)$ Strong JLM property for some $\varepsilon_0, \delta >0$. Let $\bOmega = \frac{1}{\sqrt{P}} \begin{bmatrix}
    \bOmega_1 \\ \vdots \\ \bOmega_P
\end{bmatrix}$, and consider a unit vector $\bx$ and $t \leq \log(1/\delta)$. Then,
    \[
        \| \Vert \bOmega \bx \Vert^2_2 - 1 \|_{L^t} = \left\|  \frac{1}{P} \sum_{i=1}^P \|  \bOmega_i \bx\|_2^2 - 1\right\|_{L^t} = \frac{1}{P} \left\| \sum_{i=1}^P (\|  \bOmega_i \bx\|_2^2 - 1)\right\|_{L^t}.
    \]
    Let $X_i = \|  \bOmega_i \bx\|_2^2 - 1$ for $i \in [P]$, iid realizations of a mean zero random variable. From the Lata\l{}a inequality~\cite{latala1997estimation}, and more precisely Corollary 18 in~\cite{bujanovic2025subspace} to keep track of the constant:
\begin{align*}
    \| \sum_{i=1}^P X_i \|_{L^t} &\leq 4e^2 \sup \left \{ \frac{t}{s} \left ( \frac{P}{t} \right )^{1/s} \Vert X_1 \Vert_{L^s}\ | \ \max\left (2, t/P \right ) \leq s \leq t \right \} \\
    &\leq 4e^2 \sup \left \{ \frac{t}{s} \left ( \frac{P}{t} \right )^{1/s} \frac{\varepsilon_0}{e} \sqrt{\frac{s}{\log(1/\delta)}} \ | \ \max\left (2, t/P \right ) \leq s \leq t \right \} \\
    &= 4e^2 \frac{\varepsilon_0}{e} \sqrt{\frac{Pt}{\log(1/\delta)}} \sup \left \{ \frac{1}{\sqrt{s}} \left ( \frac{t}{P} \right )^{1/2 - 1/s} \ | \ \max\left (2, t/P \right ) \leq s \leq t \right \}.
\end{align*}
A study of the function $\phi(s) = \frac{1}{\sqrt{s}} \left ( \frac{t}{P} \right )^{1/2 - 1/s}$ reveals it is monotically decreasing on the interval $s > 2 \log(t/P)$ since $\frac{\mathrm{d} \log \phi}{\mathrm{d}s} = s^{-2} ( \log(t/P) - s/2)$, and because $2 \log(t/P) < t/P$, over the range $\max\left (2, t/P \right ) \leq s \leq t$, the maximum is always attained at $s = \max\left (2, t/P \right )$. We then compute
\[
    \phi(2) = \frac{1}{\sqrt{2}}, \qquad e^{-1/e} \leq \phi(t/P) = \left ( \frac{P}{t}\right )^{\frac{P}{t}} \leq 1 \text{ when } t/P \geq 2.
\]
Therefore,
\begin{align*}
    \| \Vert \bOmega \bx \Vert^2_2 - 1 \|_{L^t} 
    &\leq \frac{4e^2  \varepsilon_0}{\sqrt{P} e} \sqrt{\frac{t}{\log(1/\delta)}}.
\end{align*}
and picking $\varepsilon_0 = \frac{\sqrt{P}\varepsilon}{4e^2} $ ensures $\bOmega$ satisfies the Strong $(\varepsilon, \delta)$-JLM property.
\end{proof}

\end{document}